\documentclass{article}
\usepackage{amsmath}
\usepackage{amssymb}
\usepackage{amsthm}
\usepackage{graphicx}
\usepackage{cite}
\usepackage[arrow,curve,matrix,arc,2cell]{xy}
\UseAllTwocells
\usepackage[utf8]{inputenc}
\usepackage[unicode]{hyperref}
\DeclareFontFamily{U}{rsfs}{} \DeclareFontShape{U}{rsfs}{n}{it}{<->
rsfs10}{} \DeclareSymbolFont{mscr}{U}{rsfs}{n}{it}
\DeclareSymbolFontAlphabet{\scr}{mscr}
\def\mathscr{\scr}
\begin{document}
\def\e#1\e{\begin{equation}#1\end{equation}}
\def\ea#1\ea{\begin{align}#1\end{align}}
\def\eq#1{{\rm(\ref{#1})}}
\theoremstyle{plain}
\newtheorem{thm}{Theorem}[section]
\newtheorem{prop}[thm]{Proposition}
\newtheorem{lem}[thm]{Lemma}
\newtheorem{cor}[thm]{Corollary}
\newtheorem{quest}[thm]{Question}
\newtheorem{claim}[thm]{Claim}
\theoremstyle{definition}
\newtheorem{dfn}[thm]{Definition}
\newtheorem{ex}[thm]{Example}
\newtheorem{rem}[thm]{Remark}
\newtheorem{hyp}[thm]{Hypothesis}
\numberwithin{equation}{section}
\numberwithin{figure}{section}
\def\dim{\mathop{\rm dim}\nolimits}
\def\vdim{\mathop{\rm vdim}\nolimits}
\def\rank{\mathop{\rm rank}\nolimits}
\def\id{\mathop{\rm id}\nolimits}
\def\Im{\mathop{\rm Im}\nolimits}
\def\Re{\mathop{\rm Re}\nolimits}
\def\SU{\mathop{\rm SU}}
\def\Ker{\mathop{\rm Ker}}
\def\Spin{\mathop{\rm Spin}}
\def\Coker{\mathop{\rm Coker}}
\def\End{\mathop{\rm End}}
\def\an{{\rm an}}
\def\bo{{\rm bo}}
\def\dm{{\rm dm}}
\def\alg{{\rm alg}}
\def\dbo{{\rm dbo}}
\def\hom{{\rm hom}}
\def\ch{{\rm ch}}
\def\virt{{\rm virt}}
\def\Hom{\mathop{\rm Hom}\nolimits}
\def\coh{\mathop{\rm coh}\nolimits}
\def\Ho{\mathop{\rm Ho}\nolimits}
\def\Sch{\mathop{\bf Sch}\nolimits}
\def\St{\mathop{\bf St}\nolimits}
\def\dSt{\mathop{\bf dSt}\nolimits}
\def\dArt{\mathop{\bf dArt}\nolimits}
\def\dSch{\mathop{\bf dSch}\nolimits}
\def\dAff{\mathop{\bf dAff}\nolimits}
\def\cdga{\mathop{\bf cdga}\nolimits}
\def\dMan{{\mathop{\bf dMan}\nolimits}}
\def\dOrb{\mathop{\bf dOrb}\nolimits}
\def\dManb{\mathop{\bf dMan^b}\nolimits}
\def\dOrbb{\mathop{\bf dOrb^b}\nolimits}
\def\dManc{\mathop{\bf dMan^c}\nolimits}
\def\dOrbc{\mathop{\bf dOrb^c}\nolimits}
\def\DerManSp{\mathop{\bf DerMan_{Spi}}\nolimits}
\def\DerManBN{\mathop{\bf DerMan_{BoNo}}\nolimits}
\def\mKur{\mathop{\bf mKur}\nolimits}
\def\MKur{\mathop{\bf MKur}\nolimits}
\def\Kur{\mathop{\bf Kur}\nolimits}
\def\KurtrG{\mathop{\bf Kur_{trG}}\nolimits}
\def\MKurb{\mathop{\bf MKur^b}\nolimits}
\def\Kurb{\mathop{\bf Kur^b}\nolimits}
\def\MKurc{\mathop{\bf MKur^c}\nolimits}
\def\Kurc{\mathop{\bf Kur^c}\nolimits}
\def\cl{{\rm cl}}
\def\Spec{\mathop{\rm Spec}\nolimits}
\def\bSpec{\mathop{\bs{\rm Spec}}\nolimits}
\def\qcoh{{\mathop{\rm qcoh}}}
\def\bs{\boldsymbol}
\def\ge{\geqslant}
\def\le{\leqslant\nobreak}
\def\pr{\prec}
\def\bA{{\mathbin{\mathbb A}}}
\def\bL{{\mathbin{\mathbb L}}}
\def\bT{{\mathbin{\mathbb T}}}
\def\bR{{\bs R}}
\def\bS{{\bs S}}
\def\bU{{\bs U}}
\def\bV{{\bs V}}
\def\bW{{\bs W}}
\def\bX{{\bs X}}
\def\bY{{\bs Y}}
\def\bZ{{\bs Z}}
\def\cA{{\mathbin{\mathcal A}}}
\def\cB{{\mathbin{\mathcal B}}}
\def\cC{{\mathbin{\mathcal C}}}
\def\cD{{\mathbin{\mathcal D}}}
\def\cE{{\mathbin{\mathcal E}}}
\def\cI{{\mathbin{\mathcal I}}}
\def\cK{{\mathbin{\mathcal K}}}
\def\cL{{\mathbin{\mathcal L}}}
\def\cM{{\mathbin{\mathcal M}}}
\def\oM{{\mathbin{\smash{\,\,\ov{\!\!\mathcal M\!}\,}}}}
\def\O{{\mathbin{\mathcal O}}}
\def\fC{{\mathbin{\mathfrak C}}}
\def\fm{{\mathbin{\mathfrak m}}}
\def\C{{\mathbin{\mathbb C}}}
\def\N{{\mathbin{\mathbb N}}}
\def\Q{{\mathbin{\mathbb Q}}}
\def\R{{\mathbin{\mathbb R}}}
\def\Z{{\mathbin{\mathbb Z}}}
\def\uX{{{\underline{X\!}\,}}{}}
\def\al{\alpha}
\def\be{\beta}
\def\ga{\gamma}
\def\de{\delta}
\def\io{\iota}
\def\ep{\epsilon}
\def\la{\lambda}
\def\ka{\kappa}
\def\th{\theta}
\def\ze{\zeta}
\def\up{\upsilon}
\def\vp{\varphi}
\def\si{\sigma}
\def\om{\omega}
\def\De{\Delta}
\def\La{\Lambda}
\def\Th{\Theta}
\def\Om{\Omega}
\def\Ga{\Gamma}
\def\Si{\Sigma}
\def\Up{\Upsilon}
\def\pd{\partial}
\def\db{\bar\partial}
\def\ts{\textstyle}
\def\st{\scriptstyle}
\def\sst{\scriptscriptstyle}
\def\sm{\setminus}
\def\bu{\bullet}
\def\op{\oplus}
\def\ot{\otimes}
\def\boxt{\boxtimes}
\def\ov{\overline}
\def\ul{\underline}
\def\bigop{\bigoplus}
\def\iy{\infty}
\def\es{\emptyset}
\def\ra{\rightarrow}
\def\ab{\allowbreak}
\def\longra{\longrightarrow}
\def\hookra{\hookrightarrow}
\def\Ra{\Rightarrow}
\def\bs{\boldsymbol}
\def\t{\times}
\def\ci{\circ}
\def\ti{\tilde}
\def\d{{\rm d}}
\def\od{\odot}
\def\bigot{\bigotimes}
\def\bd{\boxdot}
\def\sh{\sharp}
\def\w{\wedge}
\def\ha{{\ts\frac{1}{2}}}
\def\md#1{\vert #1 \vert}
\def\bmd#1{\big\vert #1 \big\vert}
\def\dd{{\rm d}_{\rm dR}}
\title{Virtual fundamental classes for moduli spaces of sheaves on Calabi--Yau four-folds}
\author{Dennis Borisov and Dominic Joyce}
\date{}
\maketitle
\begin{abstract}
Let $(\bX,\om_\bX^*)$ be a separated, $-2$-shifted symplectic derived $\C$-scheme, in the sense of Pantev, To\"en, Vezzosi and Vaqui\'e \cite{PTVV}, of complex virtual dimension $\vdim_\C\bX=n\in\Z$, and $X_\an$ the underlying complex analytic topological space. We prove that $X_\an$ can be given the structure of a derived smooth manifold $\bX_\dm$, of real virtual dimension $\vdim_\R\bX_\dm=n$. This $\bX_\dm$ is not canonical, but is independent of choices up to bordisms fixing the underlying topological space $X_\an$. There is a 1-1 correspondence between orientations on $(\bX,\om_\bX^*)$ and orientations on $\bX_\dm$.

Because compact, oriented derived manifolds have virtual classes, this means that proper, oriented $-2$-shifted symplectic derived $\C$-schemes have virtual classes, in either homology or bordism. This is surprising, as conventional algebro-geometric virtual cycle methods fail in this case. Our virtual classes have half the expected dimension.

Now derived moduli schemes of coherent sheaves on a Calabi--Yau 4-fold are expected to be $-2$-shifted symplectic (this holds for stacks). We propose to use our virtual classes to define new Donaldson--Thomas style invariants `counting' (semi)stable coherent sheaves on Calabi--Yau 4-folds $Y$ over $\C$, which should be unchanged under deformations of~$Y$.
\end{abstract}

\setcounter{tocdepth}{2}
\tableofcontents

\section{Introduction}
\label{vf1}

This paper will relate two apparently rather different classes of `derived' geometric spaces. The first class is {\it derived\/ $\C$-schemes\/} $\bX$, in the Derived Algebraic Geometry of To\"en and Vezzosi \cite{Toen1,ToVe}, equipped with a $-2$-{\it shifted symplectic structure\/} $\om_\bX^*$ in the sense of Pantev, To\"en, Vaqui\'e and Vezzosi \cite{PTVV}. Such $(\bX,\om_\bX^*)$ are the expected structure on 4-Calabi--Yau derived moduli $\C$-schemes.

The second class is {\it derived smooth manifolds\/} $\bX_\dm$, in Derived Differential Geometry. There are several different models available: the {\it derived manifolds\/} of Spivak \cite{Spiv} and Borisov--No\"el \cite{Bori,BoNo} (which form $\iy$-categories $\DerManSp,\ab\DerManBN$), and the second author's {\it d-manifolds} \cite{Joyc2,Joyc3,Joyc4} (a strict 2-category $\dMan$), and {\it m-Kuranishi spaces\/} \cite[\S 4.7]{Joyc6} (a weak 2-category $\mKur$). 

As it is known that equivalence classes of objects in all these higher categories are in natural bijection, these four models are interchangeable for our purposes. But we use theorems proved for d-manifolds or (m-)Kuranishi spaces.

Here is a summary of our main results, taken from Theorems \ref{vf3thm3}, \ref{vf3thm4} and \ref{vf3thm8} and Propositions \ref{vf3prop3} and \ref{vf3prop4} below.

\begin{thm} Let\/ $(\bX,\om_\bX^*)$ be a $-2$-shifted symplectic derived\/ $\C$-scheme, in the sense of Pantev et al.\ {\rm\cite{PTVV},} with complex virtual dimension\/ $\vdim_\C\bX=n$ in $\Z,$ and write\/ $X_\an$ for the set of\/ $\C$-points of\/ $X=t_0(\bX),$ with the complex analytic topology. Suppose that\/ $X$ is separated, and\/ $X_\an$ is second countable. Then we can make the topological space $X_\an$ into  a derived manifold\/ $\bX_\dm$ of real virtual dimension\/ $\vdim_\R\bX_\dm=n,$ in the sense of any of\/~{\rm\cite{Bori,BoNo,Joyc2,Joyc3,Joyc4,Joyc6,Spiv}}.

There is a natural\/ $1$-$1$ correspondence between orientations on $(\bX,\om_\bX^*),$ in the sense of\/ {\rm\S\ref{vf24},} and orientations on\/ $\bX_\dm,$ in the sense of\/~{\rm\S\ref{vf26}.}

The (oriented) derived manifold\/ $\bX_\dm$ above depends on arbitrary choices made in its construction. However, $\bX_\dm$ is independent of choices up to (oriented) bordisms of derived manifolds which fix the underlying topological space.

All the above extends to (oriented) $-2$-shifted symplectic derived schemes $(\bs\pi:\bX\ra Z,\om_{\bX/Z}^*)$ over a base $Z$ which is a smooth affine $\C$-scheme of pure dimension, yielding an (oriented) derived manifold\/ $\bs\pi_\dm:\bX_\dm\ra Z_\an$ over the complex manifold\/ $Z_\an$ associated to $Z,$ regarded as an (oriented) real manifold.

\label{vf1thm1}
\end{thm}

In \S\ref{vf25} we give a short definition of {\it Kuranishi atlases\/} $\cK$ on a topological space $X$. These are families of `Kuranishi neighbourhoods' $(V,E,s,\psi)$ on $X$ and `coordinate changes' between them, based on work of Fukaya, Oh, Ohta and Ono \cite{FuOn,FOOO} in symplectic geometry. The hard work in proving Theorem \ref{vf1thm1} is using $(\bX,\om_\bX^*)$ to construct a Kuranishi atlas $\cK$ on $X_\an$. Then we use results from \cite{Bori,BoNo,Joyc2,Joyc3,Joyc4,Joyc6} to convert $(X_\an,\cK)$ into a derived manifold $\bX_\dm$. 

Readers of this papers do not need to understand derived manifolds, if they do not want to. They can just think in terms of Kuranishi atlases, as is common in symplectic geometry, without passing to derived manifolds.

We prove Theorem \ref{vf1thm1} using a `Darboux Theorem' for $k$-shifted symplectic derived schemes by Bussi, Brav and the second author \cite{BBJ}. This paper is related to the series \cite{BBBJ,BBDJS,BBJ,BJM,Joyc5}, mostly concerning the $-1$-shifted (3-Calabi--Yau) case. 

An important motivation for proving Theorem \ref{vf1thm1} is that {\it compact, oriented derived manifolds have virtual classes}, in both bordism and homology. As in \S\ref{vf36}--\S\ref{vf37}, from Theorem \ref{vf1thm1} we may deduce:

\begin{cor} Let\/ $(\bX,\om_\bX^*)$ be a proper, oriented\/ $-2$-shifted symplectic derived\/ $\C$-scheme, with\/ $\vdim_\C\bX=n$. Theorem\/ {\rm\ref{vf1thm1}} gives a compact, oriented derived manifold\/ $\bX_\dm$ with\/ $\vdim_\R\bX_\dm=n$. We may define a \begin{bfseries}d-bordism class\end{bfseries} $[\bX_\dm]_\dbo$ in the bordism group $B_n(*),$ and a \begin{bfseries}virtual class\end{bfseries} $[\bX_\dm]_\virt$ in the hom\-ology group $H_n(X_\an;\Z),$ depending only on $(\bX,\om_\bX^*)$ and its orientation.

Let\/ $\bX$ be a derived\/ $\C$-scheme, $Z$ a connected\/ $\C$-scheme, $\bs\pi:\bX\ra Z$ be proper, and\/ $[\om_{\bX/Z}]$ a family of oriented\/ $-2$-shifted symplectic structures on $\bX/Z,$ with\/ $\vdim_\C\bX/Z=n$. For each\/ $z\in Z_\an$ we have a proper, oriented\/ $-2$-shifted symplectic $\C$-scheme $(\bX^z,\om_{\bX^z}^*)$ with $\vdim\bX^z=n$. Then $[\bX_\dm^{z_1}]_\dbo\ab=\ab[\bX_\dm^{z_2}]_\dbo$ and\/ $\imath^{z_1}_*([\bX_\dm^{z_1}]_\virt)=\imath^{z_2}_*([\bX_\dm^{z_2}]_\virt)$ for all\/ $z_1,z_2\in Z_\an,$ with\/ $\imath^z_*([\bX_\dm^z]_\virt)\!\in\! H_n(X_\an;\Z)$ the pushforward under the inclusion\/~$\imath^z:X_\an^z\hookra X_\an$. 

\label{vf1cor1}
\end{cor}

So, {\it proper, oriented $-2$-shifted symplectic derived\/ $\C$-schemes $(\bX,\om_\bX^*)$ have virtual classes}. This is not obvious, in fact it is rather surprising. Firstly, if $(\bX,\om_\bX^*)$ is $-2$-shifted symplectic then $X=t_0(\bX)$ has a natural obstruction theory $\bL_\bX\vert_X\ra\bL_X$ in the sense of Behrend and Fantechi \cite{BeFa}, which is perfect in the interval $[-2,0]$. But the Behrend--Fantechi construction of virtual cycles \cite{BeFa} works only for obstruction theories perfect in $[-1,0]$, and does not apply here.

Secondly, our virtual cycle has real dimension $\vdim_\C\bX=\ha\vdim_\R\bX$, which is half what we might have expected. A heuristic explanation is that one should be able to make $\bX$ into a `derived $C^\iy$-scheme' $\bX^{C^\iy}$ (not a derived manifold), in some sense similar to Lurie \cite[\S 4.5]{Luri} or Spivak \cite{Spiv}, and $(\bX^{C^\iy},\Im\om_\bX^*)$ should be a `real $-2$-shifted symplectic derived $C^\iy$-scheme', with $\Im\om_\bX^*$ the imaginary part of $\om_\bX^*$. There should be a morphism $\bX^{C^\iy}\ra\bX_\dm$ which is a `Lagrangian fibration' of $(\bX^{C^\iy},\Im\om_\bX^*)$. So $\vdim_\R\bX_\dm=\ha\vdim_\R\bX^{C^\iy}=\ha\vdim_\R\bX$, as for Lagrangian fibrations $\pi:(S,\om)\ra B$ we have~$\dim B=\ha\dim S$.

The main application we intend for these results, motivated by Donaldson and Thomas \cite{DoTh} and explained in \S\ref{vf38}--\S\ref{vf39}, is to define new invariants `counting' (semi)stable coherent sheaves on Calabi--Yau 4-folds $Y$ over $\C$, which should be unchanged under deformations of $Y$. These are similar to Donaldson--Thomas invariants \cite{JoSo,KoSo,Thom} `counting' (semi)stable coherent sheaves of Calabi--Yau 3-folds, and could be called `holomorphic Donaldson invariants', as they are complex analogues of Donaldson invariants of 4-manifolds~\cite{DoKr}.

Pantev, To\"en, Vaqui\'e and Vezzosi \cite[\S 2.1]{PTVV} show that any derived moduli stack $\bs\cM$ of coherent sheaves (or complexes of coherent sheaves) on a Calabi--Yau $m$-fold has a $(2-m)$-shifted symplectic structure $\om_{\sst\bs\cM}^*$, so in particular 4-Calabi--Yau moduli stacks are $-2$-shifted symplectic. Given an analogue of this for derived moduli schemes, and a way to define orientations upon them, Corollary \ref{vf1cor1} would give virtual classes for moduli schemes of (semi)stable coherent sheaves on Calabi--Yau 4-folds, and so enable us to define invariants.

It is well known that there is a great deal of interesting and special geometry, related to String Theory, concerning Calabi--Yau 3-folds and 3-Calabi--Yau categories --- Mirror Symmetry, Donaldson--Thomas theory, and so on. One message of this paper is that there should also be special geometry concerning Calabi--Yau 4-folds and 4-Calabi--Yau categories, which is not yet understood.

During the writing of this paper, Cao and Leung \cite{Cao,CaLe1,CaLe2} also proposed a theory of invariants counting coherent sheaves on Calabi--Yau 4-folds, based on gauge theory rather than derived geometry. We discuss their work in~\S\ref{vf39}.

Section \ref{vf2} provides background material on derived schemes, shifted symplectic structures upon them, Kuranishi atlases, and derived manifolds. The heart of the paper is \S\ref{vf3}, with the definitions, main results, shorter proofs, and discussion. Longer proofs of results in \S\ref{vf3} are deferred to sections \ref{vf4}--\ref{vf6}.
\medskip

\noindent{\bf Acknowledgements.} We would like to thank Yalong Cao, Conan Leung, Bertrand To\"en, Gabriele Vezzosi, and a referee for helpful conversations. This research was supported by EPSRC Programme Grant~EP/I033343/1.

\section{Background material}
\label{vf2}

We begin with some background material and notation needed later. Some references are To\"en and Vezzosi \cite{Toen1,ToVe} for \S\ref{vf21}--\S\ref{vf22}, Pantev, To\"en, Vezzosi and Vaqui\'e \cite{PTVV} and Brav, Bussi and Joyce \cite{BBJ} for \S\ref{vf23}, and Spivak \cite{Spiv}, Borisov and No\"el \cite{Bori,BoNo} and Joyce \cite{Joyc2,Joyc3,Joyc4,Joyc6} for~\S\ref{vf26}.

\subsection{Commutative differential graded algebras}
\label{vf21}

\begin{dfn} Write $\cdga_\C$ for the category of commutative differential graded $\C$-algebras in nonpositive degrees, and $\cdga_\C^{\bf op}$ for its opposite category. In fact $\cdga_\C$ has the additional structure of a model category (a kind of $\iy$-category), but we only use this in the proof of Theorem \ref{vf3thm1} in \S\ref{vf4}. In the rest of the paper we treat $\cdga_\C,\cdga_\C^{\bf op}$ just as ordinary categories.

Objects of $\cdga_\C$ are of the form $\cdots \ra A^{-2}\,{\buildrel\d\over\longra}\,A^{-1}\,{\buildrel\d\over\longra}\,A^0$. Here $A^k$ for $k=0,-1,-2,\ldots$ is the $\C$-vector space of degree $k$ elements of $A$, and we have a $\C$-bilinear, associative, supercommutative multiplication $\cdot:A^k\t A^l\ra A^{k+l}$ for $k,l\le 0$, an identity $1\in A^0$, and differentials $\d:A^k\ra A^{k+1}$ for $k<0$ satisfying
\begin{equation*}
\d(a\cdot b)=(\d a)\cdot b+(-1)^ka\cdot(\d b)
\end{equation*}
for all $a\in A^k$, $b\in A^l$. We write such objects as $A^\bu$ or $(A^*,\d)$. 

Here and throughout we will use the superscript `$\,{}^*\,$' to denote {\it graded\/} objects (e.g.\ graded algebras or vector spaces), where $*$ stands for an index in $\Z$, so that $A^*$ means $(A^k$, $k\in\Z)$. We will use the superscript `$\,{}^\bu\,$' to denote {\it differential graded\/} objects (e.g.\ differential graded algebras or complexes), so that $A^\bu$ means $(A^*,\d)$, the graded object $A^*$ together with the differential $\d$. 

{\it Morphisms\/} $\al:A^\bu\ra B^\bu$ in $\cdga_\C$ are $\C$-linear maps $\al^k:A^k\ra B^k$ for all $k\le 0$ commuting with all the structures on~$A^\bu,B^\bu$.

A morphism $\al:A^\bu\ra B^\bu$ is a {\it quasi-isomorphism\/} if $H^k(\al):H^k(A^\bu)\ra H^k(B^\bu)$ is an isomorphism on cohomology groups for all $k\le 0$. A fundamental principle of derived algebraic geometry is that $\cdga_\C$ is not really the right category to work in, but instead one wants to define a new category (or better, $\iy$-category) by inverting (localizing) quasi-isomorphisms in $\cdga_\C$.

We will call $A^\bu\in\cdga_\C$ of {\it standard form\/} if $A^0$ is a smooth finitely generated $\C$-algebra of pure dimension, and the graded $\C$-algebra $A^*$ is freely generated over $A^0$ by finitely many generators in each degree $i=-1,-2,\ldots.$ Here we require $A^0$ to be smooth {\it of pure dimension\/} so that $(\Spec A^0)_\an$ is a complex manifold, rather than a disjoint union of complex manifolds of different dimensions. This is not crucial, but will be convenient in~\S\ref{vf3}.
\label{vf2def1}
\end{dfn}

\begin{rem} Brav, Bussi and Joyce \cite[Def.~2.9]{BBJ} work with a stronger notion of standard form cdgas than us, as they require $A^*$ to be freely generated over $A^0$ by finitely many generators, all in negative degrees. In contrast, we allow infinitely many generators, but only finitely many in each degree $i=-1,-2,\ldots.$

The important thing for us is that since standard form cdgas in the sense of \cite{BBJ} are also standard form in the (slightly weaker) sense of this paper, we can apply some of their results \cite[Th.s 4.1, 4.2, 5.18]{BBJ} on the existence and properties of nice standard form cdga local models for derived schemes.
\label{vf2rem1}
\end{rem}

\begin{dfn} Let $A^\bu\in\cdga_\C$, and write $D(\mathop{\rm mod} A)$ for the derived category of dg-modules over $A^\bu$. Define a {\it derivation of degree $k$} from $A^\bu$ to an $A^\bu$-module $M^\bu$ to be a $\C$-linear map $\de:A^\bu\ra M^\bu$ that is homogeneous of degree $k$ with
\begin{equation*}
\de(fg)=\de(f)g+(-1)^{k|f|}f\de(g).
\end{equation*}
Just as for ordinary commutative algebras, there is a universal
derivation into an $A^\bu$-module of {\it K\"ahler differentials\/}
$\Om^1_{A^\bu}$, which can be constructed as $I/I^2$ for $I=\Ker(m:A^\bu\ot A^\bu
\ra A^\bu)$. The universal derivation $\de:A^\bu\ra \Om^1_{A^\bu}$ is $\de(a)=a\ot 1-1\ot a \in I/I^2$. One checks that $\de$ is a universal degree $0$ derivation, so that ${}\ci
\de:\Hom^\bu_{A^\bu}(\Om^1_{A^\bu},M^\bu) \ra \mathop{\rm Der}^\bu(A,M^\bu)$ is an
isomorphism of dg-modules. 

Note that $\Om^1_{A^\bu}=((\Om^1_{A^\bu})^*,\d)$ is canonical up to strict isomorphism, not just up to quasi-isomorphism of complexes, or up to equivalence in $D(\mathop{\rm mod} A)$. Also, the underlying graded vector space $(\Om^1_{A^\bu})^*$, as a module over the graded algebra $A^*$, depends only on $A^*$ and not on the differential $\d$ in $A^\bu=(A^*,\d)$.

Similarly, given a morphism of cdgas $\Phi:A^\bu\ra B^\bu$, we can define the {\it relative K\"ahler differentials\/} $\Om^1_{B^\bu/A^\bu}$.

The {\it cotangent complex\/} $\bL_{A^\bu}$ of $A^\bu$ is related to the K\"ahler differentials $\Om^1_{A^\bu}$, but is not quite the same. If $\Phi:A^\bu\ra B^\bu$ is a quasi-isomorphism of cdgas over $\C$, then $\Phi_*:\Om^1_{A^\bu}\ot_{A^\bu}B^\bu\ra \Om^1_{B^\bu}$ may not be a quasi-isomorphism of $B^\bu$-modules. So K\"ahler differentials are not well-behaved under localizing quasi-isomorphisms of cdgas, which is bad for doing derived algebraic geometry. 

The cotangent complex $\bL_{A^\bu}$ is a substitute for $\Om^1_{A^\bu}$ which is well-behaved under localizing quasi-isomorphisms. It is an object in $D(\mathop{\rm mod} A)$, canonical up to equivalence. We can define it by replacing $A^\bu$ by a quasi-isomorphic, cofibrant (in the sense of model categories) cdga $B^\bu$, and then setting $\bL_{A^\bu}=(\Om^1_{B^\bu})\ot_{B^\bu}A^\bu$. We will be interested in the $p^{\rm th}$ exterior power $\La^p\bL_{A^\bu}$, and the dual $(\bL_{A^\bu})^\vee$, which is called the {\it tangent complex}, and written $\bT_{A^\bu}=(\bL_{A^\bu})^\vee$. 

There is a {\it de Rham differential\/} $\dd:\La^p\bL_{A^\bu}\ra\La^{p+1}\bL_{A^\bu}$, a morphism of complexes, with $\dd^2=0:\La^p\bL_{A^\bu}\ra\La^{p+2}\bL_{A^\bu}$. Note that each $\La^p\bL_{A^\bu}$ is also a complex with its own internal differential $\d:(\La^p\bL_{A^\bu})^k\ra(\La^p\bL_{A^\bu})^{k+1}$, and $\dd$ being a morphism of complexes means that $\d\ci\dd=\dd\ci\d$.

Similarly, given a morphism of cdgas $\Phi:A^\bu\ra B^\bu$, we can define the {\it relative cotangent complex\/} $\bL_{B^\bu/A^\bu}$.

As in \cite[\S 2.3]{BBJ}, an important property of our standard form cdgas $A^\bu$ in Definition \ref{vf2def1} is that they are sufficiently cofibrant that the K\"ahler differentials $\Om^1_{A^\bu}$ provide a model for the cotangent complex $\bL_{A^\bu}$, so we can take $\Om^1_{A^\bu}=\bL_{A^\bu}$, without having to replace $A^\bu$ by an unknown cdga $B^\bu$. Thus standard form cdgas are convenient for doing explicit computations with cotangent complexes.

A morphism $\Phi:A^\bu\ra B^\bu$ of cdgas will be called {\it quasi-free\/} if $\Phi^0:A^0\ra B^0$ is a smooth morphism of $\C$-algebras of pure relative dimension, and as a graded $A^*\ot_{A^0}B^0$-algebra $B^*$ is free and finitely generated in each degree. Here if $A^\bu$ is of standard form and $\Phi$ is quasi-free then $B^\bu$ is of standard form, and a cdga $A^\bu$ is of standard form if and only if the unique morphism $\C\ra A^\bu$ is quasi-free. We will only consider quasi-free morphisms when $A^\bu,B^\bu$ are of standard form.

If $\Phi:A^\bu\ra B^\bu$ is a quasi-free morphism then the relative K\"ahler differentials $\Om^1_{B^\bu/A^\bu}$ are a model for the relative cotangent complex $\bL_{B^\bu/A^\bu}$, so we can take $\Om^1_{B^\bu/A^\bu}=\bL_{B^\bu/A^\bu}$. Thus quasi-free morphisms are a convenient class of morphisms for doing explicit computations with cotangent complexes.
\label{vf2def2}
\end{dfn}

\subsection{Derived algebraic geometry and derived schemes}
\label{vf22}

\begin{dfn} Write $\dSt_\C$ for the $\iy$-category of {\it derived\/ $\C$-stacks} (or $D^-$-{\it stacks\/}) defined by To\"en and Vezzosi \cite[Def.~2.2.2.14]{ToVe}, \cite[Def.~4.2]{Toen1}. Objects $\bX$ in $\dSt_\C$ are $\iy$-functors
\begin{equation*}
\smash{\bX:\{\text{simplicial commutative $\C$-algebras}\}\longra
\{\text{simplicial sets}\}}
\end{equation*}
satisfying sheaf-type conditions. There is a {\it spectrum functor\/}
\begin{equation*}
\smash{\bSpec:\cdga^{\bf op}_\C \longra\dSt_\C.}
\end{equation*}
A derived $\C$-stack $\bX$ is called an {\it affine derived\/ $\C$-scheme\/} if $\bX$ is equivalent in $\dSt_\C$ to $\bSpec A^\bu$ for some cdga $A^\bu$ over $\C$. As in \cite[\S 4.2]{Toen1}, a derived $\C$-stack $\bX$ is called a {\it derived\/ $\C$-scheme\/} if it may be covered by Zariski open $\bY\subseteq\bX$ with $\bY$ an affine derived $\C$-scheme. Write $\dSch_\C$ for the full $\iy$-subcategory of derived $\C$-schemes in $\dSt_\C$, and $\dSch_\C^{\bf aff}\subset\dSch_\C$ for the full $\iy$-subcategory of affine derived $\C$-schemes. See also To\"en \cite{Toen2} for a different but equivalent way to define derived $\C$-schemes, as an $\iy$-category of derived ringed spaces.

We shall assume throughout this paper that all derived $\C$-schemes $\bX$ are {\it locally finitely presented\/} in the sense of To\"en and Vezzosi \cite[Def.~1.3.6.4]{ToVe}. Note that this is a strong condition, for instance it implies that the cotangent complex $\bL_\bX$ is perfect \cite[Prop.~2.2.2.4]{ToVe}. A locally finitely presented classical $\C$-scheme $X$ need not be locally finitely presented as a derived $\C$-scheme. A local normal form for locally finitely presented derived $\C$-schemes is given in~\cite[Th.~4.1]{BBJ}.

There is a {\it classical truncation functor\/} $t_0:\dSch_\C\ra\Sch_\C$ taking a derived $\C$-scheme $\bX$ to the underlying classical $\C$-scheme $X=t_0(\bX)$. On affine derived schemes $\dSch_\C^{\bf aff}$ this maps~$t_0:\bSpec A^\bu\ra\Spec H^0(A^\bu)=\Spec (A^0/\d(A^{-1}))$.

To\"en and Vezzosi show that a derived $\C$-scheme $\bX$ has a {\it cotangent complex\/} $\bL_\bX$ \cite[\S 1.4]{ToVe}, \cite[\S 4.2.4--\S 4.2.5]{Toen1} in a stable $\iy$-category $L_\qcoh(\bX)$ defined in \cite[\S 3.1.7, \S 4.2.4]{Toen1}. We will be interested in the $p^{\rm th}$ exterior power $\La^p\bL_\bX$, and the dual $(\bL_\bX)^\vee$, which is called the {\it tangent complex\/} $\bT_\bX$. There is a {\it de Rham differential\/}~$\dd:\La^p\bL_\bX\ra\La^{p+1}\bL_\bX$.

Restricted to the classical scheme $X=t_0(\bX)$, the cotangent complex $\bL_\bX\vert_X$ may Zariski locally be modelled as a finite complex of vector bundles $[F^{-m}\ra F^{1-m}\ra\cdots\ra F^0]$ on $X$ in degrees $[-m,0]$ for some $m\ge 0$. The ({\it complex\/}) {\it virtual dimension\/} $\vdim_\C\bX$ is $\vdim_\C\bX=\sum_{i=0}^m(-1)^i\rank F^{-i}$. It is a locally constant function $\vdim_\C\bX:X\ra\Z$, so is constant on each connected component of $X$. We say that $\bX$ {\it has\/} ({\it complex\/}) {\it virtual dimension\/} $n\in\Z$ if~$\vdim_\C\bX=n$.

When $\bX=X$ is a classical scheme, the homotopy category of $L_\qcoh(\bX)$ is the triangulated category $D_{\qcoh}(X)$ of complexes of quasicoherent sheaves. These $\bL_\bX,\bT_\bX$ have the usual properties of (co)tangent complexes. For instance, if $\bs f:\bX\ra\bY$ is a morphism in $\dSch_\C$ there is a distinguished triangle
\begin{equation*}
\smash{\xymatrix@C=30pt{ \bs f^*(\bL_\bY) \ar[r]^(0.55){\bL_{\bs f}} &
\bL_\bX \ar[r] & \bL_{\bX/\bY} \ar[r] & \bs f^*(\bL_\bY)[1], }}
\end{equation*}
where $\bL_{\bX/\bY}$ is the {\it relative cotangent complex\/} of~$\bs f$. 

Now suppose $A^\bu$ is a cdga over $\C$, and $\bX$ a derived $\C$-scheme with $\bX\simeq\bSpec A^\bu$ in $\dSch_\C$. Then we have an equivalence of triangulated categories $L_\qcoh(\bX)\simeq \ab D(\mathop{\rm mod}A)$, which identifies cotangent complexes $\bL_\bX\simeq\bL_{A^\bu}$. If also $A^\bu$ is of standard form then $\bL_{A^\bu}\simeq\Om^1_{A^\bu}$, so~$\bL_\bX\simeq\Om^1_{A^\bu}$.
\label{vf2def3}
\end{dfn}

Bussi, Brav and Joyce \cite[Th.~4.1]{BBJ} prove:

\begin{thm} Suppose $\bX$ is a derived\/ $\C$-scheme (as always, assumed locally finitely presented), and\/ $x\in\bX$. Then there exists a standard form cdga $A^\bu$ over $\C$ and a Zariski open inclusion $\bs\al:\bSpec A^\bu\hookra\bX$ with\/ $x\in\Im\bs\al$.
\label{vf2thm1}
\end{thm}

See Remark \ref{vf2rem1} on the difference in definitions of `standard form'. Bussi et al.\ also explain \cite[Th.~4.2]{BBJ} how to compare two such standard form charts $\bSpec A^\bu\hookra\bX$, $\bSpec B^\bu\hookra\bX$ on their overlap in $\bX$, using a third chart. We will need the following conditions on derived $\C$-schemes and their morphisms.

\begin{dfn} A derived $\C$-scheme $\bX$ is called {\it separated}, or {\it proper}, or {\it quasicompact}, if the classical $\C$-scheme $X=t_0(\bX)$ is separated, or proper, or quasicompact, respectively, in the classical sense, as in Hartshorne \cite[pp.~80, 96, 100]{Hart}. Proper implies sperated. A morphism of derived schemes $\bs f:\bX\ra\bY$ is {\it proper\/} if $t_0(\bs f):t_0(\bX)\ra t_0(\bY)$ is proper in the classical sense~\cite[p.~100]{Hart}.

We will need the following nontrivial fact about the relation between classical and derived $\C$-schemes. As in To\"en \cite[\S 2.2, p.~186]{Toen2}, a derived $\C$-scheme $\bX$ is affine if and only if the classical $\C$-scheme $X=t_0(\bX)$ is affine.

Recall that a morphism $\al:X\ra Y$ in $\Sch_\C$ (or $\bs\al:\bX\ra\bY$ in $\dSch_\C$) is {\it affine\/} if whenever $\be:U\ra Y$  is a Zariski open inclusion with $U$ affine (or $\bs\be:\bU\ra\bY$ is Zariski open with $\bU$ affine), the fibre product $X\t_{\al,Y,\be}U$ in $\Sch_\C$ (or homotopy fibre product $\bX\t_{\bs\al,\bY,\bs\be}^h\bU$ in $\dSch_\C$) is also affine. Since $\bX$ is affine if and only if $X=t_0(\bX)$ is affine, we see that a morphism $\bs\al:\bX\ra\bY$ in $\dSch_\C$ is affine if and only if $t_0(\bs\al):t_0(\bX)\ra t_0(\bY)$ is affine.

Now let $\bX$ be a separated derived $\C$-scheme. Then $X=t_0(\bX)$ is a separated classical $\C$-scheme, so \cite[p.~96]{Hart} the diagonal morphism $\De_X:X\ra X\t X$ is a closed immersion. But closed immersions are affine, and $\De_X=t_0(\De_\bX)$ for $\De_\bX:\bX\ra\bX\t\bX$ the derived diagonal morphism, so $\De_\bX$ is also affine. That is, $\bX$ has {\it affine diagonal}. Therefore if $\bU_1,\bU_2\hookra\bX$ are Zariski open inclusions with $\bU_1,\bU_2$ affine, then $\bU_1\t_\bX^h\bU_2\hookra\bX$ is also Zariski open with $\bU_1\t_\bX^h\bU_2$ affine. Thus, {\it finite intersections of open affine derived\/ $\C$-subschemes in a separated derived\/ $\C$-scheme\/ $\bX$ are affine}.
\label{vf2def4}
\end{dfn}

\subsection{\texorpdfstring{PTVV's shifted symplectic geometry}{PTVV\textquoteright s shifted symplectic geometry}}
\label{vf23}

Next we summarize parts of the theory of shifted symplectic geometry, as developed by Pantev, To\"{e}n, Vaqui\'{e}, and Vezzosi in \cite{PTVV}. We explain them for derived $\C$-schemes $\bX$, although Pantev et al.\ work more generally with derived stacks.

Given a (locally finitely presented) derived $\C$-scheme $\bX$ and $p\ge 0$, $k\in\Z$, Pantev et al.\ \cite{PTVV} define complexes of $k$-{\it shifted\/ $p$-forms\/} $\cA^p_\C(\bX,k)$ and $k$-{\it shifted closed\/ $p$-forms\/} $\cA^{p,\cl}_\C(\bX,k)$. These are defined first for affine derived $\C$-schemes $\bY=\bSpec A^\bu$ for $A^\bu$ a cdga over $\C$, and shown to satisfy \'etale descent. Then for general $\bX$, $k$-shifted (closed) $p$-forms are defined as a mapping stack; basically, a $k$-shifted (closed) $p$-form $\om$ on $\bX$ is the functorial choice for all $\bY,\bs f$ of a $k$-shifted (closed) $p$-form $\bs f^*(\om)$ on $\bY$ whenever $\bY=\bSpec A^\bu$ is affine and $\bs f:\bY\ra\bX$ is a morphism.

\begin{dfn} Let $\bY\simeq\bSpec A^\bu$ be an affine derived $\C$-scheme, for $A^\bu$ a cdga over $\C$. A $k$-{\it shifted\/ $p$-form\/} on $\bY$ for $k\in\Z$ is an element $\om_{A^\bu}\in(\La^p\bL_{A^\bu})^k$ with $\d\om_{A^\bu}=0$ in $(\La^p\bL_{A^\bu})^{k+1}$, so that $\om_{A^\bu}$ defines a cohomology class $[\om_{A^\bu}]\in H^k(\La^p\bL_{A^\bu})$.
When $p=2$, we call $\om_{A^\bu}$ {\it nondegenerate}, or a $k$-{\it shifted presymplectic form}, if the induced morphism $\om_{A^\bu}\cdot:\bT_{A^\bu}\ra\bL_{A^\bu}[k]$ is a quasi-isomorphism.

A {\it $k$-shifted closed\/ $p$-form\/} on $\bY$ is a sequence $\om_{A^\bu}^*=(\om^0_{A^\bu},\om^1_{A^\bu},\om^2_{A^\bu},\ldots)$ such that $\om^m_{A^\bu}\in(\La^{p+m}\bL_{A^\bu})^{k-m}$ for $m\ge 0$, with $\d\om_{A^\bu}^0=0$ and $\d\om_{A^\bu}^{1+m}+\dd\om_{A^\bu}^{m}=0$ in $(\La^{p+m+1}\bL_{A^\bu})^{k-m}$ for all $m\ge 0$. Note that if $\om_{A^\bu}^*=(\om^0_{A^\bu},\om^1_{A^\bu},\ldots)$ is a $k$-shifted closed $p$-form then $\om^0_{A^\bu}$ is a $k$-shifted $p$-form. 

When $p=2$, we call a $k$-shifted closed 2-form $\om_{A^\bu}^*$ a $k$-{\it shifted symplectic form\/} if the associated 2-form $\om^0_{A^\bu}$ is nondegenerate (presymplectic).

If $\bX$ is a general derived $\C$-scheme, then Pantev et al. \cite[\S 1.2]{PTVV} define $k$-{\it shifted\/ $2$-forms\/} $\om_\bX$, which may be {\it nondegenerate\/} ({\it presymplectic\/}), and $k$-{\it shifted closed\/ $2$-forms\/} $\om^*_\bX$, which have an associated $k$-shifted 2-form $\om_\bX^0$, and where $\om^*_\bX$ is called a $k$-{\it shifted symplectic form\/} if $\om^0_\bX$ is nondegenerate (presymplectic). We will not go into the details of this definition for general~$\bX$. 

The important thing for us is this: if $\bY\subseteq\bX$ is a Zariski open affine derived $\C$-subscheme with $\bY\simeq\bSpec A^\bu$ then a $k$-shifted 2-form $\om_\bX$ (or a $k$-shifted closed 2-form $\om^*_\bX$) on $\bX$ induces a $k$-shifted 2-form $\om_{A^\bu}$ (or a $k$-shifted closed 2-form $\om_{A^\bu}^*$) on $\bY$ in the sense above, where $\om_{A^\bu}$ is unique up to cohomology in the complex $((\La^2\bL_{A^\bu})^*,\d)$ (or $\om_{A^\bu}^*$ is unique up to cohomology in the complex $(\prod_{m\ge 0}(\La^{2+m}\bL_{A^\bu})^{*-m},\d+\dd)$), and $\om_\bX$ nondegenerate/presymplectic (or $\om_\bX^*$ symplectic) implies $\om_{A^\bu}$ nondegenerate/presymplectic (or $\om_{A^\bu}^*$ symplectic). 

It is easy to show that if $\bX$ is a derived $\C$-scheme with a $k$-shifted symplectic or presymplectic form, then $k\le 0$, and the complex virtual dimension $\vdim_\C\bX$ satisfies $\vdim_\C\bX=0$ if $k$ is odd, and $\vdim_\C\bX$ is even if $k\equiv 0\mod 4$ (which includes classical complex symplectic schemes when $k=0$), and $\vdim_\C\bX\in\Z$ if $k\equiv 2\mod 4$. In particular, in the case $k=-2$ of interest in this paper, $\vdim_\C\bX$ can take any value in~$\Z$.
\label{vf2def5}
\end{dfn}

The main examples we have in mind come from Pantev et al.~\cite[\S 2.1]{PTVV}:

\begin{thm} Suppose $Y$ is a Calabi--Yau $m$-fold over $\C,$ and\/ $\bs\cM$ a derived moduli stack of coherent sheaves (or complexes of coherent sheaves) on $Y$. Then $\bs\cM$ has a natural\/ $(2-m)$-shifted symplectic form $\om_{\bs\cM}$.
\label{vf2thm2}
\end{thm}

In particular, derived moduli schemes and stacks on a Calabi--Yau 4-fold $Y$ are $-2$-shifted symplectic. 

Bussi, Brav and Joyce \cite{BBJ} prove `Darboux Theorems' for $k$-shifted symplectic derived $\C$-schemes $(\bX,\om_\bX)$ for $k<0$, which give explicit Zariski local models for $(\bX,\om_\bX)$. We will explain their main result for $k=-2$. The next definition is taken from \cite[Ex.~5.16]{BBJ} (with notation changed, $2q_js_j$ in place of~$s_j$). 

\begin{dfn} A pair $(A^\bu,\om_{A^\bu})$ is called in {\it $-2$-Darboux form\/} if $A^\bu$ is a standard form cdga over $\C$, and $\om_{A^\bu}\in(\La^2\bL_{A^\bu})^{-2}=(\La^2\Om^1_{A^\bu})^{-2}$ with $\d\om_{A^\bu}=0$ in $(\La^2\bL_{A^\bu})^{-1}$ and $\dd\om_{A^\bu}=0$ in $(\La^3\bL_{A^\bu})^{-2}$, so that $\om_{A^\bu}^*:=(\om_{A^\bu},0,0,\ldots)$ is a $-2$-shifted closed 2-form on $A^\bu$,  such that:
\begin{itemize}
\setlength{\itemsep}{0pt}
\setlength{\parsep}{0pt}
\item[(i)] $A^0$ is a smooth $\C$-algebra of dimension $m$, and there exist $x_1,\ldots,x_m$ in $A^0$ forming an \'etale coordinate system on $V=\Spec A^0$.
\item[(ii)] The commutative graded algebra $A^*$ is freely generated over $A^0$ by elements $y_1,\ldots,y_n$ of degree $-1$ and $z_1,\ldots,z_m$ of degree $-2$.
\item[(iii)] There are invertible elements $q_1,\ldots,q_n$ in $A^0$ such that
\e
\begin{split}
\om_{A^\bu}&=\dd z_1\,\dd x_1+\cdots+\dd z_m\,\dd
x_m\\
&\qquad+\dd\bigl(q_1y_1\bigr)\,\dd y_1+\cdots+
\dd\bigl(q_ny_n\bigr)\,\dd y_n.
\end{split}
\label{vf2eq1}
\e
\item[(iv)] There are elements $s_1,\ldots,s_n\in A^0$ satisfying
\e
q_1(s_1)^2+\cdots+q_n(s_n)^2=0\quad\text{in
$A^0$,}
\label{vf2eq2}
\e
such that the differential $\d$ on $A^\bu=(A^*,\d)$ is given by
\e
\d x_i=0,\quad \d y_j=s_j, \quad \d z_i=
\sum_{j=1}^ny_j\biggl(2q_j\frac{\pd s_j}{\pd x_i}
+s_j\,\frac{\pd q_j}{\pd x_i}\biggr).
\label{vf2eq3}
\e

\end{itemize}

Here the only assumptions are that $A^0,x_1,\ldots,x_m$ are as in (i) and we are given $q_1,\ldots,q_n,\ab s_1,\ab\ldots,\ab s_n$ in $A^0$ satisfying \eq{vf2eq2}, and everything else follows from these. Defining $A^*$ as in (ii) and $\d$ as in \eq{vf2eq3}, then $A^\bu=(A^*,\d)$ is a standard form cdga over $\C$, where to show that $\d\ci\d z_i=0$ we apply $\frac{\pd}{\pd x_i}$ to \eq{vf2eq2}. Clearly $\dd\om_{A^\bu}=0$, as $\dd\ci\dd=0$. We have
\begin{align*}
&\d\om_{A^\bu}=\ts\sum\limits_{i=1}^m(\d\ci\dd z_i)\dd x_i+\sum\limits_{j=1}^n(\d\ci\dd(q_jy_j))\dd y_j+(\d\ci\dd y_j)\dd(q_jy_j)\\
&=-\dd\ts\sum\limits_{i=1}^m\d z_i\dd x_i-\dd\sum\limits_{j=1}^n\bigl[\d(q_jy_j)\dd y_j+\d y_j\dd(q_jy_j)\bigr]\\
&=-\dd\ts\sum\limits_{i=1}^m\sum\limits_{j=1}^ny_j\bigl(2q_j\frac{\pd s_j}{\pd x_i}
+s_j\,\frac{\pd q_j}{\pd x_i}\bigr)\dd x_i
-\dd\sum\limits_{j=1}^n\bigl[q_js_j\dd y_j+s_j\dd(q_jy_j)\bigr]\\
&=-\dd\ci\dd\ts\sum\limits_{j=1}^n\bigl[(q_js_j)y_j+s_j(q_jy_j)\bigr]=0,
\end{align*}
using \eq{vf2eq1} and $\d\ci\dd x_i=0$ for degree reasons in the first step, $\d\ci\dd=-\dd\ci\d$ and $\dd\ci\dd=0$ in the second, \eq{vf2eq3} in the third, $\d s_j=\sum_{i=1}^n\frac{\pd s_j}{\pd x_i}\dd x_i$ and similarly for $q_j$ in the fourth, and $\dd\ci\dd=0$ in the fifth. Hence $\om_{A^\bu}^*$ is a $-2$-shifted closed 2-form on $A^\bu$.

The action $\om_{A^\bu}\cdot:\bT_{A^\bu}\ra\bL_{A^\bu}[-2]$ is given by
\begin{gather*}
\om_{A^\bu}\cdot \frac{\pd}{\pd x_i}=-\dd z_i+\sum_{j=1}^n\frac{\pd q_j}{\pd x_i}y_j\,\dd y_j,
\\
\om_{A^\bu}\cdot \frac{\pd}{\pd y_j}=2q_j\dd y_j-\sum_{i=1}^my_j\frac{\pd q_j}{\pd x_i}\dd x_i,\qquad \om_{A^\bu}\cdot \frac{\pd}{\pd z_i}=\dd x_i.
\end{gather*}
By writing this as an upper triangular matrix with invertible diagonal (since the $q_j$ are invertible), we see that $\om_{A^\bu}\cdot$ is actually an isomorphism of complexes, so a quasi-isomorphism, and $\om_{A^\bu}^*$ is a $-2$-shifted symplectic form on $A^\bu$.
\label{vf2def6}
\end{dfn}

The main result of Bussi, Brav and Joyce \cite[Th.~5.18]{BBJ} when $k=-2$ yields:

\begin{thm} Suppose\/ $(\bX,\om_\bX^*)$ is a $-2$-shifted symplectic derived\/ $\C$-scheme. Then for each\/ $x\in X=t_0(\bX)$ there exists a pair $(A^\bu,\om_{A^\bu})$ in $-2$-Darboux form and a Zariski open inclusion $\bs\al:\bSpec A^\bu\hookra\bX$ such that\/ $x\in\Im\bs\al$ and\/ $\bs\al^*(\om_\bX^*)\!\simeq\! \om_{A^\bu}$ in $\cA^{2,\cl}_\C(\bSpec A^\bu,-2)$. Furthermore, we can choose $A^\bu$ \begin{bfseries}minimal\end{bfseries} at\/ $x,$ in the sense that\/ $m\!=\!\dim H^0(\bT_\bX\vert_x),$ $n\!=\!\dim H^1(\bT_\bX\vert_x)$ in Definition\/~{\rm\ref{vf2def6}}.

\label{vf2thm3}
\end{thm}

\subsection{\texorpdfstring{Orientations on $k$-shifted symplectic derived schemes}{Orientations on k-shifted symplectic derived schemes}}
\label{vf24}

If $\bX$ is a derived $\C$-scheme (always assumed locally finitely presented), with classical $\C$-scheme $X=t_0(\bX)$, the cotangent complex $\bL_\bX\vert_X$ restricted to $X$ is a perfect complex, so it has a determinant line bundle $\det(\bL_\bX\vert_X)$ on~$X$.

The following notion is important for $-1$-shifted symplectic derived schemes, 3-Calabi--Yau moduli spaces, and generalizations of Donaldson--Thomas theory:

\begin{dfn} Let $(\bX,\om_\bX^*)$ be a $-1$-shifted symplectic derived $\C$-scheme (or more generally $k$-shifted symplectic, for $k<0$ odd). An {\it orientation\/} for $(\bX,\om_\bX^*)$ is a choice of square root line bundle $\det(\bL_\bX\vert_X)^{1/2}$ for~$\det(\bL_\bX\vert_X)$.

Writing $X_\an$ for the complex analytic topological space of $X$, the obstruction to existence of orientations for $(\bX,\om_\bX^*)$ lies in $H^2(X_\an;\Z_2)$, and if the obstruction vanishes, the set of orientations is a torsor for~$H^1(X_\an;\Z_2)$.
\label{vf2def7}
\end{dfn}

This notion of orientation, and its analogue for `d-critical loci', are used by Ben-Bassat, Brav, Bussi, Dupont, Joyce, Meinhardt, and Szendr\H oi in a series of papers \cite{BBBJ,BBDJS,BBJ,BJM,Joyc5}. They use orientations on $(\bX,\om_\bX^*)$ to define natural perverse sheaves, $\cD$-modules, mixed Hodge modules, and motives on $X$. A similar idea first appeared in Kontsevich and Soibelman \cite[\S 5]{KoSo} as `orientation data' needed to define motivic Donaldson--Thomas invariants of Calabi--Yau 3-folds.

This paper concerns $-2$-shifted symplectic derived schemes, and 4-Calabi--Yau moduli spaces. It turns out that there is a parallel notion of orientation in the $-2$-shifted case, needed to construct virtual cycles. 

To define this, note that determinant line bundles $\det(E^\bu)$ of perfect complexes $\cE^\bu$ satisfy $\det[(E^\bu)^\vee]\cong [\det(E^\bu)]^{-1}$, and $\det(E^\bu[k])\cong [\det(E^\bu)]^{(-1)^k}$. If $(\bX,\om_\bX^*)$ is a $k$-shifted symplectic derived $\C$-scheme, then $\bT_\bX\simeq\bL_\bX[k]$, where $\bT_\bX\simeq(\bL_\bX)^\vee$. Restricting to $X$ and taking determinant line bundles gives $\det(\bL_\bX\vert_X)^{-1}\cong \det(\bL_\bX\vert_X)^{(-1)^k}$. If $k$ is odd this is trivial, but for $k$ even, this gives a canonical isomorphism of line bundles on $X$:
\e
\smash{\io_{\bX,\om_\bX^*}:\bigl[\det(\bL_\bX\vert_X)\bigr]{}^{\ot^2}\longra \O_X\cong\O_X^{\ot^2}.}
\label{vf2eq4}
\e

The next definition is new, so far as the authors know.

\begin{dfn} Let $(\bX,\om_\bX^*)$ be a $-2$-shifted symplectic derived $\C$-scheme (or more generally $k$-shifted symplectic, for $k<0$ with $k\equiv 2\mod 4$). An {\it orientation\/} for $(\bX,\om_\bX^*)$ is a choice of isomorphism $o:\det(\bL_\bX\vert_X)\ra\O_X$ such that $o\ot o=\io_{\bX,\om_\bX^*}$, for $\io_{\bX,\om_\bX^*}$ as in~\eq{vf2eq4}.

Writing $X_\an$ for the complex analytic topological space of $X$, the obstruction to existence of orientations for $(\bX,\om_\bX^*)$ lies in $H^1(X_\an;\Z_2)$, and if the obstruction vanishes, the set of orientations is a torsor for~$H^0(X_\an;\Z_2)$.
\label{vf2def8}
\end{dfn}

This definition makes sense for $k$-shifted symplectic derived $\C$-schemes with $k$ even, but when $k\equiv 0\mod 4$ (including the classical symplectic case $k=0$) there is a natural choice of orientation $o$, so we restrict to~$k\equiv 2\mod 4$.

At a point $x\in X_\an$, we have a canonical isomorphism
\begin{equation*}
\smash{\det(\bL_\bX\vert_x)\cong \La^{\rm top} H^0(\bL_\bX\vert_x)\ot [\La^{\rm top} H^{-1}(\bL_\bX\vert_x)]^*\ot \La^{\rm top} H^{-2}(\bL_\bX\vert_x).}
\end{equation*}
Now $H^{-1}(\bL_\bX\vert_x)\cong H^1(\bT_\bX\vert_x)^*$, and $\om_\bX^0\vert_x$ gives $H^0(\bL_\bX\vert_x)\cong H^{-2}(\bL_\bX\vert_x)^*$, so $\La^{\rm top} H^0(\bL_\bX\vert_x)\cong [\La^{\rm top} H^{-2}(\bL_\bX\vert_x)]^*$. Thus we have a canonical isomorphism
\e
\smash{\det(\bL_\bX\vert_x)\cong \La^{\rm top} H^1(\bT_\bX\vert_x).}
\label{vf2eq5}
\e

Write $Q_x$ for the nondegenerate, symmetric $\C$-bilinear pairing
\e
\smash{Q_x:=\om_\bX^0\vert_x\,\cdot:H^1(\bT_\bX\vert_x)\t H^1(\bT_\bX\vert_x)\longra\C.}
\label{vf2eq6}
\e
The determinant $\det Q_x$ is an isomorphism $[\La^{\rm top} H^1(\bT_\bX\vert_x)]^{\ot^2}\ra\C$, and $\det Q_x$ corresponds to $\io_{\bX,\om_\bX^*}\vert_x$ under the isomorphism \eq{vf2eq5}. There is a natural bijection
\e
\smash{\!\!\bigl\{\text{orientations on $(\bX,\om_\bX^*)$ at $x$}\bigr\}\!\cong\!\bigl\{\text{$\C$-orientations on $(H^1(\bT_\bX\vert_x),Q_x)$}\bigr\}.\!\!}
\label{vf2eq7}
\e
To see this, note that if $(e_1,\ldots,e_n)$ is an orthonormal basis for $(H^1(\bT_\bX\vert_x),Q_x)$ then $e_1\w\cdots\w e_n$ lies in $\La^{\rm top}H^1(\bT_\bX\vert_x)$ with $\det Q_x:[e_1\w\cdots\w e_n]^{\ot^2}\mapsto 1$. Orientations for $(\bX,\om_\bX^*)$ at $x$ give isomorphisms $\la:\La^{\rm top}H^1(\bT_\bX\vert_x)\ra\C$ with $\la^2=\det Q_x$, and these correspond to orientations for $(H^1(\bT_\bX\vert_x),Q_x)$ such that $\la:e_1\w\cdots\w e_n\mapsto 1$ if $(e_1,\ldots,e_n)$ is an oriented orthonormal basis.

\subsection{Kuranishi atlases}
\label{vf25}

We now define our notion of {\it Kuranishi atlas\/} on a topological space $X$. These are a simplification of m-Kuranishi spaces in \cite[\S 4.7]{Joyc6}, which in turn are based on the `Kuranishi spaces' of Fukaya, Oh, Ohta, Ono~\cite{FOOO,FuOn}.

\begin{dfn} Let $X$ be a topological space. A {\it Kuranishi neighbourhood\/} on $X$ is a quadruple $(V,E,s,\psi)$ such that:
\begin{itemize}
\setlength{\itemsep}{0pt}
\setlength{\parsep}{0pt}
\item[(a)] $V$ is a smooth manifold.
\item[(b)] $\pi:E\ra V$ is a real vector bundle over $V$, called the {\it obstruction bundle}.
\item[(c)] $s:V\ra E$ is a smooth section of $E$, called the {\it Kuranishi section}.
\item[(d)] $\psi$ is a homeomorphism from $s^{-1}(0)$ to an open subset $R=\Im\psi$ in $X$, where $\Im\psi=\bigl\{\psi(x):x\in s^{-1}(0)\bigr\}$ is the image of $\psi$.
\end{itemize}
If $S\subseteq X$ is open, by a {\it Kuranishi neighbourhood over\/} $S$, we mean a Kuranishi neighbourhood $(V,E,s,\psi)$ on $X$ with $S\subseteq\Im\psi\subseteq X$.
\label{vf2def10}
\end{dfn}

\begin{dfn} Let $(V_J,E_J,s_J,\psi_J),(V_K,E_K,s_K,\psi_K)$ be Kuranishi neighbourhoods on a topological space $X$, and $S\subseteq\Im\psi_J\cap\Im\psi_K\subseteq X$ be open. A {\it coordinate change\/ $\Phi_{JK}:(V_J,E_J,s_J,\psi_J)\!\ra\!(V_K,E_K,s_K,\psi_K)$ over\/} $S$ is a triple $\Phi_{JK}=(V_{JK},\phi_{JK},\hat\phi_{JK})$ satisfying:
\begin{itemize}
\setlength{\itemsep}{0pt}
\setlength{\parsep}{0pt}
\item[(a)] $V_{JK}$ is an open neighbourhood of $\psi_J^{-1}(S)$ in $V_J$.
\item[(b)] $\phi_{JK}:V_{JK}\ra V_K$ is a smooth map.
\item[(c)] $\hat\phi_{JK}:E_J\vert_{V_{JK}}\ra\phi_{JK}^*(E_K)$ is a morphism of vector bundles on $V_{JK}$.
\item[(d)] $\hat\phi_{JK}(s_J\vert_{V_{JK}})=\phi_{JK}^*(s_K)$.
\item[(e)] $\psi_J=\psi_K\ci\phi_{JK}$ on $s_J^{-1}(0)\cap V_{JK}$.
\item[(f)] Let\/ $x\in S,$ and set\/ $v_J=\psi_J^{-1}(x)\in V_J$ and\/ $v_K=\psi_K^{-1}(x)\in V_K$. Then the following is an exact sequence of real vector spaces:
\end{itemize}
\e
\begin{gathered}
\smash{\xymatrix@C=11pt{ 0 \ar[r] & T_{v_J}V_J \ar[rrrr]^(0.39){\d s_J\vert_{v_J}\op\d\phi_{JK}\vert_{v_J}} &&&& E_J\vert_{v_J} \!\op\!T_{v_K}V_K \ar[rrrr]^(0.56){-\hat\phi_{JK}\vert_{v_J}\op \d s_K\vert_{v_K}} &&&& E_K\vert_{v_K} \ar[r] & 0. }}
\end{gathered}
\label{vf2eq8}
\e

We can {\it compose coordinate changes\/}: if 
$\Phi_{JK}=(V_{JK},\ab\phi_{JK},\ab\hat\phi_{JK}):(V_J,\ab E_J,\ab s_J,\ab\psi_J)\ra(V_K,E_K,s_K,\psi_K)$ and $\Phi_{KL}(V_{KL},\phi_{KL},\hat\phi_{KL}):(V_K,\ab E_K,\ab s_K,\ab\psi_K)\ra(V_L,E_L,s_L,\psi_L)$ are coordinate changes over $S_{JK},S_{KL}$, then
\begin{gather*}
\Phi_{KL}\ci\Phi_{JK}:=\bigl(V_{JK}\cap\phi_{JK}^{-1}(V_{KL}),\phi_{KL}\ci\phi_{JK}\vert_{\cdots},\phi_{JK}^*(\hat\phi_{KL})\ci\hat \phi_{JK}\vert_{\cdots}\bigr):\\
(V_J,E_J,s_J,\psi_J)\longra(V_L,E_L,s_L,\psi_L)
\end{gather*}
is a coordinate change over~$S_{JK}\cap S_{KL}$.
\label{vf2def11}
\end{dfn}

\begin{dfn} A {\it Kuranishi atlas\/ $\cK$ of virtual dimension\/} $n$ on a topological space $X$ is data $\cK=\bigl(A,\pr,(V_J,E_J,s_J,\psi_J)_{J\in A},\Phi_{JK,\;J\pr K\in A}\bigr)$, where:
\begin{itemize}
\setlength{\itemsep}{0pt}
\setlength{\parsep}{0pt}
\item[(a)] $A$ is an indexing set (not necessarily finite).
\item[(b)] $\prec$ is a partial order on $A$, where by convention $J\pr K$ only if $J\ne K$.
\item[(c)] $(V_J,E_J,s_J,\psi_J)$ is a Kuranishi neighbourhood on $X$ for each $J\in A$, with $\dim V_J-\rank E_J=n$.
\item[(d)] The images $\Im\psi_J\subseteq X$ for $J\in A$ have the property that if $J,K\in A$ with $J\ne K$ and $\Im\psi_J\cap\Im\psi_K\ne\es$ then either $J\pr K$ or $K\pr J$.
\item[(e)] $\Phi_{JK}=(V_{JK},\phi_{JK},\hat\phi_{JK}):(V_J,E_J,s_J,\psi_J)\ra(V_K,E_K,s_K,\psi_K)$ is a coordinate change for all $J,K\in A$ with $J\pr K$, over $S=\Im\psi_J\cap\Im\psi_K$.
\item[(f)] $\Phi_{KL}\ci\Phi_{JK}=\Phi_{JL}$ for all $J,K,L\in A$ with $J\pr K\pr L$.
\item[(g)] $\bigcup_{J\in A}\Im\psi_J=X$.
\end{itemize}

We call $\cK$ a {\it finite\/} Kuranishi atlas if the indexing set $A$ is finite.

If $X$ has a Kuranishi atlas then it is locally compact. In applications we invariably impose extra global topological conditions on $X$, for instance $X$ might be assumed to be compact and Hausdorff; or Hausdorff and second countable; or metrizable; or Hausdorff and paracompact.

We will also need a relative version of Kuranishi atlas in \S\ref{vf37}. Suppose $Z$ is a manifold, and $\pi:X\ra Z$ a continuous map. A {\it relative Kuranishi atlas\/} for $\pi:X\ra Z$ is a Kuranishi atlas $\cK$ on $X$ as above, together with smooth maps $\varpi_J:V_J\ra Z$ for $J\in A$, such that $\varpi_J\vert_{s_J^{-1}(0)}=\pi\ci\psi_J:s_J^{-1}(0)\ra Z$ for all $J\in A$, and $\varpi_J\vert_{V_{JK}}=\varpi_K\ci\phi_{JK}:V_{JK}\ra Z$ for all $J\pr K$ in~$A$.
\label{vf2def12}
\end{dfn}

\begin{dfn} Let $X$ be a topological space, with a Kuranishi atlas $\cK$ as in Definition \ref{vf2def12}. For each $J\in A$ we can form the $C^\iy$ real line bundle $\La^{\rm top}T^*V_J\ot\La^{\rm top}E_J$ over $V_J$, where $\La^{\rm top}(\cdots)$ means the top exterior power. Thus we can form the restriction
\begin{equation*}
\smash{(\La^{\rm top}T^*V_J\ot \La^{\rm top}E_J)\vert_{s_J^{-1}(0)}\longra s_J^{-1}(0),}
\end{equation*}
considered as a topological real line bundle over the topological space $s_J^{-1}(0)$.

If $J\pr K$ in $A$ then for each $v_J$ in $s_J^{-1}(0)\cap V_{JK}$ with $\phi_{JK}(v_J)=v_K$ in $s_K^{-1}(0)$ we have an exact sequence \eq{vf2eq8}. Taking top exterior powers in \eq{vf2eq8} (and using a suitable orientation convention) gives an isomorphism
\begin{equation*}
\smash{\La^{\rm top}T_{v_J}^*V_J\ot \La^{\rm top}E_J\vert_{v_J}\cong \La^{\rm top}T^*_{v_K}V_K\ot \La^{\rm top}E_K\vert_{v_K}.}
\end{equation*}
This depends continuously on $v_J,v_K$, and so induces an isomorphism of topological line bundles on $s_J^{-1}(0)\cap V_{JK}$
\begin{equation*}
\smash{(\Phi_{JK})_*:(\La^{\rm top}T^*V_J\ot \La^{\rm top}E_J)\vert_{s_J^{-1}(0)\cap V_{JK}}\longra\phi_{JK}\vert_{\cdots}^*(\La^{\rm top}T^*V_K\ot \La^{\rm top}E_K).}
\end{equation*}
If $J\pr K\pr L$ in $A$ then as $\Phi_{KL}\ci\Phi_{JK}=\Phi_{JL}$ by Definition \ref{vf2def12}(f), we see that $(\Phi_{KL})_*\ci(\Phi_{JK})_*=(\Phi_{JL})_*$ in topological line bundles over $s_J^{-1}(0)\cap V_{JK}\cap V_{JL}$.

An {\it orientation\/} on $(X,\cK)$ is a choice of orientation on the fibres of the topological real line bundle $(\La^{\rm top}T^*V_J\ot \La^{\rm top}E_J)\vert_{s_J^{-1}(0)}$ on $s_J^{-1}(0)$ for all $J\in A$, such that $(\Phi_{JK})_*$ is orientation-preserving on $s_J^{-1}(0)\cap V_{JK}$ for all $J\pr K$ in $A$.

An equivalent way to think about this is that there is a natural topological real line bundle $K_X\ra X$ called the {\it canonical bundle\/} with given isomorphisms
\begin{equation*}
\smash{\io_J:(\La^{\rm top}T^*V_J\ot \La^{\rm top}E_J)\vert_{s_J^{-1}(0)}\longra\psi_J^*(K_X)}
\end{equation*}
for $J\in A$, such that $\io_J\vert_{s_J^{-1}(0)\cap V_{JK}}=\phi_{JK}^*(\io_K)\ci(\Phi_{JK})_*$ for all $J\pr K$ in $A$, and an orientation on $(X,\cK)$ is an orientation on the fibres of~$K_X$.
\label{vf2def13}
\end{dfn}

\begin{rem}{\bf(a)} Our Kuranishi atlases are based on Joyce's `m-Kuranishi spaces' \cite[\S 4.7]{Joyc6}. They are similar to Fukaya--Oh--Ohta--Ono's `good coordinate systems' \cite[Lem.~A1.11]{FOOO}, \cite[Def.~6.1]{FuOn}, and McDuff--Wehrheim's `Kuranishi atlases' \cite{McDu,McWe}. Our orientations are based on \cite[Def.~5.8]{FuOn} and~\cite[Def.~A1.17]{FOOO}.

There are two important differences with \cite{FOOO,FuOn,McDu,McWe}. Firstly, \cite{FOOO,FuOn,McDu,McWe} use Kuranishi neighbourhoods $(V,E,\Ga,s,\psi)$, where $\Ga$ is a finite group acting equivariantly on $V,E,s$ and $\psi$ maps $s^{-1}(0)/\Ga\ra X$. This is because their Kuranishi spaces are a kind of derived orbifolds, not derived manifolds. 

Secondly, \cite{FOOO,FuOn,McDu,McWe} use a more restrictive notion of coordinate change $\Phi_{JK}=(V_{JK},\phi_{JK},\hat\phi_{JK})$, in which $\phi_{JK}:V_{JK}\hookra V_K$ must be an embedding, and $\hat\phi_{JK}:E_J\vert_{V_{JK}}\hookra\phi_{JK}^*(E_K)$ an embedding of vector bundles, so that $\dim V_J\le\dim V_K$ and $\rank E_J\le\rank E_K$. In the Kuranishi atlases we construct later, $\phi_{JK}:V_{JK}\ra V_K$ will be a submersion, and $\hat\phi_{JK}:E_J\vert_{V_{JK}}\ra\phi_{JK}^*(E_K)$ will be surjective, so that $\dim V_J\ge\dim V_K$ and $\rank E_J\ge\rank E_K$. That is, our coordinate changes actually go {\it the opposite way\/} to those in~\cite{FOOO,FuOn,McDu,McWe}.
\smallskip

\noindent{\bf(b)} Similar structures to Kuranishi atlases are studied \cite{FOOO,FuOn,Joyc6,McDu,McWe} because it is natural to construct them on many differential-geometric moduli spaces. Broadly speaking, any moduli space of solutions of a smooth nonlinear elliptic p.d.e.\ on a compact manifold should admit a Kuranishi atlas. References \cite{FOOO,FuOn,McDu,McWe} concern moduli spaces of $J$-holomorphic curves in symplectic geometry.
\label{vf2rem2}
\end{rem}

\subsection{Derived smooth manifolds and virtual classes}
\label{vf26}

Readers of this paper do not need to know what a derived manifold is. Here is a brief summary of the points relevant to this paper:
\begin{itemize}
\setlength{\itemsep}{0pt}
\setlength{\parsep}{0pt}
\item `Derived manifolds' are derived versions of smooth manifolds, where `derived' is in the sense of Derived Algebraic Geometry. 
\item There are several different versions, due to Spivak \cite{Spiv}, Borisov--Noel \cite{Bori,BoNo} and Joyce \cite{Joyc2,Joyc3,Joyc4,Joyc6}, which form $\iy$-categories or 2-categories. They all include ordinary manifolds $\bf Man$ as a full subcategory.
\item All these versions are roughly equivalent. There are natural 1-1 correspondences between equivalence classes of derived manifolds in each theory.
\item Much of classical differential geometry generalizes nicely to derived manifolds -- submersions, orientations, transverse fibre products, \ldots.
\item Given a Hausdorff, second countable topological space $X$ with a Kuranishi atlas $\cK$ of dimension $n$, we can construct a derived manifold $\bX$ with topological space $X$ and dimension $\vdim\bX=n$, unique up to equivalence. Orientations on $(X,\cK)$ are in 1-1 correspondence with orientations on $\bX$.
\item Compact, oriented derived manifolds $\bX$ have {\it virtual classes\/} $[\bX]_\virt$ in homology or bordism, generalizing the fundamental class $[X]\in H_{\dim X}(X;\Z)$ of a compact oriented manifold $X$.
\item These virtual classes are used to define enumerative invariants such as Gromov--Witten, Donaldson, and Donaldson--Thomas invariants. Such invariants are unchanged under deformations of the underlying geometry.
\item Given a compact Hausdorff topological space $X$ with an oriented Kuranishi atlas $\cK$, we could construct the virtual class $[\bX]_\virt$ directly from $(X,\cK)$, as in \cite{FOOO,FuOn,McDu,McWe}, without going via the derived manifold $\bX$. 
\end{itemize}
Readers who do not want to know more details can now skip forward to~\S\ref{vf3}.

\subsubsection{Different definitions of derived manifold}
\label{vf261}

The earliest reference to derived differential geometry we are aware of is a short final paragraph by Jacob Lurie \cite[\S 4.5]{Luri}. Broadly following \cite[\S 4.5]{Luri}, Lurie's student David Spivak \cite{Spiv} constructed an $\iy$-category $\DerManSp$ of `derived manifolds'. Borisov and No\"el \cite{BoNo} gave a simplified version, an $\iy$-category $\DerManBN$, and showed that~$\DerManSp\simeq\DerManBN$.

Joyce \cite{Joyc2,Joyc3,Joyc4} defined 2-categories $\dMan$ of `d-manifolds' (a kind of derived manifold), and $\dOrb$ of `d-orbifolds' (a kind of derived orbifold), and also strict 2-categories of d-manifolds and d-orbifolds with boundary $\dManb,\dOrbb$ and with corners $\dManc,\dOrbc$, and studied their differential geometry in detail.

Borisov \cite{Bori} constructed a 2-functor $F:\pi_2(\DerManBN)\ra\dMan$, where $\pi_2(\DerManBN)$ is the 2-category truncation of $\DerManBN$, and proved that $F$ is close to being an equivalence of 2-categories.

All of \cite{Bori,BoNo,Joyc2,Joyc3,Joyc4,Luri,Spiv} use `$C^\iy$-algebraic geometry', as in Joyce \cite{Joyc1}, a version of (derived) algebraic geometry in which rings are replaced by `$C^\iy$-rings', and define derived manifolds to be special kinds of `derived $C^\iy$-schemes'.

In \cite{Joyc6}, Joyce gave an alternative approach to derived differential geometry based on the work of Fukaya et al.\ \cite{FuOn,FOOO}. He defined 2-categories of `m-Kuranishi spaces' $\mKur$, a kind of derived manifold, and `Kuranishi spaces' $\Kur$, a kind of derived orbifold. Here m-Kuranishi spaces are similar to a pair $(X,\cK)$ of a Hausdorff, second countable topological space $X$ and a Kuranishi atlas $\cK$ in the sense of \S\ref{vf25}. In \cite{Joyc7} Joyce defines equivalences of 2-categories $\dMan\simeq\mKur$ and $\dOrb\simeq\Kur$, showing that the two approaches to derived differential geometry of \cite{Joyc2,Joyc3,Joyc4} and \cite{Joyc6} are essentially the same.

\subsubsection{Orientations on derived manifolds}
\label{vf262}

Derived manifolds have a good notion of {\it orientation}, which behaves much like orientations on ordinary manifolds. Some references are Joyce \cite[\S 4.8]{Joyc2}, \cite[\S 4.8]{Joyc3}, \cite[\S 4.6]{Joyc4} for d-manifolds, Joyce \cite{Joyc7} for m-Kuranishi spaces, and Fukaya, Oh, Ohta and Ono \cite[\S 5]{FuOn}, \cite[\S A1.1]{FOOO} for Kuranishi spaces in their sense.

For any kind of derived manifold $\bX$, we can define a (topological or $C^\iy$) real line bundle $K_\bX$ over the topological space $X$ called the {\it canonical bundle}. It is the determinant line bundle of the cotangent complex $\bL_\bX$. For each $x\in\bX$ we can define a {\it tangent space\/} $T_x\bX$ and {\it obstruction space\/} $O_x\bX$, and then
\begin{equation*}
\smash{K_\bX\vert_x\cong \La^{\rm top}T^*_x\bX\ot_\R \La^{\rm top}O_x\bX.}
\end{equation*}
An {\it orientation\/} on $\bX$ is an orientation on the fibres of $K_\bX$. In a similar way to \eq{vf2eq7}, at a single point $x\in X$ we have a natural bijection
\e
\smash{\bigl\{\text{orientations on $\bX$ at $x$}\bigr\}\cong\bigl\{\text{orientations on $T_x^*\bX\op O_x\bX$}\bigr\}.}
\label{vf2eq9}
\e

If $(V,E,s,\psi)$ is a Kuranishi neighbourhood on $\bX$ and $v\in s^{-1}(0)\subseteq V$ with $\psi(v)=x\in\bX$, then there is a natural exact sequence
\e
\smash{\xymatrix@C=20pt{ 0 \ar[r] & T_x\bX \ar[r] & T_vV \ar[rr]^{\d s\vert_v} && E\vert_v \ar[r] & O_x\bX \ar[r] & 0. }}
\label{vf2eq10}
\e
Taking top exterior powers in \eq{vf2eq10} gives an isomorphism
\begin{equation*}
\smash{K_\bX\vert_x\cong\La^{\rm top}T^*_x\bX\ot_\R \La^{\rm top}O_x\bX\cong\La^{\rm top}T_v^*V\ot_\R \La^{\rm top}E\vert_v,}
\end{equation*}
and thus, with a suitable orientation convention, a natural bijection
\begin{equation*}
\smash{\bigl\{\text{orientations on $\bX$ at $x$}\bigr\}\cong\bigl\{\text{orientations on $T_v^*V\op E\vert_v$}\bigr\}.}
\end{equation*}

\subsubsection{Kuranishi atlases and derived manifolds}
\label{vf263}

The next theorem relates topological spaces with Kuranishi atlases to derived manifolds. The assumption that $X$ is Hausdorff and second countable is just to match the global topological assumptions in \cite{BoNo,Joyc2,Joyc3,Joyc4,Joyc6,Spiv}. For the last part we restrict to (a),(b) as orientations have not been written down for the theories of (c),(d), although this would not be very difficult.

\begin{thm} Let\/ $X$ be a Hausdorff, second countable topological space with a Kuranishi atlas $\cK$ of dimension $n$ in the sense of\/ {\rm\S\ref{vf25}}. Then we can construct
\begin{itemize}
\setlength{\itemsep}{0pt}
\setlength{\parsep}{0pt}
\item[{\bf(a)}] An \begin{bfseries}m-Kuranishi space\end{bfseries}\/ $\bX$ in the sense of Joyce {\rm\cite[\S 4.7]{Joyc6}}.
\item[{\bf(b)}] A \begin{bfseries}d-manifold\end{bfseries}\/ $\bX$ in the sense of Joyce {\rm\cite{Joyc2,Joyc3,Joyc4}}.
\item[{\bf(c)}] A \begin{bfseries}derived manifold\end{bfseries} in the sense of Borisov and No\"el\/ {\rm\cite{BoNo}}.
\item[{\bf(d)}] A \begin{bfseries}derived manifold\end{bfseries} in the sense of Spivak\/ {\rm\cite{Spiv}}.
\end{itemize}
In each case $\bX$ has topological space $X$ and dimension $\vdim\bX=n,$ and $\bX$ is canonical up to equivalence in the $2$-categories $\mKur,\dMan$ or\/ $\iy$-categories $\DerManBN,\DerManSp$. In cases\/ {\bf(a)\rm,\bf(b)} there is a natural\/ $1$-$1$ correspondence between orientations on $\cK,$ and orientations on $\bX$ in\/ {\rm\cite{Joyc7}} and\/~{\rm\cite{Joyc2,Joyc3,Joyc4}}. 

If also $Z$ is a manifold, $\pi:X\ra Z$ is continuous, and\/ $\cK,\varpi_{J:J\in A}$ is a relative Kuranishi atlas for $\pi:X\ra Z,$ then we can construct a morphism of derived manifolds $\bs\pi:\bX\ra Z,$ canonical up to $2$-isomorphism, with continuous map $\pi$.

\label{vf2thm4}
\end{thm}

\begin{proof} Part (a) follows from \cite[Th.~4.67]{Joyc6} in the m-Kuranishi space case, and part (b) from \cite[Th.~4.16]{Joyc2}, in each case with topological space $X$, and $\vdim\bX=n$, and $\bX$ canonical up to equivalence in $\mKur,\dMan$. Part (c) then follows from (b) and Borisov \cite{Bori}, and part (d) from (c) and Borisov and No\"el \cite{BoNo}. The 1-1 correspondences of orientations can be proved by comparing Definition \ref{vf2def13} with \S\ref{vf262}. The last part also follows from \cite[Th.~4.16]{Joyc2}.
\end{proof}

\subsubsection{Bordism for derived manifolds}
\label{vf264}

We now discuss bordism, following \cite[\S 4.10]{Joyc2}, \cite[\S 15]{Joyc3} and~\cite[\S 13]{Joyc4}.

\begin{dfn} Let $Y$ be a manifold, and $k\in\N$. Consider pairs $(X,f)$, where $X$ is a compact, oriented manifold  with $\dim X=k$, and $f:X\ra Y$ is a smooth map. Define an equivalence relation $\sim$ on such pairs by $(X,f)\sim(X',f')$ if there exists a compact, oriented $(k+1)$-manifold with boundary $W$, a smooth map $e:W\ra Y$, and a diffeomorphism of oriented manifolds $j:-X\amalg X'\ra\pd W$, such that $f\amalg f'=e\ci i_W\ci j$, where $-X$ is $X$ with the opposite orientation, and $i_W:\pd W\hookra W$ is the inclusion map.

Write $[X,f]$ for the $\sim$-equivalence class ({\it bordism
class\/}) of a pair $(X,f)$. Define the {\it bordism group\/} $B_k(Y)$ of $Y$ to be the set of all such bordism classes $[X,f]$ with $\dim X=k$. It is an abelian group, with zero $0_Y=[\es,\es]$, addition $[X,f]+[X',f']=[X\amalg X',f\amalg f']$, and inverses~$-[X,f]=[-X,f]$.

Define $\Pi_\bo^\hom:B_k(Y)\ra H_k(Y;\Z)$ by $\Pi_\bo^\hom:[X,f]\mapsto f_*([X])$, where $H_*(-;\Z)$ is singular homology, and $[X]\in H_k(X;\Z)$ is the fundamental class.

When $Y$ is the point $*$, the maps $f:X\ra *$, $e:W\ra *$ are trivial, and we can omit them, and consider $B_k(*)$ to be the abelian group of bordism classes $[X]$ of compact, oriented, $k$-dimensional manifolds~$X$.
\label{vf2def16}
\end{dfn}

As in Conner \cite[\S I.5]{Conn}, bordism is a generalized homology theory. Results of Thom, Wall and others in \cite[\S I.2]{Conn} compute the bordism groups $B_k(*)$. We define d-manifold bordism by replacing manifolds $X$ in $[X,f]$ by d-manifolds~$\bX$:

\begin{dfn} Let $Y$ be a manifold, and $k\in\Z$. Consider pairs $(\bX,\bs f)$, where $\bX\in\dMan$ is a compact, oriented d-manifold with $\vdim\bX=k$, and $\bs f:\bX\ra Y$ is a 1-morphism in $\dMan$.

Define an equivalence relation $\sim$ between such pairs by
$(\bX,\bs f)\sim (\bX',\bs f')$ if there exists a compact, oriented
d-manifold with boundary $\bW$ with $\vdim\bW=k+1$, a 1-morphism
$\bs e:\bW\ra Y$ in $\dManb$, an equivalence of oriented
d-manifolds $\bs j:-\bX\amalg\bX'\ra\pd\bW$, and a 2-morphism $\eta:\bs f\amalg\bs f'\Ra\bs e\ci\bs i_\bW\ci \bs j$, where $\bs i_\bW:\pd\bW\ra\bW$ is the natural 1-morphism.

Write $[\bX,\bs f]$ for the $\sim$-equivalence class ({\it d-bordism
class\/}) of a pair $(\bX,\bs f)$. Define the {\it d-bordism group} $dB_k(Y)$ of $Y$ to be the set of d-bordism classes $[\bX,\bs f]$ with $\vdim\bX=k$. As for $B_k(Y)$, it is an abelian group, with zero $0_Y=[\bs\es,\bs\es]$, addition $[\bX,\bs f]+[\bX',\bs f']=[\bX\amalg\bX',\bs f\amalg\bs f']$, and $-[\bX,\bs f]=[-\bX,\bs f]$. Define $\Pi_\bo^\dbo:B_k(Y)\ra dB_k(Y)$ for $k\ge 0$ by~$\Pi_\bo^\dbo:[X,f]\mapsto[X,f]$.

When $Y$ is a point $*$, we can omit $\bs f:\bX\ra *$, and consider $dB_k(*)$ to be the abelian group of d-bordism classes $[\bX]$ of compact, oriented d-manifolds~$\bX$.
\label{vf2def17}
\end{dfn}

In \cite[\S 13.2]{Joyc4} we show $B_*(Y)$ and $dB_*(Y)$ are
isomorphic. See \cite[Th.~2.6]{Spiv} for an analogous (unoriented) result for Spivak's derived manifolds.

\begin{thm} For any manifold\/ $Y,$ we have $dB_k(Y)=0$ for $k<0,$
and\/ $\Pi_\bo^\dbo:B_k(Y)\ra dB_k(Y)$ is an isomorphism for $k\ge
0$.
\label{vf2thm5}
\end{thm}

The main idea of the proof of Theorem \ref{vf2thm5} is that (compact, oriented) d-manifolds $\bX$ can be turned into (compact, oriented) manifolds $\ti X$ by a small perturbation. By Theorem \ref{vf2thm5}, we may define a projection $\Pi_\dbo^\hom:dB_k(Y)\ra H_k(Y;\Z)$ for $k\ge 0$ by $\Pi_\dbo^\hom=\Pi_\bo^\hom\ci(\Pi_\bo^\dbo)^{-1}$. We think of
$\Pi_\dbo^\hom$ as a {\it virtual class map}, and call $[\bX]_\virt=\Pi_\dbo^\hom([\bX,\bs f])$ the {\it virtual class}. Virtual classes are used in several areas of geometry to construct enumerative invariants using moduli spaces, for example in \cite[\S A1]{FOOO}, \cite[\S 6]{FuOn} for Fukaya--Oh--Ohta--Ono's Kuranishi spaces, and in Behrend and Fantechi \cite{BeFa} in algebraic geometry.

\subsubsection{Virtual classes for derived manifolds in homology}
\label{vf265}

If $\bX$ is a compact, oriented derived manifold of dimension $k\in\Z$ we can also define a virtual class $[\bX]_\virt$ in the homology $H_k(X;\Z)$ of the underlying topological space $X$, for a suitable homology theory. By \cite[Cor.~4.30]{Joyc2} or \cite[Cor.~4.31]{Joyc3} or \cite[Th.~4.29]{Joyc4}, we can choose an embedding $\bs f:\bX\hookra\R^n$ for $n\gg 0$. If $Y$ is an open neighbourhood of $f(X)$ in $\R^n$ then \S\ref{vf264} defines $\Pi_\dbo^\hom([\bX,\bs f])$ in $H_k(Y;\Z)$. We also have a pushforward map $f_*:H_k(X;\Z)\ra H_k(Y;\Z)$. 

If $X$ is a Euclidean Neighbourhood Retract (ENR), we can choose $Y$ so that it retracts onto $f(X)$, and then $f_*:H_k(X;\Z)\ra H_k(Y;\Z)$ is an isomorphism, so we can define the virtual class $[\bX]_\virt=(f_*)^{-1}\ci\Pi_\dbo^\hom([\bX,\bs f])$ in ordinary homology $H_k(X;\Z)$. This $[\bX]_\virt$ is independent of the choices of~$\bs f,n,Y$.

General derived manifolds may not be ENRs. In this case we use a trick the authors learned from McDuff and Wehrheim \cite[\S 7.5]{McWe}. Choose a sequence $\R^n\supseteq Y_1\supseteq Y_2\supseteq \cdots$ of open neighbourhoods of $f(X)$ in $\R^n$ with $f(X)=\bigcap_{i\ge 1}Y_i$. Now {\it Steenrod homology\/} $H^{\rm St}_*(-;\Z)$ (see Milnor \cite{Miln}) is a homology theory with the nice properties that $H^{\rm St}_*(Y_i;\Z)\cong H_*(Y_i;\Z)$ as $Y_i$ is a manifold, and as $f(X)=\bigcap_{i\ge 1}Y_i$ we have an isomorphism with the inverse limit
\e
\smash{H^{\rm St}_k(f(X);\Z)\cong \underleftarrow{\lim}_{\,i\ge 1}\,H^{\rm St}_k(Y_i;\Z).}
\label{vf2eq11}
\e
\v Cech homology $\check H_*(-;\Q)$ over $\Q$ (the dual $\Q$-vector spaces to \v Cech cohomology $\check H^*(-;\Q)$) has the same limiting property. Then writing $\bs f_i=\bs f:\bX\ra Y_i$, so that $\Pi_\dbo^\hom([\bX,\bs f_i])\in H_k(Y_i;\Z)\cong H^{\rm St}_k(Y_i;\Z)$, using \eq{vf2eq11} we may form $\underleftarrow{\lim}_{\,i\ge 1}\Pi_\dbo^\hom([\bX,\bs f_i])$ in $H^{\rm St}_k(f(X);\Z)$, so that
\begin{equation*}
\smash{[\bX]_\virt:=(f_*)^{-1}\bigl[\underleftarrow{\lim}_{\,i\ge 1}\Pi_\dbo^\hom([\bX,\bs f_i])\bigr]}
\end{equation*}
is a virtual class in $H^{\rm St}_k(X;\Z)$, or similarly in $\check H_k(X;\Q)$. Here $[\bX]_\virt$ is independent of the choices of~$\bs f,n,Y_i$.

For the examples in this paper, $X$ is the complex analytic topological space of a proper $\C$-scheme, and therefore an ENR. Then $H^{\rm St}_k(X;\Z)\cong H_k(X;\Z)$ and $\check H_k(X;\Q)\cong H_k(X;\Q)$, and the virtual class lives in ordinary homology.

\section{The main results}
\label{vf3}

We now give our main results. We begin in \S\ref{vf31} with a general existence result for a special kind of atlas for $\bs\pi:\bX\ra Z$, where $\bX$ is a separated derived $\C$-scheme and $Z$ a smooth affine classical $\C$-scheme, an atlas in which the charts are spectra of standard form cdgas, the coordinate changes are quasi-free, and composition of coordinate changes is strictly associative.

Sections \ref{vf32}--\ref{vf35} build up to our primary goal, Theorems \ref{vf3thm3} and \ref{vf3thm4} in \S\ref{vf35}, which show that to a separated, $-2$-shifted symplectic derived $\C$-scheme $(\bX,\om_\bX^*)$ with $\vdim_\C\bX=n$ and complex analytic topological space $X_\an$, we can build a Kuranishi atlas $\cK$ on $X_\an$, and so construct a derived manifold $\bX_\dm$ with topological space $X_\an$, with $\vdim_\R\bX_\dm=n$. In \S\ref{vf36} we show that orientations on $(\bX,\om_\bX^*)$ and on $(X_\an,\cK)$ and on $\bX_\dm$ correspond, and prove that for $(\bX,\om_\bX^*)$ proper and oriented, the bordism class $[\bX_\dm]\in dB_n(*)$ is a `virtual cycle' independent of choices. 

Section \ref{vf37} extends \S\ref{vf32}--\S\ref{vf36} to families $(\bs\pi:\bX\ra Z,[\om_{\bX/Z}])$ over a connected base $\C$-scheme $Z$, and shows that the bordism class $[\bX_\dm^z]\in dB_n(*)$ associated to a fibre $\bs\pi^{-1}(z)$ is independent of $z\in Z_\an$. Finally, \S\ref{vf38}--\S\ref{vf39} discuss applying our results to define Donaldson--Thomas style invariants `counting' coherent sheaves on Calabi--Yau 4-folds, and motivation from gauge theory.

\subsection{Zariski homotopy atlases on derived schemes}
\label{vf31}

Derived schemes and stacks, discussed in \S\ref{vf22}, are very abstract objects, and difficult to do computations with. But standard form cdgas $A^\bu,B^\bu$ and quasi-free morphisms $\Phi:A^\bu\ra B^\bu$ in \S\ref{vf21} are easy to work with explicitly. Our first main result, proved in \S\ref{vf4}, constructs well-behaved homotopy atlases for a derived scheme $\bX$, built from standard form cdgas and quasi-free morphisms. 

\begin{thm} Let\/ $\bX$ be a separated derived\/ $\C$-scheme, $Z=\Spec B$ be a smooth classical affine $\C$-scheme for $B$ a smooth\/ $\C$-algebra of pure dimension, and\/ $\bs\pi:\bX\ra Z$ be a morphism. Suppose we are given data $\bigl\{(A_i^\bu,\bs\al_i,\be_i):i\in I\bigr\},$ where\/ $I$ is an indexing set and for each\/ $i\in I,$ $A_i^\bu\in\cdga_\C$ is a standard form cdga, and\/ $\bs\al_i:\bSpec A_i^\bu\hookra\bX$ is a Zariski open inclusion in $\dSch_\C,$ and\/ $\be_i:B\ra A_i^0$ is a smooth morphism of classical\/ $\C$-algebras such that the following diagram homotopy commutes in $\dSch_\C\!:$ 
\e
\begin{gathered}
\xymatrix@C=120pt@R=11pt{ *+[r]{\bSpec A_i^\bu} \ar[dr]_(0.5){\bSpec\be_i} \ar[r]_(0.6){\bs\al_i} & *+[l]{\bX} \ar[d]^{\bs\pi} \\ & *+[l]{\Spec B=Z,\!} }
\end{gathered}
\label{vf3eq1}
\e
regarding $\be_i$ as a morphism $B\ra A_i^\bu$. Then we can construct the following data:
\begin{itemize}
\setlength{\itemsep}{0pt}
\setlength{\parsep}{0pt}
\item[{\bf(i)}] For all finite subsets $\es\ne J\subseteq I,$ a standard form cdga $A_J^\bu\in\cdga_\C,$ a Zariski open inclusion $\bs\al_J:\bSpec A_J^\bu\hookra\bX,$ with image $\Im\bs\al_J=\bigcap_{i\in J}\Im\bs\al_i,$ and a smooth morphism of classical\/ $\C$-algebras $\be_J:B\ra A_J^0,$ such that the following diagram homotopy commutes in $\dSch_\C\!:$ 
\e
\begin{gathered}
\xymatrix@C=120pt@R=11pt{ *+[r]{\bSpec A_J^\bu} \ar[dr]_(0.5){\bSpec\be_J} \ar[r]_(0.6){\bs\al_J} & *+[l]{\bX} \ar[d]^{\bs\pi} \\ & *+[l]{\Spec B=Z,\!} }
\end{gathered}
\label{vf3eq2}
\e
and when $J=\{i\}$ for $i\in I$ we have $A_{\{i\}}^\bu=A_i^\bu,$ $\bs\al_{\{i\}}=\bs\al_i,$ and\/~$\be_{\{i\}}=\be_i$.
\item[{\bf(ii)}] For all inclusions of finite subsets $\es\ne K\subseteq J\subseteq I,$ a quasi-free morphism of standard form cdgas $\Phi_{JK}:A_K^\bu\ra A_J^\bu$ with\/ $\be_J=\Phi_{JK}\ci\be_K:B\ra A_J^0,$ such that the following diagram homotopy commutes in $\dSch_\C\!:$
\e
\begin{gathered}
\xymatrix@C=120pt@R=11pt{ *+[r]{\bSpec A_J^\bu} \ar[dr]_(0.6){\bs\al_J} \ar[r]_(0.6){\bSpec\Phi_{JK}} & *+[l]{\bSpec A_K^\bu} \ar[d]^{\bs\al_K} \\ & *+[l]{\bX,\!} }
\end{gathered}
\label{vf3eq3}
\e
and if\/ $\es\ne L\subseteq K\subseteq J\subseteq I$ then $\Phi_{JL}=\Phi_{JK}\ci\Phi_{KL}:A_L^\bu\ra A_J^\bu.$  
\end{itemize}

\label{vf3thm1}
\end{thm}

\subsection{Interpreting Zariski atlases using complex geometry}
\label{vf32}

Given a $-2$-shifted symplectic derived $\C$-scheme $(\bX,\om_\bX^*)$ satisfying conditions, we will construct a derived manifold structure $\bX_{\rm dm}$ on the complex analytic topological space $X_\an$ underlying $\bX$. To do this, we need a {\it change of language\/}: we have to pass from talking about derived schemes $\bX$, cdgas $A^\bu$, etc., to talking about smooth manifolds $V$, vector bundles $E\ra V$, smooth sections $s:V\ra E$, as $\bX_{\rm dm}$ will be built by gluing together such local Kuranishi models~$(V,E,s)$.

Therefore we now rewrite part of the output $A_J^\bu$, $\be_J:B\ra A_J^0$, $\Phi_{JK}:A_J^\bu\ra A_K^\bu$ of Theorem \ref{vf3thm1} in terms of complex manifolds $V$, holomorphic vector bundles $E\ra V$, and holomorphic sections $s:V\ra E$. In \S\ref{vf35} we will pass to certain real vector bundles $E^+=E/E^-$ to define~$\bX_{\rm dm}$.

First we interpret standard form cdgas $A^\bu\in\cdga_\C$ using holomorphic data. We discuss only data from degrees $0,-1,-2$ in $A^\bu$, as this is all we need, but one could also define vector bundles $G,H,\ldots$ over $V$ corresponding to $M^{-3},M^{-4},\ldots,$ and many vector bundle morphisms, satisfying equations. 

\begin{dfn} Let $A^\bu=\bigl(\cdots \ra A^{-2}\,{\buildrel\d\over\longra}\,A^{-1}\,{\buildrel\d\over\longra}\,A^0\bigr)$ be a standard form cdga over $\C$, as in \S\ref{vf21}. Then $A^0$ is a finitely generated smooth $\C$-algebra, so $V^\alg:=\Spec A^0$ is a smooth affine $\C$-scheme, assumed of pure dimension, as in \S\ref{vf21}. Now any $\C$-scheme $S$ has an underlying complex analytic space $S_\an$, which is a complex manifold if $S$ is smooth and of pure dimension.

Write $V$ for the complex manifold $(V^\alg)_\an$ associated to~$V^\alg=\Spec A^0$.

As $A^\bu$ is of standard form, the graded $\C$-algebra $A^*$ is freely generated over $A^0$ by a series of finitely generated free $A^0$-modules $M^{-1}\subseteq A^{-1}$, $M^{-2}\subseteq A^{-2},\ldots.$ Thus $A^{-1}\cong M^{-1}$, $A^{-2}\cong M^{-2}\op \La^2_{A^0}M^{-1}$, and so on, giving
\e
M^{-1}=A^{-1},\quad M^{-2}\cong A^{-2}/\La^2_{A^0}A^{-1},\;\ldots.
\label{vf3eq4}
\e
Hence, the $M^i$ are determined by $A^*$ as $A^0$-modules up to canonical isomorphism, although for $i\le -2$ the inclusions $M^i\hookra A^i$ involve an arbitrary choice.

Now finitely generated free $A^0$-modules $M$ are those of the form $M\cong H^0(C^\alg)$ for $C^\alg\ra V^\alg=\Spec A^0$ a trivial algebraic vector bundle. Write $E^\alg\ra V^\alg$, $F^\alg\ra V^\alg$ for the trivial algebraic vector bundles (unique up to canonical isomorphism) with $M^{-1}\cong H^0((E^\alg)^*)$, $M^{-2}\cong H^0((F^\alg)^*)$. That is, we set $E^\alg=\Spec\mathop{\rm Sym}_{A^0}^*(M^{-1})$, and so on. Write $E\ra V$, $F\ra V$ for the holomorphic vector bundles corresponding to~$E^\alg,F^\alg$.

We now have isomorphisms
\e
\begin{split}
A^0&\cong H^0(\O_{V^\alg}),\qquad A^{-1}\cong H^0((E^\alg)^*),\\ 
A^{-2}&\cong H^0((F^\alg)^*)\op H^0(\La^2(E^\alg)^*).
\end{split}
\label{vf3eq5}
\e
Thus $\d:A^{-1}\ra A^0$ is identified with an $A^0$-module morphism $H^0((E^\alg)^*)\ra H^0(\O_{V^\alg})$, that is, a morphism $(E^\alg)^*\ra \O_{V^\alg}$ of algebraic vector bundles, which is dual to a morphism $\O_{V^\alg}\cong \O_{V^\alg}^*\ra E^\alg$, i.e.\ a section $s^\alg\in H^0(E^\alg)$ of $E^\alg$. Write $s\in H^0(E)$ for the corresponding holomorphic section.

Similarly, write $t^\alg:E^\alg\ra F^\alg$ for the algebraic vector bundle morphism dual to the component of $\d:A^{-2}\ra A^{-1}$ mapping $H^0((F^\alg)^*)\ra H^0((E^\alg)^*)$ under \eq{vf3eq5}, and write $t:E\ra F$ for the corresponding morphism of holomorphic vector bundles. Then $\d\ci\d=0$ implies that~$t\ci s=0:\O_{V}\ra F$.

We should also consider how this data $E,F,s,t$ depends on the choice of inclusion $M^{-2}\hookra A^{-2}$. Here $E,F$ are independent of choices up to canonical isomorphism, and $s$ is independent of choices. Changing the inclusion $M^{-2}\hookra A^{-2}$ is equivalent to choosing an algebraic vector bundle morphism $\ga^\alg:\La^2E^\alg\ra F^\alg$ and identifying $M^{-2}$ with the image of $\id\op(\ga^\alg)^*:H^0((F^\alg)^*)\hookra H^0((F^\alg)^*)\op H^0(\La^2(E^\alg)^*)$. Writing $\ga:\La^2E\ra F$ for the corresponding holomorphic morphism, this changes $t$ to $\ti t$, where
\e
\ti t=t+\ga\ci( - \w s).
\label{vf3eq6}
\e
Notice that $t\vert_v:E\vert_v\ra F\vert_v$ is independent of choices at $v\in V$ with~$s(v)=0$.

Next suppose $\bX$ is a derived $\C$-scheme and $\bs\al:\bSpec A^\bu\hookra\bX$ a Zariski open inclusion. Write $X=t_0(\bX)$ for the classical $\C$-scheme, and $X_\an$ for the set of $\C$-points of $X$ equipped with the complex analytic topology. (One can give $X_\an$ the structure of a complex analytic space, but we will not use this.) Then $t_0(\bSpec A^\bu)$ is the $\C$-subscheme $(s^\alg)^{-1}(0)\subseteq V^\alg$, so $\al=t_0(\bs\al)$ is a Zariski open inclusion $(s^\alg)^{-1}(0)\hookra X$. Write $\psi:s^{-1}(0)\hookra X_\an$ for the corresponding map of $\C$-points. Then $\psi$ is a homeomorphism with an open set $R=\Im\psi\subseteq X_\an$. Note that $(V,E,s,\psi)$ is a Kuranishi neighbourhood on $X_\an$, in the sense of~\S\ref{vf25}.

As we explained in \S\ref{vf21}--\S\ref{vf22}, if $A^\bu$ is a standard form cdga then it is easy to compute the cotangent complex $\bL_{A^\bu}\simeq\Om^1_{A^\bu}$, and this also can be identified with the cotangent complex $\bL_{\bSpec A^\bu}$ of the derived scheme $\bSpec A^\bu$. Let $v\in s^{-1}(0)\subseteq V$ with $\psi(v)=x\in X_\an$. Then $v$ is a $\C$-point of $\bSpec A^\bu$ and $x$ a $\C$-point of $\bX$ with $\bs\al(v)=x$, so $\bL_{\bs\al}\vert_v:\bL_\bX\vert_x\ra\bL_{\bSpec A^\bu}\vert_v$ is a quasi-isomorphism, and induces an isomorphism on cohomology. One can show that $\bL_{\bSpec A^\bu}\vert_v$ is represented by the complex of $\C$-vector spaces
\e
\smash{\xymatrix@C=30pt{ \cdots \ar[r] & F\vert^*_v \ar[r]^{t\vert_v^*} & E\vert_v^* \ar[r]^{\d s\vert^*_v} & T_v^*V \ar[r] & 0,}}
\label{vf3eq7}
\e
with $T_v^*V$ in degree 0. Dualizing to tangent complexes and taking cohomology, we get canonical isomorphisms
\ea
H^0\bigl(\bT_{\bs\al}\vert_v\bigr)&:\Ker\bigl(\d s\vert_v:T_vV\ra E\vert_v\bigr)\longra H^0\bigl(\bT_\bX\vert_x\bigr),
\label{vf3eq8}\\
H^1\bigl(\bT_{\bs\al}\vert_v\bigr)&:\frac{\ts \Ker \bigl(t\vert_v:E\vert_v\ra F\vert_v\bigr)}{\ts \Im \bigl(\d s\vert_v:T_vV\ra E\vert_v\bigr)}\longra H^1\bigl(\bT_\bX\vert_x\bigr).
\label{vf3eq9}
\ea
 
Now suppose that $Z=\Spec B$ is a smooth classical affine $\C$-scheme of pure dimension, $\bs\pi:\bX\ra Z$ is a morphism, and $\be:B\ra A^0$ is a smooth morphism of $\C$-algebras, such that as for \eq{vf3eq1}--\eq{vf3eq2} the following homotopy commutes
\e
\begin{gathered}
\xymatrix@C=120pt@R=11pt{ *+[r]{\bSpec A^\bu} \ar[dr]_(0.5){\bSpec\be} \ar[r]_(0.6){\bs\al} & *+[l]{\bX} \ar[d]^{\bs\pi} \\ & *+[l]{\Spec B=Z.\!} }
\end{gathered}
\label{vf3eq10}
\e
Then $Z_\an$ is a complex manifold, and $\tau^\alg:=\Spec\be:V^\alg\ra Z$ is a smooth morphism of $\C$-schemes, and $\tau:=(\tau^\alg)_\an:V\ra Z_\an$ is a holomorphic submersion of complex manifolds. We can form the relative cotangent complexes $\bL_{\bX/Z},\bL_{\bSpec A^\bu/Z}$ and dual relative tangent complexes $\bT_{\bX/Z},\bT_{\bSpec A^\bu/Z}$, and \eq{vf3eq10} gives morphisms $\bL_{\bs\al}:\bL_{\bX/Z}\ra\bL_{\bSpec A^\bu/Z}$, $\bT_{\bs\al}:\bT_{\bSpec A^\bu/Z}\ra\bT_{\bX/Z}$. 

Write $T(V/Z_\an)=\Ker\bigl(\d\tau:TV\ra \tau^*(TZ_\an)\bigr)$ for the {\it relative tangent bundle\/} of $V/Z_\an$. It is a holomorphic vector subbundle of $TV$ of rank $\dim V-\dim Z$, as $\tau$ is a holomorphic submersion. Let $v\in s^{-1}(0)\subseteq V$ with $\psi(v)=x\in X_\an$ and $\tau(v)=\pi(x)=z\in Z_\an$. Then as in \eq{vf3eq7}, $\bL_{\bSpec A^\bu/Z}\vert_v$ is represented by the complex of $\C$-vector spaces
\begin{equation*}
\smash{\xymatrix@C=35pt{ \cdots \ar[r] & F\vert^*_v \ar[r]^{t\vert_v^*} & E\vert_v^* \ar[r]^(0.4){\d s\vert^*_v} & T_v^*(V/Z_\an) \ar[r] & 0,}}
\end{equation*}
with $T_v^*(V/Z_\an)$ in degree 0. As for \eq{vf3eq8}--\eq{vf3eq9} we get canonical isomorphisms
\ea
H^0\bigl(\bT_{\bs\al}\vert_v\bigr)&:\Ker\bigl(\d s\vert_v:T_v(V/Z_\an)\ra E\vert_v\bigr)\longra H^0\bigl(\bT_{\bX/Z}\vert_x\bigr),
\label{vf3eq11}\\
H^1\bigl(\bT_{\bs\al}\vert_v\bigr)&:\frac{\ts \Ker \bigl(t\vert_v:E\vert_v\ra F\vert_v\bigr)}{\ts \Im \bigl(\d s\vert_v:T_v(V/Z_\an)\ra E\vert_v\bigr)}\longra H^1\bigl(\bT_{\bX/Z}\vert_x\bigr).
\label{vf3eq12}
\ea

\label{vf3def1}
\end{dfn}

\begin{ex} Suppose $(A^\bu,\om_{A^\bu})$ is in $-2$-Darboux form, in the sense of Definition \ref{vf2def6}, with coordinates $x_1,\ldots,x_m$, $y_1,\ldots,y_n$, $z_1,\ldots,z_m$, and 2-form $\om_{A^\bu}$ in \eq{vf2eq1}, depending on invertible functions $q_1,\ldots,q_n\in A^0$. 

Let $V,E,F,s,t$ be as in Definition \ref{vf3def1}. Then $V$ is a smooth $\C$-scheme of dimension $m$, with \'etale coordinates $(x_1,\ldots,x_m)$, so that $TV$ is a trivial vector bundle with basis of sections $\frac{\pd}{\pd x_1},\ldots,\frac{\pd}{\pd x_m}$. Also $E$ is a trivial vector bundle of rank $n$, with basis $e_1:=\frac{\pd}{\pd y_1},\ldots,e_n:=\frac{\pd}{\pd y_n}$, and $F$ is trivial of rank $m$, with basis $\frac{\pd}{\pd z_1},\ldots,\frac{\pd}{\pd z_m}$. Using the first line of $\om_{A^\bu}$ in \eq{vf2eq1}, it is natural to identify $F\cong T^*V$ by identifying $\frac{\pd}{\pd z_i}\cong \dd x_i$ for~$i=1,\ldots,m$.

The natural section $s\in H^0(E)$ is $s=s_1e_1+\cdots+s_ne_n$. Write $\ep^1,\ldots,\ep^n$ for the basis of sections of $E^*$ dual to $e_1,\ldots,e_n$, so that $\ep^j\cong\dd y_j$. Motivated by the second line of $\om_{A^\bu}$ in \eq{vf2eq1}, define $Q=q_1\ep^1\ot\ep^1+\cdots+q_n\ep^n\ot\ep^n$ in $H^0(S^2E^*)$. Then $Q$ is a natural nondegenerate quadratic form on the fibres of $E$, and \eq{vf2eq2} implies that $Q(s,s)=0$.

Identifying $F=T^*V$, from \eq{vf2eq3} we see that $t:E\ra F$ is given by
\e
\smash{t(e_j)=\ts\sum\limits_{i=1}^m\bigl(2q_j\frac{\pd s_j}{\pd x_i}
+s_j\,\frac{\pd q_j}{\pd x_i}\bigr)\dd x_i=2q_j\,\dd s_j+s_j\,\dd q_j,}
\label{vf3eq13}
\e
for $j=1,\ldots,n$. Then $t\ci s=0$ follows from applying $\dd$ to~$Q(s,s)=0$.

What will matter later is that we have a complex manifold $V$, a holomorphic vector bundle $E\ra V$, a section $s\in H^0(E)$, and a nondegenerate holomorphic quadratic form $Q\in H^0(S^2E^*)$ with $Q(s,s)=0$, such that the classical complex analytic topological space $\bigl(\Spec H^0(A^\bu)\bigr){}_\an$ is~$s^{-1}(0)\subseteq V$.
\label{vf3ex1}
\end{ex}

Next we interpret quasi-free morphisms of standard form cdgas $\Phi_{JK}:A_K^\bu\ra A_J^\bu,$ as in Theorem \ref{vf3thm1}(a)(ii), in terms of complex geometry.

\begin{dfn} Let $\Phi_{JK}:A_K^\bu\ra A_J^\bu$ be a quasi-free morphism of standard form cdgas over $\C$, as in \S\ref{vf21}. Let $V_J^\alg,E_J^\alg,F_J^\alg,s_J^\alg,t_J^\alg,V_J,E_J,F_J, s_J,t_J$ be as in Definition \ref{vf3def1} for $A_J^\bu$, and $V_K^\alg,E_K^\alg,\ldots,t_K$ for~$A_K^\bu$.

Then $\phi_{JK}^\alg:=\Spec\Phi_{JK}^0:V_J^\alg=\Spec A_J^0\ra V_K^\alg=\Spec A_K^0$ is a $\C$-scheme morphism. Write $\phi_{JK}:V_J\ra V_K$ for the corresponding holomorphic map. The quasi-free condition on $\Phi_{JK}$ implies that $\d\phi_{JK}^\alg:(\phi_{JK}^\alg)^*(T^*V_K^\alg)\ra T^*V_J^\alg$ is injective, so $\d\phi_{JK}:\phi_{JK}^*(T^*V_K)\ra T^*V_J$ is injective, that is, $\phi_{JK}:V_J\ra V_K$ is a submersion of complex manifolds.

Now $\Phi_{JK}^{-1}:A_K^{-1}\ra A_J^{-1}$ induces an $A_J^0$-linear map $(\Phi_{JK}^{-1})_*:A_K^{-1}\ot_{A_K^0}A_J^0\ra A_J^{-1}$, which under \eq{vf3eq5} corresponds to an algebraic vector bundle morphism $(\phi_{JK}^\alg)^*((E_K^\alg)^*)\ra (E_J^\alg)^*$. Write $\chi_{JK}^\alg:E_J^\alg\ra (\phi_{JK}^\alg)^*(E_K^\alg)$ for the dual morphism, and $\chi_{JK}:E_J\ra \phi_{JK}^*(E_K)$ for the corresponding morphism of holomorphic vector bundles. It is surjective, as $\Phi_{JK}$ is quasi-free. Then $\d\ci\Phi_{JK}^{-1}=\Phi_{JK}^0\ci\d$ implies that
\e
\smash{\chi_{JK}(s_J)=\phi_{JK}^*(s_K)\in H^0(\phi_{JK}^*(E_K)).}
\label{vf3eq14}
\e

By \eq{vf3eq4} we have a natural composition of morphisms
\begin{equation*}
\smash{H^0((F_K^\alg)^*)\!\cong\! M_K^{-2}\!\cong\! A_K^{-2}/\La^2_{A_K^0}A_K^{-1}
\,{\buildrel(\Phi_{JK}^{-2})_*\over\longra}\, A_J^{-2}/\La^2_{A_J^0}A_J^{-1}\!\cong\! M_J^{-2}\!\cong\! H^0((F_J^\alg)^*).}
\end{equation*}
The induced $A_J^0$-linear map corresponds to a natural algebraic vector bundle morphism $(\phi_{JK}^\alg)^*((F_K^\alg)^*)\ra (F_J^\alg)^*$. Write $\xi_{JK}^\alg:F_J^\alg\ra (\phi_{JK}^\alg)^*(F_K^\alg)$ for the dual morphism, and $\xi_{JK}:F_J\ra\phi_{JK}^*(F_K)$ for the corresponding morphism of holomorphic vector bundles. It is surjective, as $\Phi_{JK}$ is quasi-free. 

These $\xi_{JK}^\alg,\xi_{JK}$ are independent of choices, as they depend on the canonical isomorphism $M^{-2}\cong A^{-2}/\La^2_{A^0}A^{-1}$ rather than on the non-canonical inclusion $M^{-2}\hookra A^{-2}$ in Definition \ref{vf3def1}. However, $\Phi_{JK}^{-2}$ need not map $M_K^{-2}\subseteq A_K^{-2}$ to $M_J^{-2}\subseteq A_J^{-2}$, and so under the isomorphisms \eq{vf3eq5} need not map $H^0((F_K^\alg)^*)\ra H^0((F_J^\alg)^*)$. Write $\de_{JK}^\alg:\La^2E_J^\alg\ra(\phi_{JK}^\alg)^*(F_K^\alg)$ for the algebraic vector bundle morphism dual to the component of $\Phi_{JK}^{-2}$ mapping $H^0((F_K^\alg)^*)\ra H^0(\La^2(E_J^\alg)^*)$, and $\de_{JK}:\La^2E_J\ra\phi_{JK}^*(F_K)$ for the corresponding morphism of vector bundles. Then $\d\ci\Phi_{JK}^{-2}=\Phi_{JK}^{-1}\ci\d$ implies that
\e
\smash{\xi_{JK}\ci t_J+\de_{JK}\ci(-\w s_J)=\phi_{JK}^*(t_K)\ci\chi_{JK}
:E_J\longra \phi_{JK}^*(F_K).}
\label{vf3eq15}
\e

Thus $\chi_{JK},\xi_{JK}$ do not strictly commute with $t_J,t_K$, which is not surprising, since $t_J,t_K$ depend on arbitrary choices as in \eq{vf3eq6}. But notice that $\xi_{JK}\vert_v\ci t_J\vert_v=t_K\vert_{\phi_{JK}(v)}\ci\chi_{JK}\vert_v$ at $v\in V_J$ with~$s_J(v)=0$.

Next suppose we are given Zariski open inclusions $\bs\al_J:\bSpec A_J^\bu\hookra\bX$, $\bs\al_K:\bSpec A_K^\bu\hookra\bX$ into a derived $\C$-scheme $\bX$, such that \eq{vf3eq3} homotopy commutes, and let $\psi_J:s_J^{-1}(0)\hookra X_\an$, $\psi_K:s_K^{-1}(0)\hookra X_\an$ be as in Definition \ref{vf3def1}. As the classical truncation of \eq{vf3eq3} commutes, we see that
\e
\smash{\psi_J=\psi_K\ci\phi_{JK}\vert_{s_J^{-1}(0)}:s_J^{-1}(0)\longra X_\an.}
\label{vf3eq16}
\e

Suppose $v_J\in s_J^{-1}(0)\subseteq V_J$ with $\phi_{JK}(v_J)=v_K\in s_K^{-1}(0)\subseteq V_K$ and $\psi_J(v_J)=\psi_K(v_K)=x\in X_\an$. As \eq{vf3eq3} homotopy commutes, the corresponding morphisms of tangent complexes $\bT_{\bSpec A_J^\bu},\bT_{\bSpec A_K^\bu},\bT_\bX$ commute up to homotopy, so restricting to $v_J,v_K,x$ and taking homology gives strictly commuting diagrams. Thus using \eq{vf3eq8}--\eq{vf3eq9}, we see that the following diagrams commute:
\ea
\begin{gathered}
\xymatrix@C=175pt@R=12pt{ *+[r]{\Ker\bigl(\d s_J\vert_{v_J}\!:\!T_{v_J}V_J\!\ra\!E_J\vert_{v_J}\bigr)} \ar[d]^{(\d\phi_{JK}\vert_{v_J})\vert_{\Ker(\cdots)}} \ar@/^1pc/[dr]^(0.7){H^0(\bT_{\bs\al_J}\vert_{v_J})} \\
*+[r]{\Ker\bigl(\d s_K\vert_{v_K}\!:\!T_{v_K}V_K\!\ra \!E_K\vert_{v_K}\bigr)} \ar[r]^(0.6){H^0(\bT_{\bs\al_K}\vert_{v_K})} 
& *+[l]{H^0\bigl(\bT_\bX\vert_x\bigr),}  }
\end{gathered}
\label{vf3eq17}\\
\begin{gathered}
\xymatrix@C=170pt@R=12pt{ *+[r]{\frac{\ts \Ker \bigl(t_J\vert_{v_J}\!:\!E_J\vert_{v_J}\!\ra\! F_J\vert_{v_J}\bigr)}{\ts \Im \bigl(\d s_J\vert_{v_J}\!:\!T_{v_J}V_J\!\ra\! E_J\vert_{v_K}\bigr)}} \ar[d]^{(\chi_{JK}\vert_{v_J})_*} \ar@/^1pc/[dr]^(0.7){H^1(\bT_{\bs\al_J}\vert_{v_J})} \\
*+[r]{\frac{\ts \Ker \bigl(t_K\vert_{v_K}\!:\!E_K\vert_{v_K}\!\ra\! F_K\vert_{v_K}\bigr)}{\ts \Im \bigl(\d s_K\vert_{v_K}\!:\!T_{v_K}V_K\!\ra\! E_K\vert_{v_K}\bigr)}} \ar[r]^(0.6){H^1(\bT_{\bs\al_K}\vert_{v_K})} 
& *+[l]{H^1\bigl(\bT_\bX\vert_x\bigr).}  }
\end{gathered}
\label{vf3eq18}
\ea

Now suppose that $Z=\Spec B$ is a smooth classical affine $\C$-scheme of pure dimension, $\bs\pi:\bX\ra Z$ is a morphism, and $\be_J:B\ra A_J^0$, $\be_K:B\ra A_K^0$ are smooth morphisms of $\C$-algebras, such that \eq{vf3eq2} homotopy commutes for $J,K$, and $\be_J=\Phi_{JK}\ci\be_K$. As in Definition \ref{vf3def1} we have holomorphic submersions $\tau_J:V_J\ra Z_\an$, $\tau_K:V_K\ra Z_\an$, with $\tau_J=\tau_K\ci\phi_{JK}:V_J\ra Z_\an$ as $\be_J=\Phi_{JK}\ci\be_K$. Let $v_J\in s_J^{-1}(0)\subseteq V_J$ with $\phi_{JK}(v_J)=v_K\in s_K^{-1}(0)\subseteq V_K$, and $\psi_J(v_J)=\psi_K(v_K)=x\in X_\an$, and $\tau_J(v_J)=\tau_K(v_K)=\pi(x)=z\in Z_\an$. Then using \eq{vf3eq11}--\eq{vf3eq12}, we see that the following diagrams commute:
\ea
\begin{gathered}
\xymatrix@C=165pt@R=12pt{ *+[r]{\Ker\bigl(\d s_J\vert_{v_J}\!:\!T_{v_J}(V_J/Z_\an)\!\ra\!E_J\vert_{v_J}\bigr)} \ar[d]^{(\d\phi_{JK}\vert_{v_J})\vert_{\Ker(\cdots)}} \ar@/^1pc/[dr]^(0.7){H^0(\bT_{\bs\al_J}\vert_{v_J})} \\
*+[r]{\Ker\bigl(\d s_K\vert_{v_K}\!:\!T_{v_K}(V_K/Z_\an)\!\ra \!E_K\vert_{v_K}\bigr)} \ar[r]^(0.68){H^0(\bT_{\bs\al_K}\vert_{v_K})} 
& *+[l]{H^0\bigl(\bT_{\bX/Z}\vert_x\bigr),}  }
\end{gathered}
\label{vf3eq19}\\
\begin{gathered}
\xymatrix@C=160pt@R=13pt{ *+[r]{\frac{\ts \Ker \bigl(t_J\vert_{v_J}\!:\!E_J\vert_{v_J}\!\ra\! F_J\vert_{v_J}\bigr)}{\ts \Im \bigl(\d s_J\vert_{v_J}\!:\!T_{v_J}(V_J/Z_\an)\!\ra\! E_J\vert_{v_K}\bigr)}} \ar[d]^{(\chi_{JK}\vert_{v_J})_*} \ar@/^1pc/[dr]^(0.7){H^1(\bT_{\bs\al_J}\vert_{v_J})} \\
*+[r]{\frac{\ts \Ker \bigl(t_K\vert_{v_K}\!:\!E_K\vert_{v_K}\!\ra\! F_K\vert_{v_K}\bigr)}{\ts \Im \bigl(\d s_K\vert_{v_K}\!:\!T_{v_K}(V_K/Z_\an)\!\ra\! E_K\vert_{v_K}\bigr)}} \ar[r]^(0.65){H^1(\bT_{\bs\al_K}\vert_{v_K})} 
& *+[l]{H^1\bigl(\bT_{\bX/Z}\vert_x\bigr).}  }
\end{gathered}
\label{vf3eq20}
\ea

\label{vf3def2}
\end{dfn}

Applying Definitions \ref{vf3def1} and \ref{vf3def2} to the conclusions of Theorem \ref{vf3thm1} yields:

\begin{cor} In the situation of Theorem\/ {\rm\ref{vf3thm1},} write $X_\an$ for the set of\/ $\C$-points of\/ $X=t_0(\bX),$ regarded as a topological space with the complex analytic topology. Then we obtain the following data in complex geometry:
\begin{itemize}
\setlength{\itemsep}{0pt}
\setlength{\parsep}{0pt}
\item[{\bf(i)}] For all finite subsets $\es\ne J\subseteq I,$ a complex manifold\/ $V_J,$ a holomorphic submersion $\tau_J:V_J\ra Z_\an,$ holomorphic vector bundles $E_J,F_J\ra V_J,$ a holomorphic section $s_J:V_J\ra E_J,$ and a homeomorphism $\psi_J:s_J^{-1}(0)\ra R_J\subseteq X_\an,$ where $R_J\subseteq X_\an$ is open, with\/ $\pi\ci\psi_J=\tau_J\vert_{s_J^{-1}(0)}:s_J^{-1}(0)\ra Z_\an$. These image subsets satisfy\/~$R_J=\bigcap_{i\in J}R_{\{i\}}$.
 
By making an additional arbitrary choice we also obtain a morphism of holomorphic vector bundles $t_J:E_J\ra F_J,$ with\/ $t_J\ci s_J=0$. Different choices $t_J,\ti t_J$ are related by \eq{vf3eq6}. The restrictions $t_J\vert_{v_J}:E_J\vert_{v_J}\ra F_J\vert_{v_J}$ for $v_J\in s_J^{-1}(0)$ are independent of choices. For each\/ $v_J\in s_J^{-1}(0)$ with\/ $\psi_J(v_J)=x\in X_\an,$ there are canonical isomorphisms \eq{vf3eq8}--\eq{vf3eq9} writing $H^i\bigl(\bT_\bX\vert_x\bigr)$ for $i=0,1$ and\/ \eq{vf3eq11}--\eq{vf3eq12} writing $H^i\bigl(\bT_{\bX/Z}\vert_x\bigr)$ for $i=0,1$ in terms of\/ $V_J,E_J,F_J,s_J,t_J,\tau_J$ at\/~$v_J$.
\item[{\bf(ii)}] For all inclusions of finite subsets $\es\ne K\subseteq J\subseteq I,$ a holomorphic submersion $\phi_{JK}:V_J\ra V_K,$ and surjective morphisms of holomorphic vector bundles $\chi_{JK}:E_J\ra \phi_{JK}^*(E_K)$ and\/ $\xi_{JK}:F_J\ra\phi_{JK}^*(F_K)$. These satisfy $\tau_J=\tau_K\ci\phi_{JK}:V_J\ra Z_\an,$ and\/ $\chi_{JK}(s_J)=\phi_{JK}^*(s_K),$ and\/~$\psi_J=\psi_K\ci\phi_{JK}\vert_{s_J^{-1}(0)}:s_J^{-1}(0)\ra X_\an$.

If\/ $t_J,t_K$ are possible choices in {\bf(i)} then $\chi_{JK},\xi_{JK},t_J,t_K$ are related as in \eq{vf3eq15}. If\/ $v_J\in s_J^{-1}(0)$ with\/ $\phi_{JK}(v_J)=v_K\in s_K^{-1}(0),$ this implies that
\begin{equation*}
\smash{\xi_{JK}\vert_{v_J}\ci t_J\vert_{v_J}=t_K\vert_{v_K}\ci\chi_{JK}\vert_{v_J}:E_J\vert_{v_J}\longra F_K\vert_{v_K}.}
\end{equation*}

If\/ $v_J\in s_J^{-1}(0)\subseteq V_J$ with\/ $\phi_{JK}(v_J)=v_K\in s_K^{-1}(0)\subseteq V_K$ and\/ $\psi_J(v_J)=\psi_K(v_K)=x\in X_\an,$ then \eq{vf3eq17}--\eq{vf3eq20} commute.

If\/ $\es\ne L\subseteq K\subseteq J\subseteq I$ then $\phi_{JL}=\phi_{KL}\ci\phi_{JK},$ $\chi_{JL}=\phi_{JK}^*(\chi_{KL})\ci\chi_{JK},$ and\/~$\xi_{JL}=\phi_{JK}^*(\xi_{KL})\ci\xi_{JK}$.
\end{itemize}
\label{vf3cor1}
\end{cor}

\subsection{\texorpdfstring{Subbundles $E^-\subseteq E$ and Kuranishi neighbourhoods}{Subbundles E⁻⊆E and Kuranishi neighbourhoods}}
\label{vf33}

Throughout \S\ref{vf33}--\S\ref{vf36}, when we apply Theorem \ref{vf3thm1} we take $B=\C$, so that $Z$ is the point $*=\Spec\C$, and the data $\bs\pi,\be_i,\be_J,\tau_J$ is trivial, so we omit it.

Suppose $(\bX,\om_\bX^*)$ is a $-2$-shifted symplectic derived $\C$-scheme, $A^\bu$ a standard form cdga over $\C$, and $\bs\al:\bSpec A^\bu\ra\bX$ a Zariski open inclusion. Then Definition \ref{vf3def1} defines complex geometric data $V,E,F,s,t,\psi,R$, such that $(V,E,s,\psi)$ is a Kuranishi neighbourhood on the topological space $X_\an$ of~$\bX$.

However these are not the Kuranishi neighbourhoods we want: they depend only on $\bX$, not on $\om_\bX^*$, and in general two such neighbourhoods $(V_J,E_J,s_J,\psi_J)$ and $(V_K,E_K,s_K,\psi_K)$ are not compatible over their intersection $R_J\cap R_K$ in $X_\an$ (e.g.\ the virtual dimensions $\dim_\R V_J-\rank_\R E_J$ and $\dim_\R V_K-\rank_\R E_K$ may be different), so we cannot glue them to make $X_\an$ into a derived manifold.

The basic problem is that the rank of $E$ may be too large -- for instance, we can modify $A^\bu$ to replace $E,F,s,t$ by $\ti E=E\op G$, $\ti F=F\op G$, $\ti s=s\op 0$, $\ti t=t\op\id_G$ for some holomorphic vector bundle $G\ra V$. Our solution is to choose a real vector subbundle $E^-\subseteq E$ satisfying some conditions involving $\om_\bX^*$, and set $E^+=E/E^-$ to be the quotient bundle and $s^+=s+E^-$ in $C^\iy(E^+)$ to be the quotient section. The conditions on $E^-$ imply that $s^{-1}(0)=(s^+)^{-1}(0)$, so $(V,E^+,s^+,\psi^+)$ is also a Kuranishi neighbourhood on $X_\an$. Under good conditions we can make two such $(V_J,E_J^+,s_J^+,\psi_J^+),(V_K,E_K^+,s_K^+,\psi_K^+)$ compatible over $R_J\cap R_K$, and glue these local models to make $X_\an$ into a derived manifold.

We define the class of subbundles $E^-\subseteq E$ we are interested in:

\begin{dfn} Let $(\bX,\om_\bX^*)$ be a $-2$-shifted symplectic derived $\C$-scheme with $\vdim_\C\bX=n$, and suppose $A^\bu\in\cdga_\C$ is of standard form and $\bs\al:A^\bu\hookra\bX$ is a Zariski open inclusion. Define complex geometric data $V,E,F,s,t$ and $\psi:s^{-1}(0)\,{\buildrel\cong\over\longra}\,R\subseteq X_\an$ as in Definition \ref{vf3def1}, and suppose $R\ne\es$. Then for each $v\in s^{-1}(0)$ with $\psi(v)=x\in X_\an$, equation \eq{vf3eq9} gives an isomorphism from a vector space depending on $V,E,F,s,t$ at $v$ to~$H^1(\bT_\bX\vert_x)$.

Equation \eq{vf2eq6} defined a quadratic form $Q_x$ on $H^1\bigl(\bT_\bX\vert_x\bigr)$. Define
\e
\ti Q_v:\frac{\ts \Ker \bigl(t\vert_v:E\vert_v\ra F\vert_v\bigr)}{\ts \Im \bigl(\d s\vert_v:T_vV\ra E\vert_v\bigr)}\t\frac{\ts \Ker \bigl(t\vert_v:E\vert_v\ra F\vert_v\bigr)}{\ts \Im \bigl(\d s\vert_v:T_vV\ra E\vert_v\bigr)}\longra\C
\label{vf3eq21}
\e
to be the nondegenerate complex quadratic form identified with $Q_x$ in \eq{vf2eq6} by the isomorphism $H^1\bigl(\bT_{\bs\al}\vert_v\bigr)$ in~\eq{vf3eq9}.

Consider pairs $(U,E^-)$, where $U\subseteq V$ is open and $E^-$ is a real vector subbundle of $E\vert_U$. Given such $(U,E^-)$, we write $E^+=E\vert_U/E^-$ for the quotient vector bundle over $U$, and $s^+\in C^\iy(E^+)$ for the image of $s\vert_U$ under the projection $E\vert_U\ra E^+$, and $\psi^+:=\psi\vert_{s^{-1}(0)\cap U}:s^{-1}(0)\cap U\ra X_\an$. We say that $(U,E^-)$ {\it satisfies condition\/} $(*)$ if:
\begin{itemize}
\setlength{\itemsep}{0pt}
\setlength{\parsep}{0pt}
\item[$(*)$] For each $v\in s^{-1}(0)\cap U$, we have 
\ea
\Im\bigl(\d s\vert_v:T_vV\ra E\vert_v\bigr)\cap E^-\vert_v&=\{0\} &&\text{in $E\vert_v$,}
\label{vf3eq22}\\
t\vert_v\bigl(E^-\vert_v\bigr)&=t\vert_v\bigl(E\vert_v\bigr)&&\text{in $F\vert_v$,}
\label{vf3eq23}
\ea
and the natural real linear map
\e
\Pi_v:E^-\vert_v\cap \Ker \bigl(t\vert_v:E\vert_v\ra F\vert_v\bigr)\longra
\frac{\ts \Ker \bigl(t\vert_v:E\vert_v\ra F\vert_v\bigr)}{\ts \Im \bigl(\d s\vert_v:T_vV\ra E\vert_v\bigr)}\,,
\label{vf3eq24}
\e
which is injective by \eq{vf3eq22}, has image $\Im\Pi_v$ a real vector subspace of dimension exactly half the real dimension of $\Ker (t\vert_v)/\Im(\d s\vert_v)$, and the real quadratic form $\Re \ti Q_v$ on $\Ker (t\vert_v)/\Im(\d s\vert_v)$ from \eq{vf3eq21} restricts to a negative definite real quadratic form on~$\Im\Pi_v$.
\end{itemize}

We say $(U,E^-)$ {\it satisfies condition\/} $(\dag)$ if:
\begin{itemize}
\setlength{\itemsep}{0pt}
\setlength{\parsep}{0pt}
\item[$(\dag)$] $(U,E^-)$ satisfies condition $(*)$ and $s^{-1}(0)\cap U=(s^+)^{-1}(0)\subseteq U$.
\end{itemize}
Then $(U,E^+,s^+,\psi^+)$ is a Kuranishi neighbourhood on~$X_\an$.

Observe that if $v\in s^{-1}(0)\cap U$ with $\psi(v)=x\in X_\an$ then using \eq{vf3eq8}--\eq{vf3eq9} and \eq{vf3eq22}--\eq{vf3eq24} we find there is an exact sequence
\e
\smash{\xymatrix@C=15pt{ 0 \ar[r] & H^0\bigl(\bT_\bX\vert_x\bigr) \ar[r] & T_vU \ar[r] & E^+\vert_v \ar[r] & H^1\bigl(\bT_\bX\vert_x\bigr)/\Im\Pi_v \ar[r] & 0. }}
\label{vf3eq25}
\e
Hence
\ea
&\dim_\R U\!-\!\rank_\R E^+\!=\!\dim_\R H^0(\bT_\bX\vert_x)\!-\!\dim_\R H^1(\bT_\bX\vert_x)\!+\!\dim_\R\Im\Pi_v
\nonumber\\
&=2\dim_\C H^0(\bT_\bX\vert_x)-\dim_\C H^1(\bT_\bX\vert_x)
\label{vf3eq26}\\
&=\dim_\C H^0(\bT_\bX\vert_x)\!-\!\dim_\C H^1(\bT_\bX\vert_x)\!+\!
\dim_\C H^2(\bT_\bX\vert_x)\!=\!\vdim_\C\bX\!=\!n.
\nonumber 
\ea
Here in the second step we use $\dim_\R\Pi_v=\ha\dim_\R H^1(\bT_\bX\vert_x)$ by $(*)$ and \eq{vf3eq9}, in the third that $H^0(\bT_\bX\vert_x)\cong H^2(\bT_\bX\vert_x)^*$ as $(\bX,\om_\bX^*)$ is $-2$-shifted symplectic (or $-2$-shifted presymplectic will do), and in the fourth that $\bT_\bX$ is perfect in the interval $[0,2]$ as $(\bX,\om_\bX^*)$ is $-2$-shifted symplectic (or presymplectic).

Equation \eq{vf3eq26} says that the Kuranishi neighbourhood $(U,E^+,s^+,\psi^+)$ has real virtual dimension $\dim U-\rank E^+=n=\vdim_\C\bX=\ha\vdim_\R\bX$. Note that this is half the virtual dimension we might have expected, and the real virtual dimension can be odd, even though $\bX,V,E,s,\ldots$ are all complex.
\label{vf3def3}
\end{dfn}

Here are some important properties of such $U,E^-,E^+,s^+$, proved in~\S\ref{vf5}.

\begin{thm} In the situation of Definition\/ {\rm\ref{vf3def3},} with\/ $\bX,\om_\bX^*,A^\bu,\bs\al,V,E,F,\ab s,\ab t,\ab\psi$ fixed, we have:
\begin{itemize}
\item[{\bf(a)}] If the conditions in $(*)$ hold at some\/ $v\in s^{-1}(0)\cap U,$ then they also hold for all\/ $v'$ in an open neighbourhood of\/ $v$ in\/~$s^{-1}(0)\cap U$. 
\item[{\bf(b)}] Suppose $C\subseteq V$ is closed, and\/ $(U,E^-)$ satisfies condition $(*)$ with\/ $C\subseteq U\subseteq V$. (We allow $C=U=\es$.) Then there exists $(\ti U,\ti E^-)$ satisfying $(*)$ with $C\cup s^{-1}(0)\subseteq\ti U\subseteq V,$ and an open neighbourhood\/ $U'$ of\/ $C$ in $U\cap\ti U$ such that\/ $E^-\vert_{U'}=\ti E^-\vert_{U'}$.
\item[{\bf(c)}] If\/ $(U,E^-)$ satisfies $(*),$ the closed subsets $s^{-1}(0)\cap U$ and\/ $(s^+)^{-1}(0)$ in $U\subseteq V$ coincide in an open neighbourhood\/ $U'$ of\/ $s^{-1}(0)\cap U$ in $U$. Hence $(U',E^-\vert_{U'})$ satisfies condition $(\dag),$ and\/ $(U',E^+\vert_{U'},s^+\vert_{U'},\psi^+)$ is a Kuranishi neighbourhood on $X_\an$. Thus, we can make $(U,E^-)$ satisfying $(*)$ also satisfy $(\dag)$ by shrinking $U,$ without changing $R=\Im\psi$ in\/~$X_\an$.
\end{itemize}
\label{vf3thm2}
\end{thm}

The next example proves Theorem \ref{vf3thm2}(c) near $v\in s^{-1}(0)\cap U$ in a special case, when $(A^\bu,\om_{A^\bu})$ is in $-2$-Darboux form and minimal at $v$. The general case in \S\ref{vf53} is proved by reducing to Example~\ref{vf3ex2}.

\begin{ex} Suppose $(\bX,\om_\bX^*)$ is a $-2$-shifted symplectic derived $\C$-scheme, and $x\in X_\an$. Then Theorem \ref{vf2thm3} gives a pair $(A^\bu,\om_{A^\bu})$ in $-2$-Darboux form and a Zariski open inclusion $\bs\al:\bSpec A^\bu\hookra\bX$ which is minimal at $x\in\Im\bs\al$, with $\bs\al^*(\om_\bX^*)\simeq\om_{A^\bu}$ in~$\cA^{2,\cl}_\C(\bSpec A^\bu,-2)$.

Example \ref{vf3ex1} describes the data $V,E,F,s,t$ associated to $A^\bu$ in \S\ref{vf32}, and defines a nondegenerate quadratic form $Q\in H^0(S^2E^*)$ with $Q(s,s)=0$ using $\om_{A^\bu}$. As $x\in\Im\bs\al$ there is $v\in s^{-1}(0)\subseteq V$ with $\bs\al(v)=x$, and $(A^\bu,\bs\al)$ minimal at $x$ means that $\d s\vert_v=0$, so that $t\vert_v=0$ by \eq{vf3eq13}. Thus in \eq{vf3eq9} we have $\Ker(t\vert_v)/\Im(\d s\vert_v)=E\vert_v$, identified with $H^1(\bT_\bX\vert_x)$. Since $\bs\al^*(\om_\bX^*)\simeq\om_{A^\bu}$, the quadratic form $\ti Q_v$ on $\Ker(t\vert_v)/\Im(\d s\vert_v)=E\vert_v$ in \eq{vf3eq21} is~$Q\vert_v$. 

Given a pair $(U,E^-)$ as in Definition \ref{vf3def3} with $v\in U$, the map $\Pi_v$ in \eq{vf3eq24} is just the inclusion $E^-\vert_v\hookra E\vert_v$. So $(*)$ at $v$ says that $E^-\vert_v$ is a real vector subspace of $E\vert_v$ with $\dim_\R E^-\vert_v=\ha\dim_\R E\vert_v=\dim_\C E\vert_v$, such that $\Re Q\vert_v$ is negative definite on $E^-\vert_v$. 

As this is an open condition, there exists an open neighbourhood $U'$ of $v$ in $U$ such that $\Re Q\vert_{\smash{U'}}$ is negative definite on $E^-\vert_{\smash{U'}}$. Define a real vector subbundle $\ti E^+$ of $E\vert_{\smash{U'}}$ to be the orthogonal subbundle of $E^-\vert_{\smash{U'}}$ w.r.t.\ the nondegenerate real quadratic form $\Re Q\vert_{\smash{U'}}$. Then $E\vert_{\smash{U'}}=\ti E^+\op E^-\vert_{\smash{U'}}$, so we can write $s\vert_{\smash{U'}}=\ti s^+\op s^-$, for $\ti s^+\in C^\iy(\ti E^+)$ and $s^-\in C^\iy(E^-\vert_{\smash{U'}})$. The projection $E\vert_{\smash{U'}}\ra E^+\vert_{\smash{U'}}=E\vert_{\smash{U'}}/E^-\vert_{\smash{U'}}$ restricts to an isomorphism $\ti E^+\ra E^+\vert_{\smash{U'}}$, which maps~$\ti s^+\mapsto s^+\vert_{\smash{U'}}$.

Because $\Re Q$ is the real part of a complex form, it has the same number of positive as negative eigenvalues. Thus $\Re Q\vert_{\smash{U'}}$ is positive definite on $\ti E^+$. Now
\e
0\!=\!\Re Q(s,s)\vert_{\smash{U'}}\!=\!\Re Q(\ti s^+\!+\! s^-,\ti s^+\!+\! s^-)\!=\!\Re Q(\ti s^+,\ti s^+)\!+\!\Re Q(s^-,s^-),
\label{vf3eq27}
\e
using $\Re Q(\ti s^+,s^-)=0$ as $\ti E^+,E^-\vert_{\smash{U'}}$ are orthogonal w.r.t.\ $\Re Q\vert_{\smash{U'}}$.

For each $u\in U'$, we now have
\begin{align*}
s^+(u)&\!=\!0 && \!\!\!\!\!\Longleftrightarrow\!\!\!\!\! & \ti s^+(u)&\!=\!0 && \!\!\!\!\!\Longleftrightarrow\!\!\!\!\! & \Re Q(\ti s^+,\ti s^+)\vert_u &\!=\!0 && \!\!\!\!\!\Longleftrightarrow\!\!\!\!\! \\ \Re Q(s^-,s^-)\vert_u &\!=\!0 && \!\!\!\!\!\Longleftrightarrow\!\!\!\!\!  & \ti s^+(u)\!=\!s^-(u)&\!=\!0 && \!\!\!\!\!\Longleftrightarrow\!\!\!\!\! & s(u)&\!=\!0,
\end{align*}
using $\ti E^+\ra E^+\vert_{\smash{U'}}$ an isomorphism mapping $\ti s^+\mapsto s^+\vert_{\smash{U'}}$ in the first step, $\Re Q$ positive definite on $\ti E^+$ in the second, \eq{vf3eq27} in the third, $\Re Q$ negative definite on $E^-\vert_{\smash{U'}}$ in the fourth, and $s\vert_{\smash{U'}}=\ti s^+\op s^-$ in the fifth.

This proves there exists an open neighbourhood $U'$ of $v$ in $U$ such that $s^{-1}(0)\cap U'=(s^+)^{-1}(0)\cap U'$, which is Theorem \ref{vf3thm2}(c), except that $U'$ is a neighbourhood of $v$ rather than of~$s^{-1}(0)\cap U$.
\label{vf3ex2}
\end{ex}

\begin{rem} Pairs $(U,E^-)$ satisfying $(\dag)$ will be used to prove our main result, constructing a derived manifold structure $\bX_{\rm dm}$ on the complex analytic topological space $X_\an$ of a $-2$-shifted symplectic derived $\C$-scheme~$(\bX,\om_\bX^*)$.

Our construction apparently uses less than the full $-2$-shifted symplectic structure $\om_\bX^*$ on $\bX$. In particular, conditions $(*),(\dag)$ only involve the nondegenerate pairings $\om_\bX^0\vert_x$ on $H^1\bigl(\bT_\bX\vert_x\bigr)$ in \eq{vf2eq6}, which depend only on the presymplectic structure $\om_\bX^0$, not the symplectic structure $\om_\bX^*=(\om_\bX^0,\om_\bX^1,\ldots)$. The proofs of Theorem \ref{vf3thm2}(a),(b) in \S\ref{vf51}--\S\ref{vf52} also use only $\om_\bX^0$ rather than~$\om_\bX^*$. 

However, the proof of Theorem \ref{vf3thm2}(c) in \S\ref{vf53} involves $\om_\bX^*$, as it uses the existence of a minimal $-2$-Darboux form presentation for $(\bX,\om_\bX^*)$ near each $x\in X_\an$, as in Theorem \ref{vf2thm3}. The authors do not know whether Theorem \ref{vf3thm2}(c) holds for $-2$-shifted presymplectic $(\bX,\om_\bX^0)$ which are not symplectic.
\label{vf3rem1}
\end{rem}

\subsection{\texorpdfstring{Comparing $(U_J,E_J^-),(U_K,E_K^-)$ under $\Phi_{JK}$}{Comparing (Uᴶ,Eᴶ⁻),(Uᴷ,Eᴷ⁻) under ϕᴶᴷ}}
\label{vf34}

Section \ref{vf33} discussed how to use standard form charts $\bs\al:\bSpec A^\bu\ra\bX$ on $(\bX,\om_\bX^*)$ to choose pairs $(U,E^-)$, and so define Kuranishi neighbourhoods $(U,E^+,s^+,\psi^+)$ on $X_\an$. We now explain how to pull back such pairs $(U_K,E_K^-)$ along a quasi-free morphism $\Phi_{JK}:A_K^\bu\ra A_J^\bu$, and construct coordinate changes between the Kuranishi neighbourhoods~$(U_J,E_J^+,s_J^+,\psi_J^+),(U_K,E_K^+,s_K^+,\psi_K^+)$.

\begin{dfn} Let $(\bX,\om_\bX^*)$ be a $-2$-shifted symplectic derived $\C$-scheme with $\vdim_\C\bX=n$, and suppose $\Phi_{JK}:A_K^\bu\ra A_J^\bu$ is a quasi-free morphism of standard form cdgas over $\C$ and $\bs\al_J:\bSpec A_J^\bu\hookra\bX$, $\bs\al_K:\bSpec A_K^\bu\hookra\bX$ are Zariski open inclusions such that \eq{vf3eq3} homotopy commutes. Define complex geometric data $V_J,E_J,F_J,s_J,t_J,\psi_J,R_J,V_K,E_K,F_K,s_K,t_K,\psi_K,R_K,\phi_{JK},\chi_{JK},\ab\xi_{JK}$ in Definitions \ref{vf3def1}, \ref{vf3def2}, and suppose $R_J\ne\es$, so $R_K\ne\es$ as~$R_J\subseteq R_K\subseteq X_\an$.

Consider pairs $(U_J,E_J^-)$ for $A_J^\bu$ and $(U_K,E_K^-)$ for $A_K^\bu$ satisfying condition $(*)$ in Definition \ref{vf3def3}. We say that $(U_J,E_J^-)$ and $(U_K,E_K^-)$ are {\it compatible\/} if $\phi_{JK}(U_J)\subseteq U_K$ and $\chi_{JK}\vert_{U_J}(E_J^-)\subseteq \phi_{JK}\vert_{U_J}^*(E_K^-)\subseteq \phi_{JK}\vert_{U_J}^*(E_K)$.

For $(U_J,E_J^-),(U_K,E_K^-)$ compatible, define a vector bundle morphism $\chi_{JK}^+:E_J^+\ra\phi_{JK}\vert_{U_J}^*(E_K^+)$ on $U_J$ by the commutative diagram with exact rows
\begin{equation*}
\xymatrix@C=19pt@R=11pt{
0 \ar[r] & E_J^- \ar[d]^{\chi_{JK}\vert_{E_J^-}} \ar[r] & E_J\vert_{U_J} \ar[d]^{\chi_{JK}\vert_{U_J}} \ar[r] & E_J^+\ar[r] \ar@{.>}[d]^{\chi_{JK}^+} & 0 \\
0 \ar[r] & \phi_{JK}\vert_{U_J}^*(E_K^-) \ar[r] & \phi_{JK}\vert_{U_J}^*(E_K) \ar[r] & \phi_{JK}\vert_{U_J}^*(E_K^+) \ar[r] & 0.\!{} }
\end{equation*}

Let $v_J\in s_J^{-1}(0)\subseteq U_J\subseteq V_J$ with $\phi_{JK}(v_J)=v_K\in s_K^{-1}(0)\subseteq U_K\subseteq V_K$ and $\psi_J(v_J)=\psi_K(v_K)=x\in X_\an$. Consider the diagram, with rows \eq{vf3eq25} for $(U_J,E_J^-),v_J$ and $(U_K,E_K^-),v_K$
\e
\begin{gathered}
\xymatrix@C=11pt@R=13pt{ 0 \ar[r] & H^0\bigl(\bT_\bX\vert_x\bigr) \ar[d]_\id \ar[r] & T_{v_J}U_J \ar[rr]_{\d s_J^+\vert_{v_J}} \ar[d]_{\d\phi_{JK}\vert_{v_J}} && E_J^+\vert_{v_J} \ar[d]^{\chi_{JK}^+\vert_{v_J}} \ar[r] & H^1\bigl(\bT_\bX\vert_x\bigr)/\Im\Pi_{v_J} \ar[d]^\id \ar[r] & 0 \\
0 \ar[r] & H^0\bigl(\bT_\bX\vert_x\bigr) \ar[r] & T_{v_K}U_K \ar[rr]^{\d s_K^+\vert_{v_K}} && E_K^+\vert_{v_K} \ar[r] & H^1\bigl(\bT_\bX\vert_x\bigr)/\Im\Pi_{v_K} \ar[r] & 0.\!{}  }\!\!\!\!\!\!{}
\end{gathered}
\label{vf3eq28}
\e
Here if we regard $\Im\Pi_{v_J},\Im\Pi_{v_K}$ from \eq{vf3eq24} as subspaces of $H^1\bigl(\bT_\bX\vert_x\bigr)$ using \eq{vf3eq9}, compatibility $\chi_{JK}(E_J^-\vert_{v_J})\subseteq E_K^-\vert_{v_K}$ and \eq{vf3eq18} imply that $\Im\Pi_{v_J}\subseteq\Im\Pi_{v_K}$, so $\Im\Pi_{v_J}=\Im\Pi_{v_K}$ as they have the same dimension by $(*)$, and the right hand column of \eq{vf3eq28} makes sense. From \eq{vf3eq14}, \eq{vf3eq17} and \eq{vf3eq18} we see that \eq{vf3eq28} commutes. Elementary linear algebra then gives an exact sequence 
\e
\smash{\raisebox{-3pt}{$\displaystyle\xymatrix@C=13.5pt{ 0 \ar[r] & T_{v_J}U_J \ar[rrrr]^{\d s_J^+\vert_{v_J}\op \d\phi_{JK}\vert_{v_J}} &&&& {\raisebox{4pt}{$\begin{subarray}{l}\ts E_J^+\vert_{v_J}\op \\ \ts T_{v_K}U_K\end{subarray}$}}  \ar[rrrr]^{-\chi_{JK}^+\vert_{v_J}\op \d s_K^+\vert_{v_K}} &&&& E_K^+\vert_{v_K} \ar[r] & 0. }\!\!\!{}$}}
\label{vf3eq29}
\e

\label{vf3def4}
\end{dfn}

From \eq{vf3eq29} and Definition \ref{vf2def11}, we deduce:

\begin{cor} In the situation of Definition\/ {\rm\ref{vf3def4},} if\/ $(U_J,E_J^-)$ and\/ $(U_K,E_K^-)$ are compatible and satisfy $(\dag)$ then in the sense of\/ {\rm\S\ref{vf25},}
\begin{equation*}
\smash{(U_J,\phi_{JK}\vert_{U_J},\chi_{JK}^+):(U_J,E_J^+,s_J^+,\psi_J)\longra(U_K,E_K^+,s_K^+,\psi_K)}
\end{equation*}
is a coordinate change of Kuranishi neighbourhoods on $X_\an$.
\label{vf3cor2}
\end{cor}

\begin{lem} In the situation of Definition\/ {\rm\ref{vf3def4},} fix $(U_K,E_K^-)$ satisfying $(*)$ for $A_K^\bu,\bs\al_K$. Set\/ $U_{JK}'=\phi_{JK}^{-1}(U_K)\subseteq V_J$. Then $E_{JK}':=\chi_{JK}\vert_{U_{JK}'}^{-1}(E_K^-)$ is a vector subbundle of\/ $E_J\vert_{U_{JK}'},$ as $\chi_{JK}$ is surjective. Choose a complementary real vector subbundle $E_{JK}'',$ so that\/ $E_J\vert_{U_{JK}'}=E_{JK}'\op E_{JK}''$. 

Choose a connection $\nabla$ on $E_J,$ so that\/ $\nabla s_J:TV_J\ra E_J$ is a vector bundle morphism. Now $\Ker\bigl(\d\phi_{JK}:TV_J\ra \phi_{JK}^*(TV_K)\bigr)$ is a vector subbundle of\/ $TV_J,$ as $\d\phi_{JK}$ is surjective, and\/ $\nabla s_J$ is injective on $\Ker\d\phi_{JK}$ near $s_J^{-1}(0),$ so $E_{JK}''':=(\nabla s_J)[\Ker\d\phi_{JK}]$ is a vector subbundle of\/ $E_J$ near $s_J^{-1}(0)$ in\/~$V_J$.

Then $(U_J,E_J^-)$ satisfies $(*)$ for $A_J^\bu,\bs\al_J$ and is compatible with\/ $(U_K,E_K^-)$ if and only if\/ $U_J$ is open in $U_{JK}',$ and\/ $E_{JK}^-$ is a vector subbundle of\/ $E_{JK}'\vert_{U_J}$ satisfying $E_J\vert_{U_J}=E_{JK}^-\op E_{JK}''\vert_{U_J}\op E_{JK}'''\vert_{U_J}$ near $s_J^{-1}(0)\cap U_J$ in\/ $U_J$. Alternatively, identifying $E_{JK}'$ with\/ $E_J\vert_{U_{JK}'}/E_{JK}'',$ this condition may be written as $E_{JK}'\vert_{U_J}=E_{JK}^-\op \bigl[(E_{JK}''\op E_{JK}''')/E_{JK}''\bigr]\vert_{U_J}$ near $s_J^{-1}(0)\cap U_J$.
\label{vf3lem1}
\end{lem}

\begin{proof} We deduce $\nabla s_J$ is injective on $\Ker\d\phi_{JK}$ at $v_J\in s_J^{-1}(0)$ using \eq{vf3eq17}, check that $(*)$ for $U_J,E_J^-$ is equivalent to $E_J=E_{JK}^-\op E_{JK}''\op E_{JK}'''$ at each $v_J\in s_J^{-1}(0)$, and note that both are open conditions.
\end{proof}

Lemma \ref{vf3lem1} shows we can always {\it pullback\/} $(U_K,E_K^-)$ satisfying $(*)$ along submersions $\phi_{JK}:V_J\ra V_K$: we just have to choose a complement $E_J^-$ to $(E_{JK}''\op E_{JK}''')/E_{JK}''$ in $E_{JK}'$ on some small open neighbourhood $U_J$ of $s_J^{-1}(0)$ in $U_{JK}'$, for instance, the orthogonal complement w.r.t.\ any metric on $E_{JK}'$. By Theorem \ref{vf3thm2}(c), making $U_J$ smaller, we can suppose $(U_J,E_J^-)$ satisfies~$(\dag)$.

\subsection{Constructing Kuranishi atlases and derived manifolds} 
\label{vf35}

Let $(\bX,\om_\bX^*)$ be a $-2$-shifted symplectic derived $\C$-scheme with $\vdim_\C\bX=n$ in $\Z$, and write $X_\an$ for the complex analytic topological space. Suppose $\bX$ is separated and $X_\an$ is a paracompact topological space. (Paracompactness is automatic if $\bX$ is proper, or quasicompact, or of finite type, or $X_\an$ is second countable.) We will construct a Kuranishi atlas on $X_\an$, in the sense of~\S\ref{vf25}.

First choose a family $\bigl\{(A_i^\bu,\bs\al_i):i\in I\bigr\},$ where $A_i^\bu\in\cdga_\C$ is a standard form cdga, and $\bs\al_i:\bSpec A_i^\bu\hookra\bX$ a Zariski open inclusion in $\dSch_\C$ for each $i$ in $I,$ an indexing set, such that $\bigl\{R_i:=(\Im\bs\al_i)_\an:i\in I\bigr\}$ is an open cover of the complex analytic topological space $X_\an$. This is possible by Theorem \ref{vf2thm1}. If $\bX$ is quasicompact (since $\bX$ is locally of finite type, this is equivalent to $\bX$ being of finite type) then we can take $I$ to be finite.

Apply Theorem \ref{vf3thm1} to get data $A_J^\bu\in\cdga_\C$, $\bs\al_J:\bSpec A_J^\bu\hookra\bX$ for finite $\es\ne J\subseteq I$ and quasi-free $\Phi_{JK}:A_K^\bu\ra A_J^\bu,$ for all finite~$\es\ne K\subseteq J\subseteq I$. 

Use the notation of \S\ref{vf32} to rewrite $A_J^\bu,\Phi_{JK}$ in terms of complex geometry. As in Corollary \ref{vf3cor1}, this gives data $V_J,E_J,F_J,s_J,t_J,\psi_J,R_J$ for all finite $\es\ne J\subseteq I$, and $\phi_{JK},\chi_{JK},\xi_{JK}$ for all finite~$\es\ne K\subseteq J\subseteq I$.

For brevity we write $A=\bigl\{J:$ $\es\ne J\subseteq I$, $J$ is finite$\bigr\}$. The proof of the next result in \S\ref{vf61} is based on McDuff and Wehrheim~\cite[Lem.~7.1.7]{McWe}.

\begin{prop} Suppose $Z$ is a paracompact, Hausdorff topological space and\/ $\{R_i:i\in I\}$ an open cover of\/ $Z$. Then we can choose closed subsets $C_J\subseteq Z$ for all finite $\es\ne J\subseteq I,$ satisfying:
\begin{itemize}
\setlength{\itemsep}{0pt}
\setlength{\parsep}{0pt}
\item[{\bf(i)}] $C_J\subseteq \bigcap_{i\in J}R_i$ for all\/ $J$.
\item[{\bf(ii)}] Each\/ $z\in Z$ has an open neighbourhood\/ $U_z\subseteq Z$ with\/ $U_z\cap C_J\ne\es$ for only finitely many $J$. 
\item[{\bf(iii)}] $C_J\cap C_K\ne\es$ only if\/ $J\subseteq K$ or $K\subseteq J$.
\item[{\bf(iv)}] $\bigcup_{\text{$\es\ne J\subseteq I$ finite}}C_J=Z$.
\end{itemize}
\label{vf3prop1}
\end{prop}

In our case, $X_\an$ is Hausdorff and second countable. It is also locally compact, as it is locally homeomorphic to closed subsets $s_J^{-1}(0)$ of complex manifolds $V_J$. But Hausdorff, locally compact and second countable imply that $X$ is paracompact and normal. Thus Proposition \ref{vf3prop1} applies to $Z=X_\an$ with the open cover $\{R_i:i\in I\}$, and we can choose closed subsets $C_J\subseteq R_J=\bigcap_{i\in J}R_i\subseteq X_\an$ for all $J\in A$ satisfying conditions~(i)--(iv).

The next proposition, proved in \S\ref{vf62} using Theorem \ref{vf3thm2} and Lemma \ref{vf3lem1}, chooses pairs $(U_J,E_J^-)$ satisfying $(\dag)$, as in \S\ref{vf33}, with $(U_J,E_J^-)$, $(U_K,E_K^-)$ compatible near $C_J\cap C_K$ under the quasi-free morphism~$\Phi_{JK}:A_K^\bu\ra A_J^\bu$.

\begin{prop} In the situation above, we can choose $(U_J,E_J^-)$ satisfying condition $(\dag)$ for $V_J,E_J,\ldots$ for each\/ $J\in A,$ such that\/ $\psi_J^{-1}(C_J)\subseteq U_J\subseteq V_J,$ and setting $S_J=\psi_J(s_J^{-1}(0)\cap U_J)$ so that\/ $S_J$ is an open neighbourhood of\/ $C_J$ in $X_\an,$ then for all\/ $J,K\in A,$ we have $S_J\cap S_K\ne\es$ only if\/ $J\subseteq K$ or $K\subseteq J,$ and if\/ $K\subsetneq J$ then there exists open $U_{JK}\subseteq U_J$ with\/ $s_J^{-1}(0)\cap U_{JK}=\psi_J^{-1}(S_J\cap S_K)$ such that\/ $(U_{JK},E_J^-\vert_{U_{JK}})$ is compatible with\/ $(U_K,E_K^-),$ in the sense of\/~{\rm\S\ref{vf34}}. 
\label{vf3prop2}
\end{prop}

We can now prove two of the central results of this paper. 

\begin{thm} Let\/ $(\bX,\om_\bX^*)$ be a $-2$-shifted symplectic derived\/ $\C$-scheme with complex virtual dimension\/ $\vdim_\C\bX=n$ in $\Z,$ and write\/ $X_\an$ for the set of\/ $\C$-points of\/ $X=t_0(\bX)$ with the complex analytic topology. Suppose that\/ $\bX$ is separated, and\/ $X_\an$ is a paracompact topological space. Then we can construct a Kuranishi atlas $\cK$ on $X_\an$ of real dimension $n,$ in the sense of\/ {\rm\S\ref{vf25}}. If\/ $\bX$ is quasicompact (equivalently, of finite type) then we can take $\cK$ to be finite.
\label{vf3thm3}
\end{thm}

\begin{proof} In the discussion from the beginning of \S\ref{vf35} up to Proposition \ref{vf3prop2}, we have constructed the following data:
\begin{itemize}
\setlength{\itemsep}{0pt}
\setlength{\parsep}{0pt}
\item[(i)] A Hausdorff, paracompact topological space $X_\an$.
\item[(ii)] An indexing set $I$, where we write $A=\{J:$ $\es\ne J\subseteq I$, $J$ is finite$\}$.
\item[(iii)] An open cover $\{S_J:J\in A\}$ of $X_\an$, such that $S_J\cap S_K\ne\es$ for $J,K\in A$ only if $J\subseteq K$ or~$K\subseteq J$.
\item[(iv)] For each $J\in A$, a Kuranishi neighbourhood $(U_J,E_J^+,s_J^+,\psi_J^+)$ on $X_\an$ with $\dim U_J-\rank E_J^+=n$, constructed as in \S\ref{vf33} from $(U_J,E_J^-)$ satisfying $(\dag)$, with $\Im\psi_J^+=S_J\subseteq X_\an$.
\item[(v)] For all $J,K\in A$ with $K\subsetneq J$, a coordinate change of Kuranishi neighbourhoods over $S_J\cap S_K$, as in Corollary \ref{vf3cor2},
\begin{equation*}
(U_{JK},\phi_{JK}\vert_{U_{JK}},\chi_{JK}^+):(U_J,E_J^+,s_J^+,\psi_J^+)\longra(U_K,E_K^+,s_K^+,\psi_K^+),
\end{equation*}
since $(U_{JK},E_J^-\vert_{U_{JK}})$ is compatible with~$(U_K,E_K^-)$.
\item[(vi)] For all $J,K,L\in A$ with $L\subsetneq K\subsetneq J$, Corollary \ref{vf3cor1} implies that $\phi_{JL}=\phi_{KL}\ci\phi_{JK}$ and $\chi^+_{JL}=\phi_{JK}^*(\chi^+_{KL})\ci\chi_{JK}^+$ on~$U_{JK}\cap U_{JL}\cap \phi_{JK}^{-1}(U_{KL})$.
\end{itemize}

All this data is a Kuranishi atlas $\cK$ in the sense of Definition \ref{vf2def12}, where the partial order $\pr$ on $A$ is $J\pr K$ if $K\subsetneq J$. If $\bX$ is quasicompact then we can take $I$ finite, so $A$ and $\cK$ are finite.
\end{proof}

Combining Theorems \ref{vf2thm4} and \ref{vf3thm3} yields:

\begin{thm} Let\/ $(\bX,\om_\bX^*)$ be a $-2$-shifted symplectic derived\/ $\C$-scheme with complex virtual dimension\/ $\vdim_\C\bX=n$ in $\Z,$ and write\/ $X_\an$ for the set of\/ $\C$-points of\/ $X=t_0(\bX)$ with the complex analytic topology. Suppose that\/ $\bX$ is separated, so that\/ $X_\an$ is Hausdorff, and also that\/ $X_\an$ is a second countable topological space, which holds if and only if\/ $X$ admits a Zariski open cover $\{X_c:c\in C\}$ with\/ $C$ countable and each\/ $X_c$ a finite type $\C$-scheme.

Then we can make the topological space $X_\an$ into a derived manifold\/ $\bX_\dm$ with real virtual dimension\/ $\vdim_\R\bX_\dm=n,$ in any of the senses {\bf(a)} Joyce's m-Kuranishi spaces $\mKur$ {\rm\cite[\S 4.7]{Joyc6}, \bf(b)} Joyce's d-manifolds\/ $\dMan$ {\rm\cite{Joyc2,Joyc3,Joyc4}, \bf(c)} Borisov--No\"el's derived manifolds $\DerManBN$ {\rm\cite{Bori,BoNo},} or\/ {\bf(d)} Spivak's derived manifolds $\DerManSp$ {\rm\cite{Spiv},} all discussed in~{\rm\S\ref{vf26}}.
\label{vf3thm4}
\end{thm}

We will discuss the dependence of $\bX_\dm$ on choices made in the constructions in \S\ref{vf36}. Note that $\bX_\dm$ in Theorem \ref{vf3thm4} has dimension $\vdim_\R\bX_\dm=\vdim_\C\bX=\ha\vdim_\R\bX$, which is exactly half what we might have expected. 

\subsection{Orientations, bordism classes, and virtual classes}
\label{vf36}

Work in the situation of Theorems \ref{vf3thm3} and \ref{vf3thm4}, so that we have a $-2$-shifted symplectic derived $\C$-scheme $(\bX,\om_\bX^*)$ with complex analytic topological space $X_\an$, a Kuranishi atlas $\cK$ on $X_\an$, and a derived manifold $\bX_\dm$. The next proposition, proved in \S\ref{vf63}, justifies our notions of orientation in~\S\ref{vf24}--\S\ref{vf26}.

\begin{prop} In the situation of Theorems\/ {\rm\ref{vf3thm3}} and\/ {\rm\ref{vf3thm4},} there are canonical\/ $1$-$1$ correspondences between:
\begin{itemize}
\setlength{\itemsep}{0pt}
\setlength{\parsep}{0pt}
\item[{\bf(a)}] Orientations on $(\bX,\om_\bX^*)$ in the sense of\/ {\rm\S\ref{vf24};}
\item[{\bf(b)}] Orientations on $(X_\an,\cK)$ in the sense of\/ {\rm\S\ref{vf25};} and
\item[{\bf(c)}] Orientations on $\bX_\dm$ in the sense of\/ {\rm\S\ref{vf262}}.
\end{itemize}
\label{vf3prop3}
\end{prop}

Next we consider how the derived manifold $\bX_\dm$ in Theorem \ref{vf3thm4} depends on choices made in the construction. Once we have chosen the Kuranishi atlas $\cK$ in Theorem \ref{vf3thm3}, Theorem \ref{vf2thm4} shows that $\bX_\dm$ is determined uniquely up to equivalence in its 2- or $\iy$-category. However, constructing $\cK$ involves many arbitrary choices, and the next proposition, proved in \S\ref{vf64} using the material of \S\ref{vf37}, explains how $\bX_\dm$ depends on these. 

\begin{prop} In the situation of Theorem {\rm\ref{vf3thm4},} for\/ $(\bX,\om_\bX^*),n$ fixed, the derived manifold\/ $\bX_\dm$ depends on choices made in the construction only up to bordisms of derived manifolds which fix the underlying topological space $X_\an$.

That is, if\/ $\bX_\dm,\bX_\dm'$ are possible derived manifolds in Theorem {\rm\ref{vf3thm4},} then we can construct a derived manifold with boundary\/ $\bW_\dm$ with topological space $X_\an\t[0,1]$ and\/ $\vdim\bW_\dm=n+1,$ and an equivalence of derived manifolds $\pd\bW_\dm\simeq \bX_\dm\amalg\bX_\dm',$ topologically identifying $\bX_\dm$ with\/ $X_\an\t\{0\}$ and\/ $\bX_\dm'$ with\/ $X_\an\t\{1\}$. We regard\/ $\bW_\dm$ as a bordism from $\bX_\dm$ to~$\bX_\dm'$.

This bordism $\bW_\dm$ is compatible with orientations in Proposition\/ {\rm\ref{vf3prop3}}. That is, given an orientation on $(\bX,\om_\bX^*),$ we get natural orientations on $\bX_\dm,\ab\bX_\dm',\ab\bW_\dm,$ and an equivalence of oriented derived manifolds\/ $\pd\bW_\dm\simeq -\bX_\dm\ab\amalg\bX_\dm',$ where $-\bX_\dm$ is $\bX_\dm$ with the opposite orientation.
\label{vf3prop4}
\end{prop}

Combining this with material in \S\ref{vf264}--\S\ref{vf265} yields:

\begin{cor} Suppose $(\bX,\om_\bX^*)$ is a proper\/ $-2$-shifted symplectic derived\/ $\C$-scheme, with\/ $\vdim_\C\bX=n,$ and with an orientation in the sense of\/ {\rm\S\ref{vf24}}. Then Theorem\/ {\rm\ref{vf3thm4}} constructs a compact derived manifold\/ $\bX_\dm$ with\/ $\vdim_\R\bX_\dm=n,$ and Proposition\/ {\rm\ref{vf3prop3}} defines an orientation on\/~$\bX_\dm$. 

Although $\bX_\dm$ depends on arbitrary choices, the d-bordism class $[\bX_\dm]_\dbo$ in $B_n(*)$ from\/ {\rm\S\ref{vf264}} and the virtual class $[\bX_\dm]_\virt$ in $H_n(X_\an;\Z)$ from\/ {\rm\S\ref{vf265}} are independent of these, and depend only on $(\bX,\om_\bX^*)$ and its orientation. 

\label{vf3cor3}
\end{cor}

\subsection{\texorpdfstring{Working relative to a smooth base $\C$-scheme $Z$}{Working relative to a smooth base ℂ-scheme Z}}
\label{vf37}

Let $Z=\Spec B$ be a smooth classical affine $\C$-scheme, which we now assume is connected. Then the set $Z_\an$ of $\C$-points of $Z$ is a complex manifold, and hence a real manifold. In this section we will show that all of \S\ref{vf31}--\S\ref{vf36} also works relatively over the base $Z$. To do this, we will need a notion of a family $(\bs\pi:\bX\ra Z,\om_{\bX/Z})$ of $-2$-shifted symplectic derived $\C$-schemes over the base~$Z$. 

To understand the next definition, recall from Remark \ref{vf3rem1} that if $(\bX,\om_\bX^*)$ is $-2$-shifted symplectic, then the derived manifold $\bX_\dm$ constructed in \S\ref{vf35} does not depend on the whole sequence $\om_\bX^*=(\om_\bX^0,\om_\bX^1,\ldots)$, but only on the nondegenerate pairings $\om_\bX^0\vert_x$ on $H^1(\bT_\bX\vert_x)$ for $x\in X_\an$, and therefore only on the cohomology class $[\om_\bX^0]\in H^{-2}(\bL_\bX)$. We require that choices of $\om_\bX^1,\om_\bX^2,\ldots$ should exist (they are needed to apply Theorem \ref{vf2thm3}, which is used in the proof of Theorem \ref{vf3thm2}(c)), but $\bX_\dm$ does not depend on them.

\begin{dfn} Let $\bX$ be a derived $\C$-scheme, $Z=\Spec B$ a smooth, connected, classical affine $\C$-scheme, and $\bs\pi:\bX\ra Z$ a morphism. A {\it family of $-2$-shifted symplectic structures on\/} $\bX/Z$ is $[\om_{\bX/Z}]\in H^{-2}(\bL_{\bX/Z})$, such that for each $z\in Z_\an$, writing $\bX^z=\bs\pi^{-1}(z)=\bX\t_{\bs\pi,Z,z}^h*$ for the fibre of $\bs\pi$ over $z$ and $[\om_{\bX/Z}]\vert_{\bX^z}\in H^{-2}(\bL_{\bX^z})$ for the restriction of $[\om_{\bX/Z}]$ to $\bX^z$, then there should exist a $-2$-shifted symplectic structure $\om_{\bX^z}^*=(\om_{\bX^z}^0,\om_{\bX^z}^1,\ldots)$ on $\bX^z$ such that $[\om_{\bX/Z}]\vert_{\bX^z}=[\om_{\bX^z}^0]$ in~$H^{-2}(\bL_{\bX^z})$.
\label{vf3def5}
\end{dfn}

That is, a family of $-2$-shifted symplectic structures on $\bX/Z$ is a $-2$-shifted relative 2-form $[\om_{\bX/Z}]$ on $\bX/Z$, which on each fibre $\bX^z$ extends to a closed 2-form which is $-2$-shifted symplectic. We will explain how to extend the arguments of \S\ref{vf33}--\S\ref{vf36} to the relative case. Here is the analogue of Definition~\ref{vf3def3}:

\begin{dfn} Let $\bX$ be a derived $\C$-scheme, $Z=\Spec B$ a smooth, classical, affine $\C$-scheme of pure dimension, $\bs\pi:\bX\ra Z$ a morphism, and $[\om_{\bX/Z}]$ in $H^{-2}(\bL_{\bX/Z})$ a family of $-2$-shifted symplectic structures on $\bX/Z$. Write $\dim_\C Z=k$ and $\vdim_\C\bX=n+k$. Suppose $A^\bu\in\cdga_\C$ is of standard form, $\bs\al:A^\bu\hookra\bX$ is a Zariski open inclusion, and $\be:B\ra A^0$ is a smooth morphism of $\C$-algebras, such that \eq{vf3eq10} homotopy commutes. Define complex geometric data $V,\tau,E,F,s,t$ and $\psi:s^{-1}(0)\,{\buildrel\cong\over\longra}\,R\subseteq X_\an$ as in Definition \ref{vf3def1}, and suppose $R\ne\es$. Then for each $v\in s^{-1}(0)$ with $\psi(v)=x\in X_\an$ and $\tau(v)=\pi(x)=z\in Z_\an$, equation \eq{vf3eq12} gives an isomorphism from a vector space depending on $V,\tau,Z_\an,E,F,s,t,\tau$ at $v$ to~$H^1(\bT_{\bX/Z}\vert_x)$.

As in \eq{vf2eq6}, the relative 2-form $[\om_{\bX/Z}]$ induces a pairing
\e
\smash{Q_x:=\om_{\bX/Z}^0\vert_x\,\cdot:H^1\bigl(\bT_{\bX/Z}\vert_x\bigr)\t H^1\bigl(\bT_{\bX/Z}\vert_x\bigr)\longra\C,}
\label{vf3eq30}
\e
which is nondegenerate as under the equivalence $\bT_{\bX/Z}\vert_x\simeq \bT_{\bX^z}\vert_x$, $Q_x$ is identified with the pairing induced by a $-2$-shifted symplectic form $\om_{\bX^z}^*$ on $\bX^z$, as in Definition \ref{vf3def5}. Define
\e
\ti Q_v\!:\!\frac{\ts \Ker \bigl(t\vert_v:E\vert_v\ra F\vert_v\bigr)}{\ts \Im \bigl(\d s\vert_v:T_v(V/Z_\an)\!\ra\! E\vert_v\bigr)}\!\t\! \frac{\ts \Ker \bigl(t\vert_v:E\vert_v\ra F\vert_v\bigr)}{\ts \Im \bigl(\d s\vert_v:T_v(V/Z_\an)\!\ra\! E\vert_v\bigr)}\!\longra\!\C
\label{vf3eq31}
\e
to be the nondegenerate complex quadratic form identified with $Q_x$ in \eq{vf3eq30} by the isomorphism $H^1\bigl(\bT_{\bs\al}\vert_v\bigr)$ in~\eq{vf3eq12}.

Consider pairs $(U,E^-)$, where $U\subseteq V$ is open and $E^-$ is a real vector subbundle of $E\vert_U$. Given such $(U,E^-)$, we write $E^+=E\vert_U/E^-$ for the quotient vector bundle over $U$, and $s^+\in C^\iy(E^+)$ for the image of $s\vert_U$ under the projection $E\vert_U\ra E^+$, and $\psi^+:=\psi\vert_{s^{-1}(0)\cap U}:s^{-1}(0)\cap U\ra X_\an$. We say that $(U,E^-)$ {\it satisfies condition\/} $(*)$ if:
\begin{itemize}
\setlength{\itemsep}{0pt}
\setlength{\parsep}{0pt}
\item[$(*)$] For each $v\in s^{-1}(0)\cap U$, we have 
\ea
\Im\bigl(\d s\vert_v:T_v(V/Z_\an)\ra E\vert_v\bigr)\cap E^-\vert_v&=\{0\} &&\text{in $E\vert_v$,}
\label{vf3eq32}\\
t\vert_v\bigl(E^-\vert_v\bigr)&=t\vert_v\bigl(E\vert_v\bigr)&&\text{in $F\vert_v$,}
\label{vf3eq33}
\ea
and the natural real linear map
\e
{}\!\!\!\!\!\Pi_v:E^-\vert_v\cap \Ker \bigl(t\vert_v:E\vert_v\ra F\vert_v\bigr)\ra
\frac{\ts \Ker \bigl(t\vert_v:E\vert_v\ra F\vert_v\bigr)}{\ts \Im \bigl(\d s\vert_v:T_v(V/Z_\an)\ra E\vert_v\bigr)}\,,
\label{vf3eq34}
\e
which is injective by \eq{vf3eq32}, has image $\Im\Pi_v$ a real vector subspace of dimension exactly half the real dimension of $\Ker (t\vert_v)/\Im(\d s\vert_v)$, and the real quadratic form $\Re \ti Q_v$ on $\Ker (t\vert_v)/\Im(\d s\vert_v)$ from \eq{vf3eq31} restricts to a negative definite real quadratic form on~$\Im\Pi_v$.
\end{itemize}

We say $(U,E^-)$ {\it satisfies condition\/} $(\dag)$ if:
\begin{itemize}
\setlength{\itemsep}{0pt}
\setlength{\parsep}{0pt}
\item[$(\dag)$] $(U,E^-)$ satisfies condition $(*)$ and $s^{-1}(0)\cap U=(s^+)^{-1}(0)\subseteq U$.
\end{itemize}
Then $(U,E^+,s^+,\psi^+)$ is a Kuranishi neighbourhood on~$X_\an$.

Observe that if $v\in s^{-1}(0)\cap U$ with $\psi(v)=x\in X_\an$ then using \eq{vf3eq11}--\eq{vf3eq12} and \eq{vf3eq32}--\eq{vf3eq34} we find as for \eq{vf3eq25} that there is an exact sequence
\e
\smash{\xymatrix@C=9.5pt{ 0 \ar[r] & H^0\bigl(\bT_{\bX/Z}\vert_x\bigr) \ar[r] & T_v(V/Z_\an) \ar[r] & E^+\vert_v \ar[r] & H^1\bigl(\bT_{\bX/Z}\vert_x\bigr)/\Im\Pi_v \ar[r] & 0. }}
\label{vf3eq35}
\e
Hence as for \eq{vf3eq26} we have
\begin{align*}
\dim_\R U&-\dim_\R Z_\an-\rank_\R E^+\\
&=\dim_\R H^0\bigl(\bT_{\bX/Z}\vert_x\bigr)-\dim_\R H^1\bigl(\bT_{\bX/Z}\vert_x\bigr)+\dim_\R\Im\Pi_v\\
&=2\dim_\C H^0\bigl(\bT_{\bX/Z}\vert_x\bigr)-\dim_\C H^1\bigl(\bT_{\bX/Z}\vert_x\bigr)\\
&=\dim_\C H^0\bigl(\bT_{\bX/Z}\vert_x\bigr)-\dim_\C H^1\bigl(\bT_{\bX/Z}\vert_x\bigr)+
\dim_\C H^2\bigl(\bT_{\bX/Z}\vert_x\bigr)\\
&=\vdim_\C\bX-\dim_\C Z=n.
\end{align*}
Thus the Kuranishi neighbourhood $(U,E^+,s^+,\psi^+)$ has virtual dimension
\begin{equation*}
\dim U-\rank E^+=n+2k=\ha(\vdim_\R\bX-\dim_\R Z_\an)+\dim_\R Z_\an,
\end{equation*}
which is the real dimension of the base $Z_\an$, plus half the real virtual dimension of the fibres $\bX^z$. 
\label{vf3def6}
\end{dfn}

Note that essentially the only important difference between Definitions \ref{vf3def3} and \ref{vf3def6} is that $T_vV$ in equations \eq{vf3eq21}, \eq{vf3eq22}, \eq{vf3eq24} is replaced by $T_v(V/Z_\an)$ in equations \eq{vf3eq31}, \eq{vf3eq32}, and~\eq{vf3eq34}. 

\begin{thm} Theorem\/ {\rm\ref{vf3thm2}} holds with Definition\/ {\rm\ref{vf3def6}} in place of Definition\/~{\rm\ref{vf3def3}}.
\label{vf3thm6}
\end{thm}

\begin{proof} In the proofs of Theorem \ref{vf3thm2}(a),(b) in \S\ref{vf51}--\S\ref{vf52}, we replace $\d s\vert_v:T_vV\ra E\vert_v$ by $\d s\vert_v:T_v(V/Z_\an)\ra E\vert_v$ throughout, and no other changes are needed.

For part (c), fix $z\in Z_\an$, so that Definition \ref{vf3def5} gives a $-2$-shifted symplectic derived $\C$-scheme $(\bX^z,\om_{\bX^z}^*)$ with $[\om_{\bX/Z}]\vert_{\bX^z}=[\om_{\bX^z}^0]$ in $H^{-2}(\bL_{\bX^z})$. Consider the complex submanifolds $V^z=\tau^{-1}(z)$ in $V$ and $U^z=U\cap V^z$ in $U$, and write $E^z,F^z,s^z,t^z$ for the restrictions of $E,F,s,t$ to $V^z$, and $E^{\pm z},s^{+z},\psi^{+z}$ for the restrictions of $E^\pm,s^+,\psi^+$ to $U^z$. Then $(\bX^z,\om_{\bX^z}^*),V^z,E^z,\ldots$ satisfy Definition \ref{vf3def3}, so Theorem \ref{vf3thm2}(c) shows that $(s^z)^{-1}(0)\cap U^z$ and $(s^{+z})^{-1}(0)$ coincide near $(s^z)^{-1}(0)\cap U^z$ in $U^z$. Hence $(s^{-1}(0)\cap U)\cap\tau^{-1}(z)$ and $((s^+)^{-1}(0))\cap\tau^{-1}(z)$ coincide near $(s^{-1}(0)\cap U)\cap\tau^{-1}(z)$ in $U$. As this holds for all $z\in Z_\an$, $s^{-1}(0)\cap U$ and $(s^+)^{-1}(0)$ coincide near $s^{-1}(0)\cap U$ in $U$, and the theorem follows.
\end{proof}

When we extend \S\ref{vf34} to the relative case, in the analogue of Definition \ref{vf3def4} we also include data $\bs\pi:\bX\ra Z=\Spec B$ and smooth $\be_J:B\ra A_J^0$, $\be_K:B\ra A_K^0$ with $\be_J=\Phi_{JK}\ci\be_K$ and \eq{vf3eq2} homotopy commuting for $J,K$. We obtain an analogue of \eq{vf3eq28} with rows \eq{vf3eq35} rather than \eq{vf3eq25}, and so as for \eq{vf3eq29} we get an exact sequence
\begin{equation*}
\smash{\xymatrix@C=11.5pt{ 0 \ar[r] & T_{v_J}(U_J/Z_\an) \ar[rrrr]^{\d s_J^+\vert_{v_J}\op \d\phi_{JK}\vert_{v_J}} &&&& {\raisebox{4pt}{$\begin{subarray}{l}\ts E_J^+\vert_{v_J}\op \\ \ts T_{v_K}(U_K/Z_\an)\end{subarray}$}}  \ar[rrrr]^{-\chi_{JK}^+\vert_{v_J}\op \d s_K^+\vert_{v_K}} &&&& E_K^+\vert_{v_K} \ar[r] & 0. }\!\!\!{}}
\end{equation*}
But by taking the direct sum of this with $\id:T_zZ_\an\ra T_zZ_\an$ in the second and third positions, we see that this implies \eq{vf3eq29} is exact, and the analogue of Corollary \ref{vf3cor2} follows. The relative analogue of Lemma \ref{vf3lem1}, in which we replace $TV_J,TV_K$ by $T(V_J/Z_\an),T(V_K/Z_\an)$, is immediate.

For \S\ref{vf35}, we prove the following relative analogue of Theorem~\ref{vf3thm3}:

\begin{thm} Let\/ $\bX$ be a separated derived\/ $\C$-scheme, $Z=\Spec B$ a smooth, connected, classical affine $\C$-scheme, $\bs\pi:\bX\ra Z$ a morphism, and\/ $[\om_{\bX/Z}]$ a family of $-2$-shifted symplectic structures on $\bX/Z,$ with\/ $\dim_\C Z=k$ and\/ $\vdim_\C\bX=n+k$. Write\/ $X_\an,Z_\an$ for the sets of\/ $\C$-points of\/ $X=t_0(\bX),Z$ with the complex analytic topology, and suppose $X_\an$ is paracompact. Then we can construct a relative Kuranishi atlas\/ $\cK,\varpi_{J:J\in A}$ for\/ $\pi_\an:X_\an\ra Z_\an$ of real dimension\/ $n+2k,$ as in Definition\/ {\rm\ref{vf2def12},} with $\varpi_J:U_J\ra Z_\an$ a submersion. If\/ $\bX$ is quasicompact (equivalently, of finite type) then we can take $\cK$ to be finite.

\label{vf3thm7}
\end{thm}

\begin{proof} First choose a family $\bigl\{(A_i^\bu,\bs\al_i,\be_i):i\in I\bigr\},$ where $A_i^\bu\in\cdga_\C$ is a standard form cdga, and $\bs\al_i:\bSpec A_i^\bu\hookra\bX$ is a Zariski open inclusion in $\dSch_\C$ for each $i$ in $I,$ an indexing set, and $\be_i:B\ra A_i^0$ is a smooth morphism of classical $\C$-algebras such that \eq{vf3eq1} homotopy commutes, with $\bigl\{R_i:=(\Im\bs\al_i)_\an:i\in I\bigr\}$ an open cover of the complex analytic topological space $X_\an$. This is possible by a relative version of Theorem \ref{vf2thm1}, easily proved by modifying the proof of \cite[Th.~4.1]{BBJ} to work over the base $Z=\Spec B$. Apply Theorem \ref{vf3thm1} to get data $A_J^\bu\in\cdga_\C$, $\bs\al_J:\bSpec A_J^\bu\hookra\bX$, $\be_J:B\ra A_J^0$ for finite $\es\ne J\subseteq I$ and quasi-free morphisms $\Phi_{JK}:A_K^\bu\ra A_J^\bu,$ for all finite~$\es\ne K\subseteq J\subseteq I$. 

Use the notation of \S\ref{vf32} to rewrite $A_J^\bu,\be_J,\Phi_{JK}$ in terms of complex geometry. As in Corollary \ref{vf3cor1}, this gives data $V_J,\tau_J,E_J,F_J,s_J,t_J,\psi_J,R_J$ for all finite $\es\ne J\subseteq I$, and $\phi_{JK},\chi_{JK},\xi_{JK}$ for all finite $\es\ne K\subseteq J\subseteq I$. Note that the holomorphic submersions $\tau_J:V_J\ra Z_\an$ with $\tau_J=\tau_K\ci\phi_{JK}$ for $K\subseteq J$ were not used in \S\ref{vf33}--\S\ref{vf36} as there $Z_\an$ was the point $*$, but now we need them. 

Proposition \ref{vf3prop2} now also holds in our relative situation. Its proof in \S\ref{vf62} uses Theorem \ref{vf3thm2} and Lemma \ref{vf3lem1}, which as above hold in the relative situation with Definition \ref{vf3def6} and $T(V_J/Z_\an)$ in place of Definition \ref{vf3def3} and $TV_J$. As in the proof of Theorem \ref{vf3thm3}, we have now constructed a Kuranishi atlas $\cK$ on $X_\an$, with dimension $n+2k$. Setting $\varpi_J:=\tau_J\vert_{U_J}:U_J\ra Z_\an$ for $J\in A$, we see that $\cK,\varpi_{J:J\in A}$ is a relative Kuranishi atlas for $\pi_\an$, with $\varpi_J$ a submersion. If $X$ is quasicompact we can take $I$ finite, so $A$ and $\cK$ are finite.
\end{proof}

We then deduce the following relative analogue of Theorem~\ref{vf3thm4}:

\begin{thm}{\bf(i)} Let\/ $\bX$ be a separated derived\/ $\C$-scheme, $Z=\Spec B$ a smooth, connected, classical affine $\C$-scheme, $\bs\pi:\bX\ra Z$ a morphism, and\/ $[\om_{\bX/Z}]$ a family of $-2$-shifted symplectic structures on $\bX/Z,$ with\/ $\dim_\C Z=k$ and\/ $\vdim_\C\bX=n+k$. Write\/ $X_\an,Z_\an$ for the sets of\/ $\C$-points of\/ $X=t_0(\bX),Z$ with the complex analytic topology, and suppose $X_\an$ is second countable.

Then we can make the topological space $X_\an$ into a derived manifold\/ $\bX_\dm$ with real virtual dimension\/ $\vdim_\R\bX_\dm=n+2k,$ in any of the senses {\bf(a)} Joyce's m-Kuranishi spaces $\mKur$ {\rm\cite[\S 4.7]{Joyc6}, \bf(b)} Joyce's d-manifolds\/ $\dMan$ {\rm\cite{Joyc2,Joyc3,Joyc4}, \bf(c)} Borisov--No\"el's derived manifolds $\DerManBN$ {\rm\cite{Bori,BoNo},} or\/ {\bf(d)} Spivak's derived manifolds $\DerManSp$ {\rm\cite{Spiv},} all discussed in~{\rm\S\ref{vf26}}.
\smallskip

\noindent{\bf(ii)} We can also define a morphism of derived manifolds $\bs\pi_\dm:\bX_\dm\ra Z_\an,$ with underlying continuous map\/~$\pi_\an:X_\an\ra Z_\an$.
\smallskip

\noindent{\bf(iii)} For each\/ $z\!\in\! Z_\an,$ the fibre $\bX_\dm^z\!=\!\bs\pi_\dm^{-1}(z)\!=\!\bX_\dm\t_{\bs\pi_\dm,Z_\an,z}*$ is a derived manifold with\/ $\vdim_\R\bX_\dm^z=n$. From Definition\/ {\rm \ref{vf3def5},} $\bX^z=\bs\pi^{-1}(z)$ has a $-2$-shifted symplectic structure $\om_{\bX^z}^*,$ and both\/ $\bX_\dm^z,\bX^z$ have (complex analytic) topological space $\pi_\an^{-1}(z)\subseteq X_\an$. Then $\bX_\dm^z$ is up to equivalence a possible choice for the derived manifold associated to $(\bX^z,\om_{\bX^z}^*)$ in Theorem\/~{\rm\ref{vf3thm4}}.

\label{vf3thm8}
\end{thm}

\begin{proof} Parts (i),(ii) follow from Theorems \ref{vf2thm4} and \ref{vf3thm7}. For (iii), if $z\in Z_\an$ then as $\tau_J:V_J\ra Z_\an$ is a holomorphic submersion for $J\in A$, the fibre $V_J^z:=\tau_J^{-1}(z)$ is a complex submanifold of $V_J$. Setting $U_J^z=U_J\cap V_J^z$ and writing $E_J^z,F_J^z,s_J^z,t_J^z$ for the restrictions of $E_J,F_J,s_J,t_J$ to $V_J^z$, and $E^{-z}_J,E^{+z}_J,s^{+z}_J,\psi^{+z}_J$ for the restrictions of $E^-_J,E^+_J,s^+_J,\psi^+_J$ to $U_J^z$, we see that $I,A,V_J^z,E_J^z,F_J^z,s_J^z,t_J^z,U_J^z,\ldots$ are a possible choice for the data $I,A,V_J,E_J,\ldots$ in the application of Theorems \ref{vf3thm3} and \ref{vf3thm4} to $(\bX^z,\om_{\bX^z}^*)$. But from facts about fibre products of derived manifolds in \cite{Joyc2,Joyc3,Joyc4,Joyc7} we see that the derived manifold $\bX_\dm^z=\bX_\dm\t_{\bs\pi_\dm,Z_\an,z}*$ may be constructed as above from the data $I,A,\ab U_J^z,\ab E^{+z}_J,\ab s^{+z}_J,\ab\psi^{+z}_J,\ldots.$ The theorem follows.
\end{proof}

Next we discuss orientations, generalizing \S\ref{vf24} and \S\ref{vf36} to the relative case. Here is the analogue of Definition~\ref{vf2def8}:

\begin{dfn} Let $\bX$ be a derived $\C$-scheme, $Z=\Spec B$ a smooth, connected, classical affine $\C$-scheme, $\bs\pi:\bX\ra Z$ a morphism, and $[\om_{\bX/Z}]\in H^{-2}(\bL_{\bX/Z})$ a family of $-2$-shifted symplectic structures on $\bX/Z$. Then as in \eq{vf2eq4}, $[\om_{\bX/Z}]$ induces a canonical isomorphism of line bundles on $X=t_0(\bX)$:
\begin{equation*}
\smash{\io_{\bX/Z,\om_{\bX/Z}}:\bigl[\det(\bL_{\bX/Z}\vert_X)\bigr]{}^{\ot^2}\longra \O_X\cong\O_X^{\ot^2}.}
\end{equation*}
An {\it orientation\/} for $(\bs\pi:\bX\ra Z,[\om_{\bX/Z}])$ is an isomorphism $o:\det(\bL_{\bX/Z}\vert_X)\ra\O_X$ such that~$o\ot o=\io_{\bX/Z,\om_{\bX/Z}}$.
\label{vf3def7}
\end{dfn}

Here is the relative analogue of Proposition \ref{vf3prop3}. In (b),(c), we could use also use notions of relative orientation for $(X_\an,\cK)\ra Z_\an$ and $\bX_\dm\ra Z_\an$. But as $Z_\an$ is a complex manifold with a natural orientation, these are equivalent to absolute orientations for $(X_\an,\cK),\bX_\dm$, so we do not bother. The proof is an easy modification of that in~\S\ref{vf63}.
 
\begin{prop} In the situation of Theorems\/ {\rm\ref{vf3thm7}} and\/ {\rm\ref{vf3thm8},} there are canonical\/ $1$-$1$ correspondences between:
\begin{itemize}
\setlength{\itemsep}{0pt}
\setlength{\parsep}{0pt}
\item[{\bf(a)}] Orientations on $(\bs\pi:\bX\ra Z,[\om_{\bX/Z}])$ in the sense of Definition\/ {\rm\ref{vf3def7};}
\item[{\bf(b)}] Orientations on $(X_\an,\cK)$ in the sense of\/ {\rm\S\ref{vf25};} and
\item[{\bf(c)}] Orientations on $\bX_\dm$ in the sense of\/ {\rm\S\ref{vf262}}.
\end{itemize}
\label{vf3prop5}
\end{prop}

The relative analogue of Proposition \ref{vf3prop4} does hold, but we will not prove it, as we do not need it. The next theorem says that the virtual classes $[\bX_\dm]_\dbo,[\bX_\dm]_\virt$ of a proper oriented $-2$-shifted symplectic derived $\C$-scheme $(\bX,\om_\bX^*)$ defined in Corollary \ref{vf3cor3} are unchanged under deformation in families. Note that it is essential that the base $\C$-scheme $Z$ be connected in Theorem~\ref{vf3thm9}.

\begin{thm} Let\/ $\bX$ be a separated derived\/ $\C$-scheme, $Z=\Spec B$ a smooth, connected, classical affine $\C$-scheme, $\bs\pi:\bX\ra Z$ a proper morphism, and\/ $[\om_{\bX/Z}]$ a family of $-2$-shifted symplectic structures on $\bX/Z,$ equipped with an orientation, with\/ $\dim_\C Z=k$ and\/ $\vdim_\C\bX=n+k$.

For each\/ $z\in Z_\an$ we have a proper, oriented\/ $-2$-shifted symplectic $\C$-scheme $(\bX^z,\om_{\bX^z}^*)$ with $\vdim\bX^z=n,$ so Corollary\/ {\rm\ref{vf3cor3}} defines a d-bordism class $[\bX_\dm^z]_\dbo\in dB_n(*)$ and a virtual class $[\bX_\dm^z]_\virt\in H_n(X_\an^z;\Z),$ depending only on $(\bX^z,\om_{\bX^z}^*)$. Then $[\bX_\dm^{z_1}]_\dbo=[\bX_\dm^{z_2}]_\dbo$ and\/ $\imath^{z_1}_*([\bX_\dm^{z_1}]_\virt)=\imath^{z_2}_*([\bX_\dm^{z_2}]_\virt)$ for all\/ $z_1,z_2\in Z_\an,$ where $\imath^z_*([\bX_\dm^z]_\virt)\in H_n(X_\an;\Z)$ is the pushforward under the inclusion\/~$\imath^z:X_\an^z\hookra X_\an$. 
\label{vf3thm9}
\end{thm}

\begin{proof} Theorem \ref{vf3thm8} constructs a derived manifold $\bX_\dm$ with $\vdim\bX_\dm=n+2k$ and a morphism $\bs\pi_\dm:\bX_\dm\ra Z_\an$, which is proper as $\bs\pi$ is proper, and Proposition \ref{vf3prop5} gives an orientation on~$\bX_\dm$.

Let $z_1,z_2\in Z_\an$. As $Z$ is connected we can choose a smooth map $\ga:[0,1]\ra Z_\an$ with $\ga(0)=z_1$ and $\ga(1)=z_2$. The fibre product
\begin{equation*}
\smash{\bW_\dm=\bX_\dm\t_{\bs\pi_\dm,Z_\an,\ga}[0,1]}
\end{equation*}
exists as a derived manifold with boundary by \cite[\S 7.5]{Joyc3}, \cite[\S 7.6]{Joyc4}, \cite{Joyc7}, with $\vdim\bW_\dm=n+1$, and $\bW_\dm$ is compact as $[0,1]$ is and $\bs\pi_\dm$ is proper, and oriented since $\bX_\dm,Z_\an,[0,1]$ are. As $\pd\bX_\dm=\pd Z_\an=\es$, the boundary is
\begin{equation*}
\smash{\pd\bW_\dm=\bX_\dm\t_{\bs\pi_\dm,Z_\an,\ga}\pd[0,1]=\bX_\dm^{z_1}\amalg \bX_\dm^{z_2},}
\end{equation*}
where $\bX_\dm^{z_1},\bX_\dm^{z_2}$ are the fibres of $\bs\pi_\dm:\bX_\dm\ra Z_\an$ at~$z_1,z_2$. 

Since $\pd[0,1]=-\{0\}\amalg\{1\}$ in oriented 0-manifolds, we have $\pd\bW_\dm=-\bX_\dm^{z_1}\amalg\bX_\dm^{z_2}$ in oriented derived manifolds. Therefore Definition \ref{vf2def17} gives $[\bX_\dm^{z_1}]_\dbo=[\bX_\dm^{z_2}]_\dbo$ in $dB_n(*)$. By Theorem \ref{vf3thm6}(c) $\bX_\dm^{z_1},\bX_\dm^{z_2}$ are outcomes of Theorem \ref{vf3thm4} applied to $(\bX^{z_1},\om_{\bX^{z_1}}^*),(\bX^{z_2},\om_{\bX^{z_2}}^*)$, so $[\bX_\dm^{z_1}]_\dbo,\ab[\bX_\dm^{z_2}]_\dbo$ are the d-bordism classes associated to $(\bX^{z_1},\om_{\bX^{z_1}}^*)$, $(\bX^{z_2},\om_{\bX^{z_2}}^*)$ in Corollary \ref{vf3cor3}. A similar argument works for the homology classes.
\end{proof}

\begin{rem} The assumptions that $Z$ is smooth, classical, and affine, and $\bX$ is separated, in Theorem \ref{vf3thm9} are easily removed; we can work over a base $\bZ$ which is a general classical or derived $\C$-scheme, provided it is connected.

To see this, suppose $\bs\pi:\bX\ra\bZ$ is a proper morphism of derived $\C$-schemes with $\bZ$ connected, and $[\om_{\bX/\bZ}]\in H^{-2}(\bL_{\bX/\bZ})$ is a family of $-2$-shifted symplectic structures on $\bX/\bZ$ equipped with an orientation, extending Definitions \ref{vf3def5} and \ref{vf3def7} to general $\bZ$ in the obvious way.

Suppose $z_1,z_2\in Z_\an$. As $\bZ$ is connected we can find a sequence $z_1=z^0,z^1,\ldots,z^N=z_2$ of points in $Z_\an$, and a sequence of smooth, connected, affine curves $C^1,\ldots,C^N$ over $\C$ with morphisms $\pi^i:C^i\ra\bZ$, such that $\pi^i(C^i)$ contains $z^{i-1},z^i$ for $i=1,\ldots,N$. Then $\bX^i=\bX\t_{\bs\pi,\bZ,\pi^i}^hC^i$ is a derived $\C$-scheme, and $[\om_{\bX/\bZ}]$ pulls back to a family $[\om_{\bX^i/C^i}]$ of oriented $-2$-shifted symplectic structures on $\bX^i/C^i$. Applying Theorem \ref{vf3thm9} to $(\bX^i\ra C^i,[\om_{\bX^i/C^i}])$ we see that
$[\bX_\dm^{z^{i-1}}]=[\bX_\dm^{z^i}]$ in $dB_n(*)$ for $i=1,\ldots,N$, so that
\begin{equation*}
\smash{[\bX_\dm^{z_1}]_\dbo=[\bX_\dm^{z^0}]=[\bX_\dm^{z^1}]=\cdots=[\bX_\dm^{z^N}]=[\bX_\dm^{z_2}]_\dbo.}
\end{equation*}
The same argument works for virtual classes $[\bX_\dm^{z_i}]_\virt$ in homology.

We took $Z$ to be smooth above to avoid defining families $\bs\pi_\dm:\bX_\dm\ra\bZ$ of derived manifolds over a base $\bZ$ which is not a (derived) manifold.
\label{vf3rem2}
\end{rem}

\subsection[\texorpdfstring{`Holomorphic Donaldson invariants' of Calabi--Yau 4-folds}{\textquoteleft Holomorphic Donaldson invariants\textquoteright\ of Calabi-Yau 4-folds}]{`Holomorphic Donaldson invariants' of C--Y 4-folds}
\label{vf38}

We now outline how the results of \S\ref{vf31}--\S\ref{vf37} can be used to define new enumerative invariants of (semi)stable coherent sheaves on Calabi--Yau 4-folds $Y$, which we could call `holomorphic Donaldson invariants', and which should be unchanged under deformations of $Y$. A related programme using gauge theory has recently been proposed by Cao and Leung \cite{Cao,CaLe1,CaLe2}, which we discuss in~\S\ref{vf39}.

We begin by discussing {\it Donaldson--Thomas invariants\/} $DT^\al(\tau)$ of Calabi--Yau 3-folds, introduced by Thomas \cite{Thom}. Suppose $Z$ is a Calabi--Yau 3-fold over $\C$ with an ample line bundle $\O_Z(1)$, which defines a Gieseker stability condition $\tau$ on coherent sheaves on $Z$, and $\al\in H^{\rm even}(Z;\Q)$. Then one can form coarse moduli $\C$-schemes $\cM_{\rm st}^\al(\tau),\cM_{\rm ss}^\al(\tau)$ of $\tau$-(semi)stable coherent sheaves on $Z$ of Chern character $\al$, with $\cM_{\rm st}^\al(\tau)\subseteq\cM_{\rm ss}^\al(\tau)$ Zariski open, and $\cM_{\rm ss}^\al(\tau)$ proper.

Thomas \cite{Thom} showed that $\cM_{\rm st}^\al(\tau)$ carries an `obstruction theory' $\phi:E^\bu\ra\bL_{\cM_{\rm st}^\al(\tau)}$ of virtual dimension 0, in the sense of Behrend and Fantechi \cite{BeFa}. Thus, if there are no strictly $\tau$-semistable sheaves in class $\al$, so that $\cM_{\rm st}^\al(\tau)=\cM_{\rm ss}^\al(\tau)$ and $\cM_{\rm st}^\al(\tau)$ is proper, then \cite{BeFa} gives a virtual count $DT^\al(\tau)=[\cM_{\rm st}^\al(\tau)]_\virt\in\Z$. Thomas proved that $DT^\al(\tau)$ is unchanged under continuous deformations of~$Z$.

Later, Joyce and Song \cite{JoSo} extended the definition of $DT^\al(\tau)$ to invariants $\bar{DT}{}^\al(\tau)\in\Q$ for all $\al\in H^{\rm even}(Z;\Q)$, dropping the condition that 
there are no strictly $\tau$-semistable sheaves in class $\al$, and proved a wall-crossing formula for $\bar{DT}{}^\al(\tau)$ under change of stability condition $\tau$. At about the same time, Kontsevich and Soibelman \cite{KoSo} defined a motivic generalization of Donaldson--Thomas invariants (assuming existence of `orientation data' as in \S\ref{vf24}), and proved their own wall-crossing formula under change of~$\tau$.

Thomas \cite{Thom} called his invariants $DT^\al(\tau)$ `holomorphic Casson invariants', though they are now generally known as Donaldson--Thomas invariants. Here {\it Casson invariants\/} are integer invariants of oriented real 3-manifolds $Z_\R$ which are homology 3-spheres, which `count' flat connections on~$Z_\R$. 

This followed a programme of Donaldson and Thomas \cite{DoTh}, which starting with some well-known geometry in real dimensions 2,3 and 4, aimed to find analogues in complex dimensions 2,3 and 4; so the complex analogues of homology 3-spheres, and flat connections upon them, are Calabi--Yau 3-folds, and holomorphic vector bundles (or coherent sheaves) upon them.

Donaldson invariants \cite{DoKr} are invariants of compact, oriented 4-manifolds $Y_\R$, defined by `counting' moduli spaces $\cM_{\rm inst}^\al$ of $\SU(2)$-instantons $E$ on $Y_\R$ with $c_2(E)=\al\in\Z$. In contrast to Casson and Donaldson--Thomas invariants, the (virtual) dimension $d^\al$ of $\cM_{\rm inst}^\al$ need not be zero. Oversimplifying / lying a bit, one first constructs an orientation on $\cM_{\rm inst}^\al$, \cite[\S 5.4]{DoKr}. Then we have a virtual class $[\cM_{\rm inst}^\al]_\virt\in H_{d^\al}(\cM_{\rm inst}^\al;\Z)$. For each $\be\in H_2(Y_\R;\Z)$ we construct a natural cohomology class $\mu(\be)\in H^2(\cM_{\rm inst}^\al;\Z)$, with $\mu(\be_1+\be_2)=\mu(\be_1)+\mu(\be_2)$. Then if $d^\al=2k$, we define {\it Donaldson invariants\/} $D^\al(\be_1,\ldots,\be_k)=(\mu(\be_1)\cup\cdots\cup\mu(\be_k))\cdot [\cM_{\rm inst}^\al]_\virt\in\Z$ for all $\be_1,\ldots,\be_k\in H_2(Y_\R;\Z)$. We can think of $D^\al$ as a $\Z$-valued homogeneous degree $k$ polynomial on~$H_2(Y_\R;\Z)$.

We propose, following \cite{DoTh}, to define `holomorphic Donaldson invariants' of Calabi--Yau 4-folds. The gauge theory ideas which were the primary focus of \cite{DoTh} will be discussed in \S\ref{vf39}; here we work in the world of (derived) algebraic geometry. Suppose $Y$ is a Calabi--Yau 4-fold over $\C$ (i.e. $Y$ is smooth and projective with $H^i(\O_Y)=\C$ if $i=0,4$ and $H^i(\O_Y)=0$ otherwise), and $\al=(\al^0,\al^2,\al^4,\al^6,\al^8)\in H^{\rm even}(Y;\Q)$. As above we can form coarse moduli $\C$-schemes $\cM_{\rm st}^\al(\tau)\subseteq\cM_{\rm ss}^\al(\tau)$ of Gieseker (semi)stable coherent sheaves on $Y$ of Chern character $\al$, with $\cM_{\rm ss}^\al(\tau)$ proper.

To make contact with the work of \S\ref{vf31}--\S\ref{vf37}, we need to show:

\begin{claim} There is a $-2$-shifted symplectic derived\/ $\C$-scheme $\bigl(\bs\cM_{\rm st}^\al(\tau),\om^*\bigr),$ natural up to equivalence, with classical truncation\/ $t_0\bigl(\bs\cM_{\rm st}^\al(\tau)\bigr)=\cM_{\rm st}^\al(\tau),$ of virtual dimension\/ $\vdim_\C\bs\cM_{\rm st}^\al(\tau)=d^\al:=2-{\rm deg}\bigl(\al\cup\bar\al\cup{\rm td}(TY)\bigr){}_8,$ where $\bar\al=(\al^0,-\al^2,\al^4,-\al^6,\al^8),$ and\/ ${\rm td}(-)$ is the Todd class.
\label{vf3claim1}
\end{claim}

Pantev et al.\ \cite[\S 2.1]{PTVV} prove the analogue of Claim \ref{vf3claim1} in the context of (derived) Artin stacks, but we want to reduce to (derived) schemes. Roughly this means factoring out the $\C^*$ stabilizer groups at each point of the $\tau$-stable derived moduli stack. Actually, it should not be difficult to extend \S\ref{vf31}--\S\ref{vf37} to derived algebraic $\C$-spaces rather than derived $\C$-schemes, and then it would be enough to construct $\bs\cM_{\rm st}^\al(\tau)$ as a derived algebraic $\C$-space.

Next we would need to answer the:

\begin{quest} Does $\bigl(\bs\cM_{\rm st}^\al(\tau),\om^*\bigr)$ in Claim\/ {\rm\ref{vf3claim1}} have a natural orientation, in the sense of\/ {\rm\S\ref{vf24},} possibly depending on some choice of data on $Y?$
\label{vf3quest1}
\end{quest}

Following the argument of Donaldson \cite[\S 5.4]{DoKr}, Cao and Leung \cite[Th.~2.2]{CaLe2} prove an orientability result, which should translate to the statement that if the Calabi--Yau 4-fold $Y$ has holonomy $\SU(4)$ with $H_*(Y;\Z)$ torsion-free, and $\bs\cM_{\rm st}^\al(\tau)$ is a derived moduli scheme of coherent sheaves on $Y$, then orientations on $\bs\cM_{\rm st}^\al(\tau)$ exist, though they do not construct a natural choice. 

If both these problems are solved, then Theorem \ref{vf3thm4} makes $\cM_{\rm st}^\al(\tau)_\an$ into a derived manifold $\bs\cM_{\rm st}^\al(\tau)_\dm$ of real virtual dimension $d^\al$, which is oriented by Proposition \ref{vf3prop3}. If there are no strictly $\tau$-semistable sheaves in class $\al$ then $\bs\cM_{\rm st}^\al(\tau)_\dm$ is also compact, and has a d-bordism class $[\bs\cM_{\rm st}^\al(\tau)_\dm]_\dbo$ in $dB_{d^\al}(*)$ and virtual class $[\bs\cM_{\rm st}^\al(\tau)_\dm]_\virt$ in~$H_{d^\al}(\cM_{\rm st}^\al(\tau)_\an;\Z)$.

If $d^\al=0$ then $[\bs\cM_{\rm st}^\al(\tau)_\dm]_\dbo\in dB_0(*)\cong\Z$ is the virtual count we want. But if $d^\al>0$ we should aim to find suitable cohomology classes on $\cM_{\rm st}^\al(\tau)_\an$ and integrate them over $[\bs\cM_{\rm st}^\al(\tau)_\dm]_\virt$, as for Donaldson invariants above.

\begin{claim} One can define natural cohomology classes $\mu(\be)$ on $\cM_{\rm st}^\al(\tau)_\an$ depending on homology classes $\be$ on $Y,$ which can be combined with\/ $[\bs\cM_{\rm st}^\al(\tau)_\dm]_\virt$ to give integer invariants, in a similar way to Donaldson invariants.
\label{vf3claim2}
\end{claim}

If $\cM_{\rm st}^\al(\tau)$ is a fine moduli space, there is a universal sheaf $\cE$ on $\cM_{\rm st}^\al(\tau)\t Y$, with Chern classes $c_i(\cE)\!\in\! H^{2i}(\cM_{\rm st}^\al(\tau)_\an\t Y;\Q)\!\cong \!\bigop_kH^{2i-k}(\cM_{\rm st}^\al(\tau)_\an;\Q)\ot H^k(Y;\Q)$, and we can make $\mu_i(\be)\in H^{2i-k}(\cM_{\rm st}^\al(\tau)_\an;\Q)$ by contracting $c_i(\cE)$ with $\be\in H_k(Y;\Q)$. Using the results of \S\ref{vf37}, we should be able to prove that the resulting invariants are unchanged under continuous deformations of~$Y$.

This would take us to the same point as Thomas \cite{Thom} in the Calabi--Yau 3-fold case: we could `count' moduli spaces $\cM_{\rm st}^\al(\tau)$ for those classes $\al$ containing no strictly $\tau$-semistable sheaves, and get a deformation-invariant answer. Many questions would remain, for instance, how to count strictly $\tau$-semistables, wall-crossing formulae as in \cite{JoSo,KoSo}, computation in examples, and so on.

We hope to return to these issues in future work.

\subsection{\texorpdfstring{Motivation from gauge theory, and `$\SU(4)$ instantons'}{Motivation from gauge theory, and \textquoteleft SU(4) instantons\textquoteright}}
\label{vf39}

Finally we discuss some ideas of Donaldson and Thomas \cite{DoTh}, which were part of the motivation for this paper, and the work of Cao and Leung~\cite{Cao,CaLe1,CaLe2}.

Let $Y$ be a Calabi--Yau 4-fold over $\C$, regarded as a compact real 8-manifold $Y$ with complex structure $J$, Ricci-flat K\"ahler metric $g$, K\"ahler form $\om$, and holomorphic volume form $\Om$. Fix a complex vector bundle $E\ra Y$ of rank $r>0$ with Hermitian metric $h$ and Chern character $\ch(E)=\al$, and as in \cite{Cao,CaLe1} assume for simplicity that $c_1(E)=0$. Consider connections $\nabla$ on $E$ preserving $h$ with curvature $F\in C^\iy\bigl(\End(E)\ot_\C(\La^2T^*Y\ot_\R\C)\bigr)$. The splitting
\begin{equation*}
\smash{\La^2T^*Y\ot_\R\C=\langle\om\rangle_\C\op \La^{1,1}_0T^*Y\op \La^{2,0}T^*Y\op \La^{0,2}T^*Y}
\end{equation*}
induces a corresponding decomposition $F=F^\om\op F^{1,1}_0\op F^{2,0}\op F^{0,2}$.

We call $\nabla$ a {\it Hermitian--Einstein connection\/} if $F^\om=F^{2,0}=F^{0,2}=0$. We can split $\nabla=\pd_E\op\db_E$, where $\db_E$ gives $E$ the structure of a holomorphic vector bundle on $(Y,J)$, as $F^{0,2}=0$. The {\it Hitchin--Kobayashi correspondence\/} says that if $(E,\db_E)$ is a holomorphic vector bundle and is slope-stable, then $\db_E$ extends to a unique Hermitian--Einstein connection $\nabla=\pd_E\op\db_E$ preserving $h$. Also, holomorphic vector bundles on $Y$ are algebraic. Thus, studying moduli spaces $\cM^\al_{\text{alg-vb}}$ of stable algebraic vector bundles is roughly equivalent to studying moduli spaces $\cM^\al_{\rm HE}$ of Hermitian--Einstein connections, modulo gauge.

As a system of p.d.e.s, the Hermitian--Einstein equations are {\it overdetermined\/}: there are $8r^2$ unknowns, $13r^2$ equations and $r^2$ gauge equivalences, with $8r^2-13r^2-r^2<0$. Algebraically, this corresponds to the fact that the natural obstruction theory on $\cM_{\text{alg-vb}}$ is not perfect, so we cannot form virtual classes.
 
Using $\Om,g$ we can define real splittings $\La^{2,0}T^*Y=\La^{2,0}_+T^*Y\op \La^{2,0}_-T^*Y$, $\La^{0,2}T^*Y=\La^{0,2}_+T^*Y\op \La^{0,2}_-T^*Y$ and corresponding decompositions $F^{2,0}=F^{2,0}_+\op F^{2,0}_-$, $F^{0,2}=F^{0,2}_+\op F^{0,2}_-$. Following Donaldson and Thomas \cite[\S 3]{DoTh}, we call $\nabla$ an $\SU(4)$-{\it instanton\/} if $F^\om=F^{2,0}_+=F^{0,2}_+=0$. This gives $8r^2$ unknowns, $7r^2$ equations and $r^2$ gauge equivalences, with $8r^2-7r^2-r^2=0$. It is a determined elliptic system, so that we can hope to define virtual classes. This is special to Calabi--Yau 4-folds, a complex analogue of instantons on real 4-manifolds.

Writing $\cM_{\SU(4)}^\al$ for the moduli space of $\SU(4)$-instantons, we have $\cM^\al_{\rm HE}\subseteq\cM^\al_{\SU(4)}$, as the $\SU(4)$ instanton equations are weaker than the Hermitian--Einstein equations. Now $\al=\ch(E)\in\bigop_{p=0}^4H^{p,p}(Y)$ if $E$ admits Hermitian--Einstein connections. Conversely, as in \cite[p.~36]{DoTh}, if $\al\in\bigop_pH^{p,p}(Y)$ then one can use $L^2$-norms of components of $F$ to show that any $\SU(4)$-instanton is Hermitian--Einstein. Thus, either $\cM^\al_{\rm HE}=\cM^\al_{\SU(4)}$, or~$\cM^\al_{\rm HE}=\es$.

However, the equality $\cM^\al_{\rm HE}=\cM^\al_{\SU(4)}$ holds only at the level of sets, or topological spaces. Since $\cM^\al_{\rm HE}$ is defined by more equations, if we regard $\cM^\al_{\rm HE},\cM^\al_{\SU(4)}$ as (derived) $C^\iy$-schemes, for instance, then~$\cM^\al_{\rm HE}\subsetneq\cM^\al_{\SU(4)}$.

In the setting of \S\ref{vf31}--\S\ref{vf36}, we should compare $\cM^\al_{\rm HE}$ (a Calabi--Yau 4-fold moduli space, without a virtual class, equivalent to an algebraic moduli scheme $\cM^\al_{\text{alg-vb}}$) with the $-2$-shifted symplectic derived $\C$-scheme $(\bX,\om_\bX^*)$, and $\cM^\al_{\SU(4)}$ (an elliptic moduli space, hopefully with a virtual class, equal to $\cM^\al_{\rm HE}$ on the level of topological spaces) with the derived manifold $\bX_\dm$. It was these ideas from Donaldson and Thomas \cite{DoTh} that led the authors to believe that one could modify a $-2$-shifted symplectic derived $\C$-scheme to get a derived manifold with the same topological space, and so define a virtual class.

Donaldson and Thomas \cite{DoTh} envisaged using gauge theory to define invariants of Calabi--Yau 4-folds `counting' moduli spaces $\cM^\al_{\SU(4)}$, and also invariants of compact $\Spin(7)$-manifolds `counting' moduli spaces of `$\Spin(7)$-instantons'.

This would require finding suitable compactifications $\oM^\al_{\SU(4)}$ of the moduli spaces $\cM^\al_{\SU(4)}$, and giving them a nice enough geometric structure to define virtual classes, which is a formidably difficult problem in gauge theory in dimensions $>4$. A {\it huge advantage of our approach\/} is that, working in algebraic geometry, with moduli spaces of coherent sheaves rather than vector bundles, {\it we often get compactness of moduli spaces for free}, without doing any work.
\smallskip

Cao and Leung \cite{Cao,CaLe1,CaLe2} also aim to define enumerative invariants of Calabi--Yau 4-folds $Y$, which they call `$DT_4$-invariants', and their ideas overlap with ours. As for our outline in \S\ref{vf38}, their general theory is still rather incomplete, but they prove many partial results, and do computations in examples.

Given a vector bundle moduli space $\cM^\al_{\text{alg-vb}}\cong\cM^\al_{\rm HE}\cong\cM^\al_{\SU(4)}$ in topological spaces, assuming it is compact, and with an orientation (compare Question \ref{vf3quest1}), Cao and Leung \cite[\S 5]{CaLe1} define a virtual class $[\cM^\al_{\SU(4)}]_\virt$ for $\cM^\al_{\SU(4)}$, and contract this with some cohomology classes $\mu(\be)$ (compare Claim \ref{vf3claim2}) to get integer invariants, which they prove are unchanged under deformations of $Y$. All this involves fairly standard material from gauge theory. 

They also discuss the case in which one has a compact moduli space of coherent sheaves $\cM^\al_{\text{coh-sh}}$, which contains the vector bundle moduli space $\cM^\al_{\text{alg-vb}}$ as an open subset. They want to define a virtual class for $\cM^\al_{\text{coh-sh}}$, as we want to, and they can do this under the assumptions that either $\cM^\al_{\text{coh-sh}}$ is smooth, or (in our language) that the $-2$-shifted symplectic derived scheme $(\bs\cM^\al_{\text{coh-sh}},\om^*)$ is locally of the form $T^*\bX[2]$ for $\bX$ a quasi-smooth derived $\C$-scheme.
\smallskip

To compare our work with theirs: given $\cM^\al_{\text{alg-vb}}\subset\cM^\al_{\text{coh-sh}}$ as above, assuming Claim \ref{vf3claim1}, our Theorem \ref{vf3thm4} gives $\cM^\al_{\text{coh-sh}}$ the structure of a derived manifold, but one depending on arbitrary choices. By topologically identifying $\cM^\al_{\text{alg-vb}}\cong\cM^\al_{\SU(4)}$, in effect Cao and Leung make $\cM^\al_{\text{alg-vb}}$ into a derived manifold, {\it canonically up to equivalence\/} (though depending on the K\"ahler metric $g$ and holomorphic volume form $\Om$). However, there seems no reason why their derived manifold structure on $\cM^\al_{\text{alg-vb}}\subset\cM^\al_{\text{coh-sh}}$ should extend smoothly to $\cM^\al_{\text{coh-sh}}$. This is a reason why our approach may in the end be more effective.

\section{Proof of Theorem \ref{vf3thm1}}
\label{vf4}

In this proof we write $\cdga_\C$ for the ordinary category of cdgas over $\C$, and $\cdga_\C^\iy$ for the $\iy$-category of cdgas over $\C$, defined using the model structure on $\cdga_\C$. All objects in $\cdga_\C$ are fibrant. A cdga $A$ is cofibrant if it is a retract of a cdga $A'$ which is {\it almost-free}, that is, free as a graded commutative algebra. If $\phi:A\ra B$ is a morphism in $\cdga_\C$ then $\phi:A\ra B$ is also a morphism in $\cdga_\C^\iy$. However, morphisms $\phi:A\ra B$ in $\cdga_\C^\iy$ may not correspond to morphisms $A\ra B$ in $\cdga_\C$ unless $A$ is cofibrant. 

The spectrum functor $\bSpec$ maps $(\cdga_\C)^{\bf op}\ra\dSch_\C$ and $(\cdga^\iy_\C)^{\bf op}\ra\dSch_\C$, and $(\cdga^\iy_\C)^{\bf op}\ra\dSch_\C$ is an equivalence with the full $\iy$-subcategory of $\dSch_\C$ with affine objects. So, morphisms $\phi:A\ra B$ in $\cdga_\C^\iy$ are essentially the same thing as morphisms $\bSpec B\ra\bSpec A$ in~$\dSch_\C$.

Let $\bs\pi:\bX\ra Z=\Spec B$ and $\bigl\{(A_i^\bu,\bs\al_i,\be_i):i\in I\bigr\}$ be as in Theorem \ref{vf3thm1}. Our task is to construct a standard form cdga $A_J^\bu=(A_J^*,\d)$, a Zariski open inclusion $\bs\al_J:\bSpec A_J^\bu\hookra\bX$, and a morphism $\be_J:B\ra A_J^0$ for all finite $\es\ne J\subseteq I$, and a quasi-free morphism $\Phi_{JK}:A_K^\bu\ra A_J^\bu$ for all finite $\es\ne K\subseteq J\subseteq I$, satisfying conditions. We will do this by induction on increasing $k=\md{J}$. Here is our inductive hypothesis:

\begin{hyp} Let $k=1,2,\ldots$ be given. Then
\begin{itemize}
\setlength{\itemsep}{0pt}
\setlength{\parsep}{0pt}
\item[(a)] We are given finite subsets $S_J^n$ for all $\es\ne J\subseteq I$ with $\md{J}\le k$ and all $n=-1,-2,\ldots.$
\item[(b)] For all $\es\ne J\subseteq I$ with $\md{J}\le k$ we have $A_J^0=\bigot_{i\in J}^{\text{over $B$}}A_i^0$ as a smooth $\C$-algebra of pure dimension, where the tensor products are over $B$ using $\be_i:B\ra A_i^0$ to make $A_i^0$ into a $B$-algebra, so that if $J=\{i_1,\ldots,i_j\}$ then
\e
\smash{A_J^0=A_{i_1}\ot_B A_{i_2}\ot_B\cdots\ot_B A_{i_j}.}
\label{vf4eq1}
\e
The morphism $\be_J:B\ra A_J^0$ is induced by \eq{vf4eq1} and the $\be_i:B\ra A_i^0$ for $i\in J$, and is smooth as the $\be_i$ are.
\item[(c)] For all $\es\ne J\subseteq I$ with $\md{J}\le k$, as a graded $\C$-algebra, $A_J^*$ is freely generated over $A_J^0$ by generators $\coprod_{\es\ne K\subseteq J}S_K^n$ in degree $n$ for~$n=-1,-2,\ldots.$
\item[(d)] For all $\es\ne K\subseteq J\subseteq I$ with $\md{J}\le k$, the morphism $\Phi_{JK}^0:A_K^0\ra A_J^0$ in degree 0 is the morphism 
\begin{equation*}
\smash{A_K^0\!=\!\ts\bigot_{i\in K}^{\text{over $B$}}A_i^0\!=\!\bigl(\bigot_{i\in K}^{\text{over $B$}}A_i^0\bigr)\!\ot_B\!\bigl(\bigot_{i\in J\sm K}^{\text{over $B$}}B\bigr)\longra\bigot_{i\in J}^{\text{over $B$}}A_i^0\!=\!A_J^0}
\end{equation*}
induced by the morphisms $\id:A_i^0\ra A_i^0$ for $i\in K$ and $\be_i:B\ra A_i^0$ for $i\in J\sm K$. Then $\Phi_{JK}:A_K^*\ra A_J^*$ is the unique morphism of graded $\C$-algebras acting by $\Phi_{JK}^0$ in degree zero, and mapping $\Phi_{JK}:\ga\mapsto\ga$ for each $\ga\in S^n_L$ for $\es\ne L\subseteq K\subseteq J\subseteq I$ and $n=-1,-2,\ldots,$ so that $\ga$ is a free generator of both $A_K^*$ over $A_K^0$ and $A_J^*$ over~$A_J^0$.

Note that $\Phi_{JK}^0:A_K^0\ra A_J^0$ is a smooth morphism of $\C$-algebras of pure relative dimension, since $\id:A_i^0\ra A_i^0$ and $\be_i:B\ra A_i^0$ are. Also $\Phi_{JK}$ maps independent generators $\coprod_{\es\ne L\subseteq K}S_L^n$ of $A_K^*$ over $A_K^0$ to independent generators of $A_J^*$ over $A_J^0$. Hence $\Phi_{JK}:A_K^*\ra A_J^*$ is quasi-free.

Clearly $\be_J=\Phi_{JK}^0\ci\be_K=\Phi_{JK}\ci\be_K:B\ra A_J^0$. 

Also, if $\es\ne L\subseteq K\subseteq J\subseteq I$ with $\md{J}\le K$ then clearly $\Phi_{JL}^0=\Phi_{JK}^0\ci\Phi_{KL}^0:A_L^0\ra A_J^0$, and $\Phi_{JL}=\Phi_{JK}\ci\Phi_{KL}:A_L^*\ra A_J^*$.
\item[(e)] For all $\es\ne J\subseteq I$ with $\md{J}\le k$ and all $n=-1,-2,\ldots,$ we are given maps~$\de_J^n:S_J^n\ra A_J^{n+1}$.
\item[(f)] Let $\es\ne J\subseteq I$ with $\md{J}\le k$. Define $\d:A_J^*\ra A_J^{*+1}$ uniquely by the conditions that $\d$ satisfies the Leibnitz rule, and
\e
\smash{\d\ga=\Phi_{JK}\ci\de_K^n(\ga)\quad\text{for all $\es\ne K\subseteq J$, $n\le -1$ and $\ga\in S_K^n$.}}
\label{vf4eq2}
\e
We require that $\d\ci\d=0:A_J^*\ra A_J^{*+2}$, so that $A_J^\bu=(A_J^*,\d)$ is a cdga.

This defines $A_J^\bu=(A_J^*,\d)$, as a standard form cdga over $\C$. Observe that if $\es\ne K\subseteq J\subseteq I$ with $\md{J}\le k$ then as $\Phi_{JK}:A_K^*\ra A_J^*$ is a morphism of graded $\C$-algebras with $\Phi_{JK}\ci\d\ga=\d\ci\Phi_{JK}(\ga)$ for all $\ga$ in the generating sets $\coprod_{\es\ne L\subseteq K}S_L^n$ for $A_K^*$ over $A_K^0$, we have $\Phi_{JK}\ci\d=\d\ci\Phi_{JK}:A_K^*\ra A_J^{*+1}$, and so $\Phi_{JK}:A_K^\bu\ra A_J^\bu$ is a morphism of cdgas.
\item[(g)] For all $\es\ne J\subseteq I$ with $\md{J}\le k$, we are given a Zariski open inclusion $\bs\al_J:\bSpec A_J^\bu\hookra\bX$, with image $\Im\bs\al_J=\bigcap_{i\in J}\Im\bs\al_i$, such that \eq{vf3eq2} homotopy commutes.

If $\es\ne K\subseteq J\subseteq I$ with $\md{J}\le k$ then \eq{vf3eq3} homotopy commutes.
\end{itemize}
\label{vf4hyp1}
\end{hyp}

\begin{rem}{\bf(i)} In Hypothesis \ref{vf4hyp1}, the only actual data required are the finite sets $S_J^n$ in (a), the maps $\de_J^n:S_J^n\ra A_J^{n+1}$ in (e), and the morphisms $\bs\al_J:\bSpec A_J^\bu\hookra\bX$ in (g). 

Also, the only statements requiring proof are that $\d\ci\d=0$ in (f), and that $\bs\al_J$ is a Zariski open inclusion with image $\bigcap_{i\in J}\Im\bs\al_i$, and that \eq{vf3eq2} and \eq{vf3eq3} homotopy commute in (g). All of (b),(c),(d) are definitions and deductions.

\smallskip

\noindent{\bf(ii)} Most of the conclusions of Theorem \ref{vf3thm1} are immediate from the definitions in (a)--(g): that $A_J^\bu$ is a standard form cdga, and $\be_J:B\ra A_J^0$ is smooth, and $\Phi_{JK}:A_K^\bu\ra A_J^\bu$ is quasi-free, and $\be_J=\Phi_{JK}\ci\be_K$, and $\Phi_{JL}=\Phi_{JK}\ci\Phi_{KL}$.
\label{vf4rem1}
\end{rem}

For the first step in the induction, we prove Hypothesis \ref{vf4hyp1} when $k=1$. Then the only subsets $\es\ne J\subseteq I$ with $\md{J}\le k$ are $J=\{i\}$ for $i\in I$, and the only subsets $\es\ne K\subseteq J\subseteq I$ with $\md{J}\le k$ are $J=K=\{i\}$ for $i\in I$.

As in Theorem \ref{vf3thm1} we are given data $\bigl\{(A_i^\bu,\bs\al_i,\be_i):i\in I\bigr\}$, where $A_i^\bu$ is a standard form cdga, so that $A_i^*$ is freely generated over $A_i^0$ by finitely many generators in each degree $n=-1,-2,\ldots,$ as in Definition \ref{vf2def1}. For each $i\in I$ and each $n=-1,-2,\ldots,$ choose a subset $S_{\{i\}}^n\subseteq A_i^n$, as in part (a) for $J=\{i\}$, such that $A_i^*$ is freely generated over $A_i^0$ by $\coprod_{n\le -1}S_{\{i\}}^n$. Set $A_{\{i\}}^\bu=A_i^\bu$ and $\be_{\{i\}}=\be_i$, so that parts (b),(c) hold for~$J=\{i\}$.

Part (d) is a definition, and when $k=1$ only says that when $J=K=\{i\}$ we have $\Phi_{\{i\}\{i\}}=\id:A_{\{i\}}^\bu\ra A_{\{i\}}^\bu$. For (e), define $\de_{\{i\}}^n:S_{\{i\}}^n\ra A_{\{i\}}^{n+1}=A_i^{n+1}$ by $\de_{\{i\}}^n(\ga)=\d\ga$, using $\d$ in the cdga $A_i^\bu=(A_i^*,\d)$. Given (e), part (f) says that the differentials $\d$ in $A_{\{i\}}^\bu=(A_{\{i\}}^*,\d)$ and $A_i^\bu=(A_i^*,\d)$ agree, consistent with setting $A_{\{i\}}^\bu=A_i^\bu$, so $\d\ci\d=0$ in $A_{\{i\}}^\bu$ as $A_i^\bu$ is a cdga.

For (g), if $i\in I$ define $\bs\al_{\{i\}}=\bs\al_i:A_{\{i\}}^\bu=A_i^\bu\ra\bX$. Then the assumptions on $\bigl\{(A_i^\bu,\bs\al_i,\be_i):i\in I\bigr\}$ in Theorem \ref{vf3thm1} imply that $\bs\al_{\{i\}}$ is a Zariski open inclusion, with image $\Im\bs\al_{\{i\}}=\Im\bs\al_i$, and \eq{vf3eq2} homotopy commutes for $J=\{i\}$ as \eq{vf3eq1} does. The only $\es\ne K\subseteq J\subseteq I$ with $\md{J}\le k=1$ are $J=K=\{i\}$, and then \eq{vf3eq3} homotopy commutes as $\bs\al_J=\bs\al_K=\bs\al_{\{i\}}$ and $\Phi_{JK}=\id$. This completes Hypothesis \ref{vf4hyp1} when $k=1$. Note that our definitions $A_{\{i\}}^\bu=A_i^\bu,$ $\bs\al_{\{i\}}=\bs\al_i,$ and $\be_{\{i\}}=\be_i$ for $i\in I$ are as required in Theorem~\ref{vf3thm1}(i).

Next we prove the inductive step. Let $l\ge 1$ be given, and suppose Hypothesis \ref{vf4hyp1} holds with $k=l$. Keeping all the data in (a),(e),(g) for $\md{J}\le l$ the same, we will prove Hypothesis \ref{vf4hyp1} with $k=l+1$. To do this, for each $J\subseteq I$ with $\md{J}=l+1$, we have to construct the data of finite sets $S_J^n$ for $n=-1,-2,\ldots$ in (a), and maps $\de_J^n:S_J^n\ra A_J^{n+1}$ in (e), and the morphism $\bs\al_J:\bSpec A_J^\bu\hookra\bX$ in (g), and then prove the claims in (f) that $\d\ci\d=0$, and in (g) that $\bs\al_J$ is a Zariski open inclusion with image $\bigcap_{i\in J}\Im\bs\al_i$, and that \eq{vf3eq2} and \eq{vf3eq3} homotopy commute. 

Note that as Hypothesis \ref{vf4hyp1} involves no compatibility conditions between data for distinct $J,J'\subseteq I$ with $\md{J}=\md{J'}=k$, we can do this independently for each $J\subseteq I$ with $\md{J}=l+1$, that is, it is enough to give the proof for a single such $J$. So fix a subset $J\subseteq I$ with~$\md{J}=l+1$.

We first define a standard form cdga $\ti A_J^\bu$ which is an approximation to the cdga $A_J^\bu$ that we want, and morphisms $\ti\be_J:B\ra\ti A_J^0$, $\ti\Phi_{JK}:A_K^\bu\ra\ti A_J^\bu$ for all $\es\ne K\subsetneq J$, so that $\md{K}\le l$ and $A_K^\bu$ is already defined:
\begin{itemize}
\setlength{\itemsep}{0pt}
\setlength{\parsep}{0pt}
\item Define $\ti A_J^0=A_J^0$ and $\ti\be_J=\be_J:B\ra \ti A_J^0=A_J^0$ as in Hypothesis \ref{vf4hyp1}(b).
\item Define $\ti A_J^*$ to be the graded $\C$-algebra freely generated over $A_J^0$ by generators $\coprod_{\es\ne K\subsetneq J}S_K^n$ in degree $n$ for $n=-1,-2,\ldots.$ This is the same as for $A_J^*$ in Hypothesis \ref{vf4hyp1}(c), except that we do not include generators $S_J^n$, since $S_J^n$ is not yet defined.
\item If $\es\ne K\subsetneq J$, so that $A_K^\bu$ is defined, define $\Phi_{JK}^0:A_K^0\ra A_J^0=\ti A_J^0$ as in Hypothesis \ref{vf4hyp1}(d), and define $\ti\Phi_{JK}:A_K^*\ra\ti A_J^*$ to be the unique morphism of graded $\C$-algebras acting by $\Phi_{JK}^0$ in degree zero, and mapping $\Phi_{JK}:\ga\mapsto\ga$ for each $\ga\in S^n_L$ for $\es\ne L\subseteq K$ and $n=-1,-2,\ldots.$
\item The differential $\d:\ti A_J^*\ra\ti A_J^{*+1}$ in the cdga $\ti A_J^\bu=(\ti A_J^*,\d)$ is determined uniquely as in \eq{vf4eq2} by
\begin{equation*}
\smash{\d\ga=\ti\Phi_{JK}\ci\de_K^n(\ga)\quad\text{for all $\es\ne K\subsetneq J$, $n\le -1$ and $\ga\in S_K^n$.}}
\end{equation*}
Then $\ti\Phi_{JK}:A_K^\bu\ra\ti A_J^\bu$ is a cdga morphism for all $\es\ne K\subsetneq J$, as in Hypothesis \ref{vf4hyp1}(f) for~$\Phi_{JK}$.
\end{itemize}

That is, $\ti A_J^\bu$ is the colimit in the ordinary category $\cdga_\C$ of the commutative diagram $\Ga$ with vertices the objects $B$ and $A_K^\bu$ for all $K$ with $\es\ne K\subsetneq J$, and edges the morphisms $\be_K:B\ra A_K^\bu$ and $\Phi_{K_1K_2}:A_{K_2}^\bu\ra A_{K_1}^\bu$ for $\es\ne K_2\subsetneq K_1\subsetneq J$, and $\ti\be_J:B\ra\ti A_J^\bu$, $\ti\Phi_{JK}:A_K^\bu\ra\ti A_J^\bu$ are the projections to the colimit. Since all the morphisms in $\Ga$ are almost-free in negative degrees and smooth in degree $0$, these morphisms are sufficiently cofibrant to compute the homotopy colimits as well. Indeed, having such a morphism $A^\bu\rightarrow C^\bu$ we can factor it into $A^\bu\rightarrow A^\bu\underset{A^0}\otimes C^0\rightarrow C^\bu$. Each one of these morphisms is flat, and hence homotopy pullbacks can be computed without resolving. Finally we notice that the colimit of the entire diagram $\Ga$ can be calculated as a sequence of pullbacks. So $\ti A_J^\bu$ is also the homotopy colimit of $\Ga$ in the $\iy$-category $\cdga_\C^\iy$. Hence $\bSpec\ti A_J^\bu$ is the homotopy limit of $\bSpec\Ga$ in the $\iy$-category~$\dSch_\C$.

For $\es\ne K\subsetneq J$, consider $\bigcap_{i\in K}\Im\bs\al_i$ as an open derived $\C$-subscheme of $\bX$. Then by Hypothesis \ref{vf4hyp1}(g), $\bs\al_K:\bSpec A_K^\bu\ra \bigcap_{i\in K}\Im\bs\al_i$ is an equivalence in $\dSch_\C$. We also have the open derived $\C$-subscheme $\bigcap_{i\in J}\Im\bs\al_i$ in $\bX$, which is affine by Definition \ref{vf2def4} as $\bX$ has affine diagonal and $\Im\bs\al_i\simeq\bSpec A_i^\bu$ is affine for $i\in J$. Thus we may choose a standard form cdga $\hat A_J^\bu$ and an equivalence $\bs{\hat\al_J}:\bSpec\hat A_J^\bu\hookra\bigcap_{i\in J}\Im\bs\al_i$.

Define morphisms $\bs{\hat\be}_J:\bSpec\hat A_J^\bu\ra Z=\Spec B$ by $\bs\be_J=\bs\pi\ci\bs{\hat\al_J}$, and $\bs{\hat\phi}_{JK}:\bSpec\hat A_J^\bu\ra\bSpec A_K^\bu$ for $\es\ne K\subsetneq J$ to be the composition
\begin{equation*}
\smash{\xymatrix@C=30pt{ \bSpec\hat A_J^\bu \ar[r]^(0.45){\bs{\hat\al}_J} & \bigcap_{i\in J}\Im\bs\al_i\,
\ar@{^{(}->}[r] & \,\bigcap_{i\in K}\Im\bs\al_i \ar[r]^{\bs\al_K^{-1}} & \bSpec A_K^\bu, }}
\end{equation*}
where $\bs\al_K^{-1}$ is a quasi-inverse for the equivalence $\bs\al_K:\bSpec A_K^\bu\ra \bigcap_{i\in K}\Im\bs\al_i$. 

Now by the homotopy limit property of $\bSpec\ti A_J^\bu$, there exists a morphism $\bs\psi:\bSpec\hat A_J^\bu\ra\bSpec\ti A_J^\bu$ in $\dSch_\C$ unique up to homotopy, with homotopies $\bs{\hat\be}_J\simeq\bSpec\ti\be_J\ci\bs\psi$ and $\bs{\hat\phi}_{JK}\simeq\bSpec\ti\Phi_{JK}\ci\bs\psi$ for $\es\ne K\subsetneq J$. We can then write $\bs\psi\simeq \bSpec\Psi$ for $\Psi:\ti A_J^\bu\ra\hat A_J^\bu$ a morphism in $\cdga_\C^\iy$, unique up to homotopy. However, we do not yet know that $\Psi$ descends to a morphism in $\cdga_\C$. The definitions of $\bs{\hat\be}_J,\bs{\hat\phi}_{JK}$ and $\bs\psi\simeq \bSpec\Psi$ give homotopies
\e
\begin{split}
\bs\pi\ci\bs{\hat\al_J}\simeq \bSpec\ti\be_J\ci\bSpec\Psi&:\bSpec\hat A_J^\bu\longra Z,\\
\bs{\hat\al}_J\simeq \bs\al_K\ci \bSpec\ti\Phi_{JK}\ci\bSpec\Psi&:\bSpec\hat A_J^\bu\longra\bX,\;\> \es\ne K\subsetneq J.
\end{split}
\label{vf4eq3}
\e 

Consider the composition of morphisms of classical $\C$-algebras
\e
\smash{\xymatrix@C=30pt{ A_J^0 \ar@{=}[r] & \ti A_J^0 \ar[r] & H^0(\ti A_J^\bu) \ar[rr]^{H^0(\Psi)} && H^0(\hat A_J^\bu). }}
\label{vf4eq4}
\e
Here $\Spec H^0(\Psi)$ is the natural morphism
\e
\smash{\Spec H^0(\Psi):X_J \longra \ts\prod\nolimits^{\text{over $Z$}}_{\es\ne K\subsetneq J}X_K,}
\label{vf4eq5}
\e
writing $X_K$ for the open $\C$-subscheme $\bigcap_{k\in K}t_0(\Im\bs\al_k)$ in $X$. This is the restriction of the multidiagonal $\De_X^{2^{\md{J}-2}}:X\ra X\t_ZX\t_Z\cdots\t_ZX$, with $2^{\md{J}-2}$ copies of $X$ on the right. As $X$ is separated, $\De_X^2:X\ra X\t_ZX$ is a closed immersion, so $\De_X^{2^{\md{J}-2}}$ is a closed immersion. Also the domain $X_J$ of \eq{vf4eq5} is the preimage under $\De_X^{2^{\md{J}-2}}$ of the target, since $X_J=\bigcap_{\es\ne K\subsetneq J}X_K$ as~$\md{J}\ge 2$. 

Hence \eq{vf4eq5} is a closed immersion, so $H^0(\Psi)$ in \eq{vf4eq4} is surjective. Also $\ti A_J^0\ra H^0(\ti A_J^\bu)$ is surjective, so the composition \eq{vf4eq4} is surjective. Therefore we can replace $\hat A_J^\bu$ by an equivalent object in $\cdga_\C^\iy$, such that $\hat A_J^0=\ti A_J^0$, and the following homotopy commutes in $\cdga_\C^\iy$:
\e
\begin{gathered}
\xymatrix@C=90pt@R=11pt{ *+[r]{\ti A_J^0} \ar@{=}[r] \ar[d] & *+[l]{\hat A_J^0} \ar[d] \\
*+[r]{\ti A_J^\bu} \ar[r]^\Psi & *+[l]{\hat A_J^\bu.\!\!} }
\end{gathered}
\label{vf4eq6}
\e

Now $\Psi:\ti A_J^\bu\ra\hat A_J^\bu$ is a morphism in $\cdga_\C^\iy$. For this to descend to a morphism in $\cdga_\C$, the simplest condition is that $\ti A_J^\bu$ should be cofibrant and $\hat A_J^\bu$ fibrant in the model category $\cdga_\C$. Here $\hat A_J^\bu$ is fibrant, as all objects are, but $\ti A_J^\bu$ may not be cofibrant, i.e.\ a retract of an almost-free cdga. However, $\ti A_J^\bu$ is cofibrant as an $\ti A_J^0$-algebra, as it is free in negative degrees, and \eq{vf4eq6} says that $\Psi$ does descend to a morphism in $\cdga_\C$ in degree 0. Together these imply that $\Psi$ descends to a morphism $\Psi:\ti A_J^\bu\ra\hat A_J^\bu$ in $\cdga_\C$.

Next, by induction on decreasing $n=-1,-2,\ldots$ we will choose the data  $S_J^n,\de_J^n$ in Hypothesis \ref{vf4hyp1}(a),(e). Here is our inductive hypothesis:

\begin{hyp} Let $N=0,-1,-2,\ldots$ be given. Then:
\begin{itemize}
\setlength{\itemsep}{0pt}
\setlength{\parsep}{0pt}
\item[(a)] We are given finite subsets $S_J^n$ for $n=-1,-2,\ldots,N$. Write $A_{J,N}^*=\ti A_J^*[S_J^1,\ldots,S_J^N]$ for the graded $\C$-algebra freely generated over $\ti A_J^*$ by the sets of extra generators $S_J^n$ in degree $n$ for all $n=-1,-2,\ldots,N$.
\item[(b)] We are given maps $\de_J^n:S_J^n\ra A_{J,N}^{n+1}$ for $n=-1,-2,\ldots,N$. Define $\d:A_{J,N}^*\ra A_{J,N}^{*+1}$ uniquely by the conditions that $\d$ satisfies the Leibnitz rule, and $\d$ is as in $\ti A_J^\bu=(\ti A_J^*,\d)$ on $\ti A_J^*\subseteq A_{J,N}^*$, and on the extra generators $\ga\in S_J^n$ for $n=-1,-2,\ldots,N$, we have $\d\ga=\de_J^n(\ga)\in A_{J,N}^{n+1}$. We require that $\d\ci\d=0:A_{J,N}^*\ra A_{J,N}^{*+2}$, so that $A_{J,N}^\bu=(A_{J,N}^*,\d)$ is a cdga. 
\item[(c)] We are given maps $\xi_J^n:S_J^n\ra\hat A_J^n$ for $n=-1,-2,\ldots,N$.

Define $\Xi_N:A_{J,N}^*\ra\hat A_J^*$ to be the morphism of graded $\C$-algebras such that $\Xi_N=\Psi$ on $\ti A_J^*\subseteq A_{J,N}^*$, and on the extra generators $\ga\in S_J^n$ for $n=-1,-2,\ldots,N$, we have $\Xi_N(\ga)=\xi_J^n(\ga)\in\hat A_{J,N}^n$.

We require that $\Xi_N\ci\d=\d\ci\Xi_N:A_{J,N}^*\ra\hat A_J^{*+1}$, so that $\Xi_N:A_{J,N}^\bu\ra\hat A_J^\bu$ is a cdga morphism.

We also require that $H^n(\Xi_N):H^n(A_{J,N}^\bu)\ra H^n(\hat A_J^\bu)$ should be an isomorphism for $n=0,-1,-2,\ldots,N+1$, and surjective for~$n=N$.
\end{itemize}

\label{vf4hyp2}
\end{hyp}

For the first step $N=0$, there is no data $S_J^n,\de_J^n,\xi_J^n$, and $A_{J,0}^\bu=\ti A_J^\bu$, and $\Xi_0=\Psi$, and the only thing to prove is that $H^0(\Psi):H^0(\ti A_J^\bu)\ra H^0(\hat A_J^\bu)$ is surjective, which holds as $\Psi^0=\id:\ti A_J^0\ra\ti A_J^0=\hat A_J^0$ from above. So Hypothesis \ref{vf4hyp2} holds for~$N=0$. 

For the inductive step, let $m=0,-1,-2,\ldots$ be given, and suppose Hypothesis \ref{vf4hyp2} holds with $N=m$. Keeping all the data $S_J^n,\de_J^n,\xi_J^n$ for $n=-1,\ldots,m$ the same, we will prove Hypothesis \ref{vf4hyp2} with $N=m-1$. Note that with $S_J^{-1},\ldots,S_J^m$ the same, the graded $\C$-algebras $A_{J,m}^*,A_{J,m-1}^*$ agree in degrees $0,-1,\ldots,m$, so it makes sense to say that $\de_J^n:S_J^n\ra A_{J,m}^{n+1}$ and $\de_J^n:S_J^n\ra A_{J,m-1}^{n+1}$ are equal for $n=-1,-2,\ldots,m$. We must choose data $S_J^{m-1}$, $\de_J^{m-1}:S_J^{m-1}\ra A_{J,m-1}^m$ and $\xi_J^{m-1}:S_J^{m-1}\ra\hat A_J^{m-1}$, and verify the last two conditions of Hypothesis~\ref{vf4hyp2}(c).

Choose a finite subset $\dot S_J^{m-1}$ of $\Ker\bigl(H^m(\Xi_m):H^m(A_{J,m}^\bu)\ra H^m(\hat A_J^\bu)\bigr)$ which generates $\Ker(\cdots)$ as an $H^0(A_{J,m}^\bu)$-module, and a finite subset $\ddot S_J^{m-1}$ of $H^{m-1}(\hat A_J^\bu)$ such that $\ddot S_J^{m-1}$ and $\Im\bigl(H^{m-1}(\Xi_m):H^{m-1}(A_{J,m}^\bu)\ra H^{m-1}(\hat A_J^\bu)\bigr)$ generate $H^{m-1}(\hat A_J^\bu)$ as an $H^0(\hat A_J^\bu)$-module. Finite subsets suffice in each case since $A_{J,m}^\bu,\hat A_J^\bu$ are of standard form, so that $H^n(A_{J,m}^\bu),H^n(\hat A_J^\bu)$ are finitely generated over $H^0(A_{J,m}^\bu),H^0(\hat A_J^\bu)$ for all $n$. Set~$S_J^{m-1}=\dot S_J^{m-1}\amalg \ddot S_J^{m-1}$.

Then Hypothesis \ref{vf4hyp2}(a) defines $A_{J,m-1}^*$ as a graded $\C$-algebra, with $A_{J,m-1}^n\ab=A_{J,m}^n$ in degrees $n\ge m$. For all $\ga\in\dot S_J^{m-1}$, choose a representative $\de_J^{m-1}(\ga)$ in $A_{J,m-1}^m=A_{J,m}^m$ for the cohomology class $\ga$ in $H^m(A_{J,m}^\bu)$, so that $\d(\de_J^{m-1}(\ga))=0$ in $A_{J,m}^{m+1}$. Define $\de_J^{m-1}(\ga)=0$ in $A_{J,m-1}^m$ for all $\ga\in \ddot S_J^{m-1}$. This defines $\de_J^{m-1}:S_J^{m-1}\ra A_{J,m-1}^m$ in Hypothesis \ref{vf4hyp2}(b), and hence $\d:A_{J,m-1}^*\ra A_{J,m-1}^{*+1}$.

To see that $\d\ci\d=0:A_{J,m-1}^*\ra A_{J,m-1}^{*+2}$, note that $A_{J,m-1}^*=A_{J,m}^*[S_J^{m-1}]$, so $\d$ on $A_{J,m-1}^*$ is determined by $\d$ on $A_{J,m}^*$, which already satisfies $\d\ci\d=0$ by induction, and $\d$ on the extra generators $S_J^{m-1}$, which satisfy $\d\ci\d=0$ as for $\ga\in\dot S_J^{m-1}$ we have $\d\ci\d\ga=\d(\de_J^{m-1}(\ga))=0$, and for $\ga\in\ddot S_J^{m-1}$ we have $\d\ga=0$ so $\d\ci\d\ga=0$. Hence $A_{J,m-1}^\bu=(A_{J,m-1}^*,\d)$ is a cdga, as we have to prove.

For all $\ga\in\dot S_J^{m-1}$, as $\de_J^{m-1}(\ga)\in A_{J,m}^m$ represents a cohomology class in $\Ker\bigl(H^m(\Xi_m):H^m(A_{J,m}^\bu)\ra H^m(\hat A_J^\bu)\bigr)$, we see that $\Xi_m\ci\de_J^{m-1}(\ga)$ is exact in $\hat A_J^\bu$, so we can choose an element $\xi_J^{m-1}(\ga)\in\hat A_J^{m-1}$ with $\d\ci\xi_J^{m-1}(\ga)=\Xi_m\ci \de_J^{m-1}(\ga)$. For all $\ga\in\ddot S_J^{m-1}\subset H^{m-1}(\hat A_J^\bu)$, choose an element $\xi_J^{m-1}(\ga)\in\hat A_J^{m-1}$ representing $\ga$, so that $\d\ci\xi_J^{m-1}(\ga)=0$. This defines~$\xi_J^{m-1}:S_J^{m-1}\ra\hat A_J^{m-1}$.

Hypothesis \ref{vf4hyp2}(c) now defines $\Xi_{m-1}:A_{J,m-1}^*\ra\hat A_J^*$. To prove that $\Xi_{m-1}\ci\d=\d\ci\Xi_{m-1}$, note that $A_{J,m-1}^*=A_{J,m}^*[S_J^{m-1}]$, and on $A_{J,m}^*\subseteq A_{J,m-1}^*$ we have $\Xi_{m-1}=\Xi_m$, and $\Xi_m\ci\d=\d\ci\Xi_m$ by induction. So it is enough to prove that $\Xi_{m-1}\ci\d(\ga)=\d\ci\Xi_{m-1}(\ga)$ for all $\ga\in S_J^{m-1}$. If $\ga\in\dot S_J^{m-1}$ then
\begin{equation*}
\smash{\Xi_{m-1}\ci\d(\ga)\!=\!\Xi_{m-1}\ci\de_J^{m-1}(\ga)\!=\!\Xi_m\ci\de_J^{m-1}(\ga)\!=\!\d\ci\xi_J^{m-1}(\ga)\!=\!\d\ci\Xi_{m-1}(\ga),}
\end{equation*}
as we want. Similarly, if $\ga\in\ddot S_J^{m-1}$ then
\begin{equation*}
\smash{\Xi_{m-1}\ci\d(\ga)=\Xi_{m-1}\ci\de_J^{m-1}(\ga)=0=\d\ci\xi_J^{m-1}(\ga)=\d\ci\Xi_{m-1}(\ga).}
\end{equation*}
Therefore $\Xi_{m-1}\ci\d=\d\ci\Xi_{m-1}$, and $\Xi_{m-1}:A_{J,m-1}^\bu\ra\hat A_J^\bu$ is a cdga morphism.

Finally we have to show that $H^n(\Xi_{m-1}):H^n(A_{J,m-1}^\bu)\ra H^n(\hat A_J^\bu)$ is be an isomorphism for $n=-1,-2,\ldots,m$, and surjective for $n=m-1$. Since $\Xi_m:A_{J,m}^\bu\ra\hat A_J^\bu$ and $\Xi_{m-1}:A_{J,m-1}^\bu\ra\hat A_J^\bu$ coincide in degrees $0,-1,\ldots,m$, in cohomology they coincide in degrees $0,-1,\ldots,m+1$, so $H^n(\Xi_{m-1})$ is an isomorphism for $n=0,-1,\ldots,m+1$ as $H^n(\Xi_m)$ is, by induction.

As $H^m(\Xi_m):H^m(A_{J,m}^\bu)\ra H^m(\hat A_J^\bu)$ is surjective, and the added generators $\dot S_J^{m-1}$ in $A_{J,m-1}^\bu$ span $\Ker(H^m(\Xi_m))$, adding $\dot S_J^{m-1}$ makes $H^m(\Xi_{m-1}):H^m(A_{J,m-1}^\bu)\ra H^m(\hat A_J^\bu)$ into an isomorphism. Also, since the added generators $\ddot S_J^{m-1}$ together with $\Im(H^{m-1}(\Xi_m))$ generate $H^{m-1}(\hat A_J^\bu)$, adding $\ddot S_J^{m-1}$ makes $H^{m-1}(\Xi_{m-1}):H^{m-1}(A_{J,m-1}^\bu)\ra H^{m-1}(\hat A_J^\bu)$ surjective.

This proves Hypothesis \ref{vf4hyp2} for $N=m-1$, so by induction Hypothesis \ref{vf4hyp2} holds for all $N=0,-1,-2,\ldots.$ Taking the limit $\lim_{N\ra -\iy}A_{J,N}^\bu$ gives the cdga $A_J^\bu$ defined in Hypothesis \ref{vf4hyp1} using the data $S_J^n,\de_J^n$ for all $n=-1,-2,\ldots$ from Hypothesis \ref{vf4hyp2}(a),(b) as $N\ra-\iy$. The data $\xi_J^n$ for $n=-1,-2,\ldots$ from part (c) defines a morphism $\Xi=\lim_{N\ra -\iy}\Xi_N:A_J^\bu\ra \hat A_J^\bu$, where $\Xi,A_J^\bu$ agree with $\Xi_N,A_{J,N}^\bu$ in degrees $0,-1,\ldots,N$ for all $N\le 0$. 

Hence $H^n(\Xi):H^n(A_J^\bu)\ra H^n(\hat A_J^\bu)$ agrees with $H^n(\Xi_N):H^n(A_{J,N}^\bu)\ra H^n(\hat A_J^\bu)$ for all $n=0,-1,\ldots,N+1$, so $H^n(\Xi)$ is an isomorphism for all $n\le 0$ by Hypothesis \ref{vf4hyp2}(c), and $\Xi:A_J^\bu\ra \hat A_J^\bu$ is a quasi-isomorphism in $\cdga_\C$, and hence an equivalence in $\cdga_\C^\iy$. Thus $\bSpec\Xi:\bSpec\hat A_J^\bu\ra\bSpec A_J^\bu$ is an equivalence in $\dSch_\C$. So we can choose a quasi-inverse $\bs\chi:\bSpec A_J^\bu\ra\bSpec\hat A_J^\bu$ in~$\dSch_\C$.

Write $\io:\ti A_J^\bu\hookra A_J^\bu$ for the inclusion. Then $\Psi=\Xi\ci\io:\ti A_J^\bu\ra\hat A_J^\bu$, since $\Xi_N\vert_{\ti A_J^\bu}=\Psi$, so taking the limit $N\ra -\iy$ gives $\Xi\vert_{\ti A_J^\bu}=\Psi$. Also the definitions of $\be_J:B\ra A_J^\bu$ and $\Phi_{JK}:A_K^\bu\ra A_J^\bu$ for $\es\ne K\subsetneq J$ in Hypothesis \ref{vf4hyp1}(b),(d) satisfy $\be_J=\io\ci\ti\be_J$ and~$\Phi_{JK}=\io\ci\ti\Phi_{JK}$.

Define $\bs\al_J=\bs{\hat\al}_J\ci\bs\chi:\bSpec A_J^\bu\ra\bX$. Since $\bs{\hat\al}_J$ is a Zariski open inclusion with image $\bigcap_{i\in J}\Im\bs\al_i$, and $\bs\chi$ is an equivalence, $\bs\al_J:\bSpec A_J^\bu\ra\bX$ is a Zariski open inclusion with image $\bigcap_{i\in J}\Im\bs\al_i$, as in Hypothesis \ref{vf4hyp1}(g). Then we have
\begin{align*}
\bs\pi\ci\bs\al_J&=\bs\pi\ci\bs{\hat\al}_J\ci\bs\chi\simeq \bSpec\ti\be_J\ci\bSpec\Psi\ci\bs\chi\\
&\simeq\bSpec\ti\be_J\ci\bSpec\io\ci\bSpec\Xi\ci\bs\chi\simeq\bSpec\ti\be_J\ci\bSpec\io=\bSpec\be_J,
\end{align*}
using \eq{vf4eq3} in the second step, $\Psi=\Xi\ci\io$ in the third, $\bSpec\Xi,\bs\chi$ quasi-inverse in the fourth, and $\be_J=\io\ci\ti\be_J$ in the fifth. Thus \eq{vf3eq2} homotopy commutes. 

Similarly, if $\es\ne K\subsetneq J$ then
\begin{align*}
\bs\al_J&=\bs{\hat\al}_J\ci\bs\chi\simeq\bs\al_K\ci\bSpec\ti\Phi_{JK}\ci\bSpec\Psi\ci\bs\chi\\
&\simeq\bs\al_K\ci\bSpec\ti\Phi_{JK}\ci\bSpec\io\ci\bSpec\Xi\ci\bs\chi\simeq\bs\al_K\ci\bSpec\Phi_{JK}
\end{align*}
using \eq{vf4eq3} in the second step, $\Psi=\Xi\ci\io$ in the third, and $\Phi_{JK}=\io\ci\ti\Phi_{JK}$ and $\bSpec\Xi,\bs\chi$ quasi-inverse in the fourth. Hence \eq{vf3eq3} homotopy commutes. 

This proves that Hypothesis \ref{vf4hyp1} holds with $k=l+1$, and completes the inductive step begun shortly after Remark \ref{vf4rem1}. Hence by induction, Hypothesis \ref{vf4hyp1} holds for all $k=1,2,\ldots,$ so Hypothesis \ref{vf4hyp1} holds for $k=\iy$. Theorem \ref{vf3thm1} follows, since all the conclusions of Theorem \ref{vf3thm1}(i),(ii) are either part of Hypothesis \ref{vf4hyp1}, or for $A_{\{i\}}^\bu=A_i^\bu,$ $\bs\al_{\{i\}}=\bs\al_i,$ $\be_{\{i\}}=\be_i$ in part (i) were included in the first step of the induction. This completes the proof.

\section{Proof of Theorem \ref{vf3thm2}}
\label{vf5}

\subsection{\texorpdfstring{Theorem \ref{vf3thm2}(a): $(*)$ is an open condition}{Theorem \ref{vf3thm2}(a): (*) is an open condition}}
\label{vf51}

Suppose $\bX,\om_\bX^*,A^\bu,\bs\al,V,E,F,s,t,\psi$ are as in Definition \ref{vf3def3}, and suppose that $U\subseteq V$ is open, $E^-$ is a real vector subbundle of $E\vert_U$, and $v\in s^{-1}(0)\cap U$, such that the assumptions on $E^-\vert_v$ in condition $(*)$ hold at $v$. We must show that these assumptions also hold for all $v'$ in an open neighbourhood of $v$ in $s^{-1}(0)\cap U$. Suppose for a contradiction that this is false. Then we can choose a sequence $(v_i)_{i=1}^\iy$ in $s^{-1}(0)\cap U$ such that $v_i\ra v$ as $i\ra\iy$, and the assumptions on $E^-\vert_{v_i}$ in $(*)$ do not hold for any~$i=1,2,\ldots.$ 

By passing to a subsequence of $(v_i)_{i=1}^\iy$, we can assume that $\dim\Im\d s\vert_{v_i}$ and $\dim\Ker t\vert_{v_i}$ are independent of $i=1,2,\ldots.$ By trivializing $E$ near $v$, we can regard $\bigl(\Im\d s\vert_{v_i}\bigr){}_{i=1}^\iy$ and $\bigl(\Ker t\vert_{v_i}\bigr){}_{i=1}^\iy$ as sequences in complex Grassmannians, which are compact. Thus, passing to a subsequence of $(v_i)_{i=1}^\iy$, we can assume they converge, and there are complex vector subspaces $I_v,K_v\subseteq E\vert_v$ such that $\Im\d s\vert_{v_i}\ra I_v$ and $\Ker t\vert_{v_i}\ra K_v$ as $i\ra\iy$. 

As $t\ci\d s=0$ on $s^{-1}(0)$ we have $\Im\d s\vert_{v_i}\subseteq \Ker t\vert_{v_i}$, and so $I_v\subseteq K_v$. Also $\Im\d s\vert_v\subseteq I_v$, since if $w\in T_vV$ we can find $w_i\in T_{v_i}V$ with $w_i\ra w$ as $i\ra \iy$, and then $\d s\vert_{v_i}(w_i)\ra\d s\vert_v(w)$ as $i\ra\iy$. Similarly $K_v\subseteq \Ker t\vert_v$.

We now have a quotient vector space $(\Ker t\vert_v)/(\Im\d s\vert_v)$, which as in \eq{vf3eq21} carries a nondegenerate quadratic form $\ti Q_v$. There are subspaces
$I_v/(\Im\d s\vert_v)\subseteq K_v/(\Im\d s\vert_v)\subseteq(\Ker t\vert_v)/(\Im\d s\vert_v)$. Also, for each $i=1,2,\ldots$ we have quotient space $(\Ker t\vert_{v_i})/(\Im\d s\vert_{v_i})$ with quadratic forms $\ti Q_{v_i}$. As $i\ra\iy$ we have
\e
\smash{(\Ker t\vert_{v_i})/(\Im\d s\vert_{v_i})\longra K_v/I_v\cong \bigl[K_v/(\Im\d s\vert_v)\bigr]\big/\bigl[I_v/(\Im\d s\vert_v)\bigr].}
\label{vf5eq1}
\e

One can prove using a representative $\om_{A^\bu}$ for $\bs\al^*(\om_\bX^0)$ that
\begin{equation*}
\smash{I_v/(\Im\d s\vert_v)=\bigl\{e\in (\Ker t\vert_v)/(\Im\d s\vert_v):\ti Q_v(e,k)=0\;\>\forall k\in K_v/(\Im\d s\vert_v)\bigr\},}
\end{equation*}
that is, $I_v/(\Im\d s\vert_v)$ and $K_v/(\Im\d s\vert_v)$ are orthogonal subspaces w.r.t.\ $\ti Q_v$. Hence the restriction of $\ti Q_v$ to $K_v/(\Im\d s\vert_v)$ is null along $I_v/(\Im\d s\vert_v)$, and descends to a nondegenerate quadratic form $\check Q_v$ on
$[K_v/(\Im\d s\vert_v)]/[I_v/(\Im\d s\vert_v)]\ab\cong K_v/I_v$. Then under the limit \eq{vf5eq1}, we have $\ti Q_{v_i}\ra \check Q_v$ as~$i\ra\iy$.

By $(*)$ for $(U,E^-)$ at $v$, we have $\Im(\d s\vert_v)\cap E^-\vert_v=\{0\}$, and the map $\Pi_v:E^-\vert_v\cap\Ker(t\vert_v)\ra (\Ker t\vert_v)/(\Im\d s\vert_v)$ in \eq{vf3eq24} has image $\Im\Pi_v$ of half the total dimension, with $\Re\ti Q_v$ negative definite on $\Im\Pi_v$. Since $\ti Q_v$ is zero on $I_v/(\Im\d s\vert_v)$, it follows that $\Im\Pi_v\cap(I_v/(\Im\d s\vert_v))=\{0\}$, and thus
\e
\smash{E^-\vert_v\cap I_v=\{0\}.}
\label{vf5eq2} 
\e

Condition \ref{vf3eq23}, that $t\vert_v(E^-\vert_v)=t\vert_v(E\vert_v)$, is equivalent to the equation $E^-\vert_v+\Ker(t\vert_v)=E\vert_v$, in subspaces of $E\vert_v$. As $\Im\Pi_v$ is a maximal negative definite subspace for $\Re\ti Q_v$ in $(\Ker t\vert_v)/(\Im\d s\vert_v)$, and $K_v/(\Im\d s\vert_v)$ is the orthogonal to a null subspace $I_v/(\Im\d s\vert_v)$ w.r.t.\ $\Re\ti Q_v$, it follows that $\Im\Pi_v+K_v/(\Im\d s\vert_v)=(\Ker t\vert_v)/(\Im\d s\vert_v)$. Lifting to $\Ker t\vert_v$ gives $[E^-\vert_v\cap (\Ker t\vert_v)]+K_v=\Ker t\vert_v$. Thus the subspace $E^-\vert_v+K_v$ in $E\vert_v$ contains $E^-\vert_v$ and $\Ker t\vert_v$, so as $E^-\vert_v+\Ker(t\vert_v)=E\vert_v$, we see that
\e
\smash{E^-\vert_v+K_v=E\vert_v.}
\label{vf5eq3}
\e

Write $\check\Pi_v:E^-\vert_v\cap K_v\ra K_v/I_v$ for the natural projection. It is injective by \eq{vf5eq2}. Using \eq{vf5eq2}--\eq{vf5eq3} and the facts that $\Im\Pi_v$ has half the dimension of $(\Ker t\vert_v)/(\Im\d s\vert_v)$, and $\dim [I_v/(\Im\d s\vert_v)]+\dim[K_v/(\Im\d s\vert_v)]=\dim[(\Ker t\vert_v)/(\Im\d s\vert_v)]$ as $I_v/(\Im\d s\vert_v),K_v/(\Im\d s\vert_v)$ are orthogonal subspaces, by a dimension count we find that $\Im\check\Pi_v$ has half the total dimension of $K_v/I_v$. Also, since the quadratic form on $\check Q_v$ on $K_v/I_v\cong [K_v/(\Im\d s\vert_v)]/[I_v/(\Im\d s\vert_v)]$ descends from the restriction of $\ti Q_v$ to $K_v/(\Im\d s\vert_v)$, and $\Im\check\Pi_v$ descends from $\Im\Pi_v\cap[K_v/(\Im\d s\vert_v)]$, and $\Re\ti Q_v$ is negative definite on $\Im\Pi_v$, we see that $\Re\check Q_v$ is negative definite on~$\Im\check\Pi_v$.

As $E^-\vert_{v_i}\ra E^-\vert_v$ and $\Im\d s\vert_{v_i}\ra I_v$ as $i\ra\iy$, we see from \eq{vf5eq2} that
\e
\smash{E^-\vert_{v_i}\cap (\Im\d s\vert_{v_i})=\{0\}\qquad\text{for $i\gg 0$.}}
\label{vf5eq4}
\e
Since $E^-\vert_{v_i}\ra E^-\vert_v$ and $\Ker t\vert_{v_i}\ra K_v$ as $i\ra\iy$, we see from \eq{vf5eq3} that $E^-\vert_{v_i}+\Ker t\vert_{v_i}=E\vert_{v_i}$ for $i\gg 0$. But this is equivalent to
\e
\smash{t\vert_{v_i}(E^-\vert_{v_i})=t\vert_{v_i}(E\vert_{v_i})
\qquad\text{in $F\vert_{v_i}$ for $i\gg 0$.}}
\label{vf5eq5}
\e

Using \eq{vf5eq4}--\eq{vf5eq5}, the same dimension count as above implies that $\Im\ti\Pi_{v_i}$ has half the dimension of $(\Ker t\vert_{v_i})/(\Im\d s\vert_{v_i})$ for $i\gg 0$. Under the limit \eq{vf5eq1}, we have $\ti Q_{v_i}\ra \check Q_v$ and $\Im\ti\Pi_{v_i}\ra\Im\check\Pi_v$. Thus, as $\Re\check Q_v$ is negative definite on $\Im\check\Pi_v$, we see that $\Re\ti Q_{v_i}$ is negative definite on $\Im\ti\Pi_{v_i}$ for $i\gg 0$. Together with \eq{vf5eq4}--\eq{vf5eq5}, this shows that the assumptions on $E^-\vert_{v_i}$ in $(*)$ hold for $i\gg 0$, which contradicts the choice of sequence $(v_i)_{i=1}^\iy$. This proves Theorem~\ref{vf3thm2}(a).

\subsection{\texorpdfstring{Theorem \ref{vf3thm2}(b): extending pairs $(U,E^-)$ satisfying $(*)$}{Theorem \ref{vf3thm2}(b): extending pairs (U,E⁻) satisfying (*)}}
\label{vf52}

Suppose $\bX,\om_\bX^*,A^\bu,\bs\al,V,E,F,s,t,\psi$ are as in Definition \ref{vf3def3}, and $(U,E^-)$ satisfying $(*)$ is as in Definition \ref{vf3def3}, and $C\subseteq V$ is closed with $C\subseteq U$. Our goal is to construct $(\ti U,\ti E^-)$ satisfying $(*)$ for $V,E,\ldots$ with $C\cup s^{-1}(0)\subseteq\ti U\subseteq V$, such that $E^-\vert_{U'}=\ti E^-\vert_{U'}$ for $U'$ an open neighbourhood of $C$ in~$U\cap\ti U$.

Using the notation of \S\ref{vf32}, $s^{-1}(0)^\alg$ is a finite type closed $\C$-subscheme of $V^\alg$, and the maps $v\mapsto \dim\Ker\d s\vert_v$, $v\mapsto \dim\Ker t\vert_v$ are upper semicontinuous, algebraically constructible functions $s^{-1}(0)^\alg\ra\N$, noting that $t\vert_v$ is independent of choices for $v\in s^{-1}(0)^\alg$. Therefore by some standard facts about constructible sets in algebraic geometry, we can choose a stratification of Zariski topological spaces $s^{-1}(0)^\alg=\ts\coprod_{a\in A}W_a^\alg$, where $A$ is a finite indexing set, and $W_a^\alg$ is a smooth, connected, locally closed $\C$-subscheme of $s^{-1}(0)^\alg\subseteq V^\alg$ for each $a\in A$, with closure $\ov W{}_a^\alg$ in $s^{-1}(0)^\alg$ a finite union of strata $W_b$, such that $v\mapsto \dim\Ker\d s\vert_v$ and $v\mapsto \dim\Ker t\vert_v$ are both constant functions on~$W_a^\alg$.

Writing $W_a\subseteq s^{-1}(0)\subseteq V$ for the set of $\C$-points of $W_a^\alg$, each $W_a$ is a connected, locally closed complex submanifold of $V$ lying in $s^{-1}(0)$, with closure $\ov W{}_a$ a finite union of submanifolds $W_b$, such that $s^{-1}(0)=\ts\coprod_{a\in A}W_a$. On each $W_a$, the maps $v\mapsto \dim\Ker\d s\vert_v$ and $v\mapsto \dim\Ker t\vert_v$ are constant. This implies that $\Ker\d s\vert_{W_a}$ is a holomorphic vector subbundle of $TV\vert_{W_a}$, and $\Im\d s\vert_{W_a}$ a holomorphic vector subbundle of $E\vert_{W_a}$, and $\Ker t\vert_{W_a}$ a holomorphic vector subbundle of $E\vert_{W_a}$, and $\Im t\vert_{W_a}$ a holomorphic vector subbundle of $F\vert_{W_a}$. We have $\Im\d s\vert_{W_a}\subseteq \Ker t\vert_{W_a}\subseteq E\vert_{W_a}$ as $t\ci\d s=0$ on~$s^{-1}(0)$.

Thus we have a holomorphic vector bundle $(\Ker t\vert_{W_a})/(\Im\d s\vert_{W_a})$ over $W_a$, whose fibre at $v\in W_a$ is identified with $H^1(\bT_\bX\vert_x)$ for $x=\psi(v)$ by \eq{vf3eq9}. As in \eq{vf2eq6} we have a quadratic form $Q_x$ on $H^1\bigl(\bT_\bX\vert_x\bigr)$, and as in \eq{vf3eq21} $\ti Q_v$ is the quadratic form on $(\Ker t\vert_{W_a})/(\Im\d s\vert_{W_a})\vert_v$ identified with $Q_x$ by \eq{vf3eq9}. One can prove using a representative $\om_{A^\bu}$ for $\bs\al^*(\om_\bX^0)$ that $\ti Q_v$ depends holomorphically on $v\in W_a$. Hence $\ti Q_v=\ti Q_a\vert_v$ for $\ti Q_a\in H^0\bigl(S^2[(\Ker t\vert_{W_a})/(\Im\d s\vert_{W_a})]^*\bigr)$, a nondegenerate holomorphic quadratic form on the fibres of~$(\Ker t\vert_{W_a})/(\Im\d s\vert_{W_a})$.

The idea of the proof is to choose $\ti E^-$ near $W_a$ by induction on increasing $\dim W_a$, starting with $a\in A$ with $\dim W_a=0$, then $\dim W_a=1$, and so on. Since $\dim (\ov W_a\sm W_a)<\dim W_a$, we see that $\ov W_a\sm W_a$ is a finite union of $W_b$ with $\dim W_b<\dim W_a$, so when we choose $\ti E^-$ near $W_a$ we will already have chosen $\ti E^-$ near $\ov W_a\sm W_a$, and the extension over $W_a$ should be compatible with this.

Our inductive hypothesis $(\ddag)_m$ for $m=0,1,2,\ldots$ is:
\begin{itemize}
\item[$(\ddag)_m$] For all $a\in A$ with $\dim W_a\le m$ we have chosen a pair $(\check U_a,\check E^-_a)$ satisfying $(*)$ for $V,E,F,s,t,\ldots$ with $W_a\subseteq\check U_a\subseteq V$, such that there is an open neighbourhood $\hat U_a$ of $C\cap\check U_a$ in $U\cap\check U_a$ with $E^-\vert_{\smash{\hat U_a}}=\check E^-_a\vert_{\smash{\hat U_a}}$, and if $b\in A$ with $W_b\subseteq\ov W_a\sm W_a$ (which implies that $\dim W_b<\dim W_a\le m$, so $(\check U_b,\check E^-_b)$ is defined), then there is an open neighbourhood $\hat U_{ab}$ of $W_b$ in $\check U_b$ such that $\check E^-_a\vert_{\smash{\check U_a\cap\hat U_{ab}}}=\check E^-_b\vert_{\smash{\check U_a\cap\hat U_{ab}}}$.
\end{itemize}

First consider how to choose $(\check U_a,\check E^-_a)$ satisfying $(*)$ with $W_a\subseteq\check U_a\subseteq V$ for $a\in A$ with no compatibility conditions, either with $(U,E^-)$ near $C$, or with $(\check U_b,\check E^-_b)$ for $W_b\subseteq\ov W_a\sm W_a$. We can do this as follows:
\begin{itemize}
\setlength{\itemsep}{0pt}
\setlength{\parsep}{0pt}
\item[(i)] Choose a real vector subbundle $\dot E_a$ of $(\Ker t\vert_{W_a})/(\Im\d s\vert_{W_a})$, whose real rank is half the real rank of $(\Ker t\vert_{W_a})/(\Im\d s\vert_{W_a})$, such that $\Re\ti Q_a$ is negative definite on $\dot E_a$. 
\item[(ii)] Lift $\dot E_a$ to a real vector subbundle $\ddot E_a$ of $\Ker t\vert_{W_a}$. That is, the projection $\Ker t\vert_{W_a}\ra(\Ker t\vert_{W_a})/(\Im\d s\vert_{W_a})$ induces an isomorphism $\ddot E_a\ra \dot E_a$.
\item[(iii)] Choose a real vector subbundle $\dddot E_a$ of $E\vert_{W_a}$ with~$E\vert_{W_a}=\dddot E_a\op\Ker t\vert_{W_a}$.
\item[(iv)] Set $\check E^-_a\vert_{W_a}=\ddot E_a\op \dddot E_a$. Then $\check E^-_a\vert_{W_a}$ is a real vector subbundle of $E\vert_{W_a}$, and the assumptions on $\check E^-_a\vert_v$ in condition $(*)$ in \S\ref{vf33} hold for all~$v\in W_a$.
\item[(v)] Choose any real vector subbundle $\check E^-_a$ of $E\vert_{\check U_a}$ on a small open neighbourhood $\check U_a$ of $W_a$ in $V$, extending the given $\check E^-_a\vert_{W_a}=\ddot E_a\op\dddot E_a$ on~$W_a$.

Observe that by Theorem \ref{vf3thm2}(a), proved in \S\ref{vf51}, condition $(*)$ holds for $\check E^-_a$ on an open neighbourhood of $W_a$. So by making $\check U_a$ smaller, we can suppose $(\check U_a,\check E^-_a)$ satisfies~$(*)$.
\end{itemize}
All of these are possible. Any $(\check U_a,\check E^-_a)$ satisfying $(*)$ with $W_a\subseteq\check U_a\subseteq V$ arises from steps (i)--(v) (though $\dddot E_a$ in (iii) is not uniquely determined by $\check E^-_a$). Furthermore (taking germs in (v)), the space of choices in each step is contractible.

Now suppose $m=0,1,\ldots,$ and $(\ddag)_{m-1}$ holds if $m>0$, and $a\in A$ with $\dim W_a=m$. To choose $(\check U_a,\check E^-_a)$ with the compatibility conditions required in $(\ddag)_m$, we follow (i)--(v), but modified as follows. In step (i), we choose $\dot E_a$ with
\e
\smash{\dot E_a\vert_{W_a\cap\hat U_a}=\bigl[((E^-\cap\Ker t)\vert_{W_a\cap\hat U_a})+(\Im\d s\vert_{W_a\cap\hat U_a})\bigr]\big/\bigl(\Im\d s\vert_{W_a\cap\hat U_a}\bigr),}
\label{vf5eq6}
\e
for some small open neighbourhood $\hat U_a$ of $C\cap W_a$ in $U$, and if $b\in A$ with $W_b\subseteq\ov W_a\sm W_a$ then
\e
\smash{\dot E_a\vert_{W_a\cap\check U_{ab}}\!=\!\bigl[((\check E^-_b\cap\Ker t\vert_{W_a\cap\hat U_{ab}}))\!+\!(\Im\d s\vert_{W_a\cap\hat U_{ab}})\bigr]\big/\bigl(\Im\d s\vert_{W_a\cap\hat U_{ab}}\bigr),}
\label{vf5eq7}
\e
for some small open neighbourhood $\hat U_{ab}$ of $W_b$ in $\check U_b$.

To see this is possible, first note that the first part of $(\ddag)_{m-1}$ with $b$ in place of $a$ implies that equations \eq{vf5eq6} and \eq{vf5eq7} are compatible, that is they prescribe the same value for $\dot E_a$ on $W_a\cap \hat U_a\cap\hat U_{ab}$, provided the open neighbourhoods $\hat U_a,\hat U_{ab}$ are small enough. Also given distinct $b,b'\in A$ with $W_b,W_{b'}\subseteq\ov W_a\sm W_a$, either (a) $W_{b'}\subseteq\ov W_b\sm W_b$, or (b) $W_b\subseteq\ov W_{b'}\sm W_{b'}$, or (c) $W_b\cap\ov W_{b'}=\ov W_b\cap W_{b'}=\es$. In cases (a),(b) we can use the second part of $(\ddag)_{m-1}$ to show that \eq{vf5eq7} for $b,b'$ are compatible provided $\hat U_{ab},\hat U_{ab'}$ are small enough, and in case (c) we can choose $\hat U_{ab},\hat U_{ab'}$ with $\hat U_{ab}\cap\hat U_{ab'}=\es$, so compatibility is trivial.

Thus, if $\hat U_a$ and $\hat U_{ab}$ for all $b$ are small enough then \eq{vf5eq6} and \eq{vf5eq7} for all $b$ are compatible, and can be combined into a single equation prescribing $\dot E_a$ on $\check W_a:=W_a\cap(\hat U_a\cup\bigcup_b\hat U_{ab})$. We then have to extend $\dot E_a$ from $\check W_a$ to $W_a$, satisfying the required conditions. This may not be possible: if we have chosen $E^-$ or $\check E^-_b$ badly near the `edge' of $\check W_a$ in $W_a$, then the prescribed values of $\dot E_a$ may not extend continuously to the closure $\ov{\check W_a}$ of $\check W_a$ in $W_a$. However, we can deal with this problem by shrinking all the $\hat U_a,\hat U_{ab}$, such that the closure $\ov{\check W_a}$ of the new $\check W_a$ lies inside the old $\check W_a$. Then it is guaranteed that the prescribed value of $\dot E_a$ on $\check W_a$ extends smoothly to an open neighbourhood of $\ov{\check W_a}$ in $W_a$, so we can choose $\dot E_a$ on $W_a$ satisfying all the required conditions~\eq{vf5eq6}--\eq{vf5eq7}.

In a similar way, for each of steps (ii)--(v) we can show that making the open neighbourhoods $\hat U_a,\hat U_{ab}$ smaller if necessary, we can make choices consistent with the compatibility conditions on $(\check U_a,\check E^-_a)$ in $(\ddag)_m$. So by induction, $(\ddag)_m$ holds for all $m=0,1,\ldots.$ Fix data $(\check U_a,\check E_a^-),\hat U_a,\hat U_{ab}$ satisfying $(\ddag)_m$ for~$m=\dim V$.

Next, choose open neighbourhoods $U'$ of $C$ in $U\subseteq V$ and $\ti U_a$ of $W_a$ in $\check U_a$ for each $a\in A$, such that $U'\cap\ti U_a\subseteq\hat U_a$ for $a\in A$, and $\ti U_a\cap\ti U_b\subseteq\hat U_{ab}$ if $a,b\in A$ with $W_b\subseteq\ov W_a\sm W_a$, and $\ti U_a\cap\ti U_b=\es$ if $a,b\in A$ with $\ov W_a\cap W_b=W_a\cap\ov W_b=\es$. This is possible provided $U'$ and $\ti U_a$ for $a\in A$ are all small enough. 

Define $\ti U=U'\cup\bigcup_{a\in A}\ti U_a$, which is an open neighbourhood of $C\cup\bigcup_{a\in A}W_a=C\cup s^{-1}(0)$ in $V$. Define a vector subbundle $\ti E^-$ of $E\vert_{\smash{\ti U}}$ by $\ti E^-\vert_{U'}=E^-\vert_{U'}$ and $\ti E^-\vert_{\smash{\ti U_a}}=\check E_a^-\vert_{\smash{\ti U_a}}$ for $a\in A$. These values are consistent on overlaps $U'\cap\ti U_a$ and $\ti U_a\cap\ti U_b$ by construction, so $\ti E^-$ is well-defined. Also $(\ti U,\ti E^-)$ satisfies $(*)$, since $(U,E^-)$ and the $(\check U_a,\check E^-_a)$ do, and $U'$ is an open neighbourhood of $C$ in $U\cap\ti U$ with $E^-\vert_{U'}=\ti E^-\vert_{U'}$ by definition. This proves Theorem~\ref{vf3thm2}(b).

\subsection{\texorpdfstring{Theorem \ref{vf3thm2}(c): $s^{-1}(0)=(s^+)^{-1}(0)$ locally in $U$}{Theorem \ref{vf3thm2}(c): s⁻¹(0)=(s⁺)⁻¹(0) locally in U}}
\label{vf53}

In \S\ref{vf34} we explained how to pullback pairs $(U_K,E_K^-)$ satisfying $(*)$ along a quasi-free $\Phi_{JK}:A_K^\bu\ra A_J^\bu$. We can also {\it pushforward\/} $(U_J,E_J^-)$ along~$\Phi_{JK}$.

\begin{dfn} Let $\bX,\om_\bX^*,n,\Phi_{JK}:A_K^\bu\ra A_J^\bu$ and $V_J,E_J,\ldots,\chi_{JK},\xi_{JK}$ be as in Definition \ref{vf3def4}, and suppose $(U_J,E_J^-)$ satisfies $(*)$ for $A_J^\bu$. Our goal is to construct $(U_K,E_K^-)$ satisfying $(*)$ for $A_K^\bu$, with $\psi_J(s_J^{-1}(0)\cap U_J)=\psi_K(s_K^{-1}(0)\cap U_K)\subseteq X_\an$, and if $(U_J,E_J^-),(U_K,E_K^-)$ also satisfy $(\dag)$, a coordinate change of Kuranishi neighbourhoods, as in~\S\ref{vf25}:
\e
\smash{(U_K,\th_{KJ},\eta_{KJ}):(U_K,E_K^+,s_K^+,\psi^+_K)\longra(U_J,E_J^+,s_J^+,\psi^+_J).}
\label{vf5eq8}
\e

Let $v_J\in s_J^{-1}(0)\cap U_J$ with $\phi_{JK}(v_J)=v_K\in s_K^{-1}(0)\subseteq V_K$ and $\psi_J(v_J)=\psi_K(v_K)=x\in X_\an$. We claim that we can choose splittings of real vector spaces
\e
\begin{aligned}
T_{v_J}V_J&=\ti T_{v_J}V_J\op T_{v_J}'V_J,&
E_J\vert_{v_J}&=\ti E_J\vert_{v_J}\op E'_J\vert_{v_J}\op E''_J\vert_{v_J},\\ 
E_J^-\vert_{v_J}&=\ti E_J^-\vert_{v_J}\op \ti E_J''\vert_{v_J}, &
F_J\vert_{v_J}&=\ti F_J\vert_{v_J}\op F''_J\vert_{v_J}\op F'''_J\vert_{v_J},
\end{aligned}
\label{vf5eq9}
\e
fitting into a commutative diagram of the form
\e
\begin{gathered}
\xymatrix@!0@C=10.1pt@R=40pt{
&&&&&&&&&&&&& *+[r]{E_J^-\vert_{v_J}=\ti E_J^-\vert_{v_J}\op \ts \ti E''_J\vert_{v_J}}
\ar[dd]_(0.25){\raisebox{4pt}{$\st{\rm inc}\,=$}\begin{pmatrix} \st * & \st * \\
\st 0 & \st * \\ \st 0 & \st * \end{pmatrix}}
\ar@/^2pc/[ddrrrrrrrrrrrr]^(0.7){\!\!\!\!\!\!\!\!\!\!\!\!\!\!\!\!\raisebox{4pt}{$\st t_J\vert_{E_J^-\vert_{v_J}}=$}\begin{pmatrix} \st * & \st 0 \\
\st 0 & \st \cong \\ \st 0 & \st 0 \end{pmatrix}}
\\
\\
0 \ar[rrrr] &&&&
{\begin{subarray}{l}\ts \ti T_{v_J}V_J\op \\ \ts T_{v_J}'V_J\end{subarray}} \ar@<-3ex>[d]^(0.55){\d\phi_{JK}\vert_{v_J}=\begin{pmatrix} \st \cong & \st 0  \end{pmatrix}} \ar[rrrrrrrrr]^(0.33){\raisebox{4pt}{$\st\d s_J\vert_{v_J}=$}\begin{pmatrix} \st \widetilde{\d s_K\vert_{v_K}}\!\!\!{} & \st 0 \\
\st * & \st \cong \\ \st 0 & \st 0 \end{pmatrix}}
&&&&&&&&& 
{\begin{subarray}{l} \ts \ti E_J\vert_{v_J}\op \\ \ts E'_J\vert_{v_J}\op \\ \ts E''_J\vert_{v_J} \end{subarray}}
\ar@<-3ex>[d]^(0.55){\chi_{JK}\vert_{v_J}=\begin{pmatrix} \st \cong & \st 0 & \st 0 \end{pmatrix}} \ar[rrrrrrrrrrrr]^(0.43){\raisebox{4pt}{$\st t_J\vert_{v_J}=$}\begin{pmatrix} \st \widetilde{t_K\vert_{v_K}}\!\!\!{} & \st 0 & \st 0 \\
\st 0 & \st 0 & \st \cong \\ \st 0 & \st 0 & \st 0 \end{pmatrix}} &&&&&&&&&&&&
{\begin{subarray}{l}\ts \ti F_J\vert_{v_J}\op \\ \ts F''_J\vert_{v_J}\op \\ \ts F'''_J\vert_{v_J} \end{subarray}}
\ar@<-3ex>[d]^(0.55){\xi_{JK}\vert_{v_J}=\begin{pmatrix} \st \cong\!\!\! & \st 0\!\!\! & \st 0 \end{pmatrix}} \ar[rrrr] &&&& \cdots 
\\
0 \ar[rrrr] &&&& T_{v_K}V_K \ar[rrrrrrrrr]^{\d s_K\vert_{v_K}} &&&&&&&&& E_K\vert_{v_K} \ar[rrrrrrrrrrrr]^{t_K\vert_{v_K}} &&&&&&&&&&&&
F_K\vert_{v_K} \ar[rrrr] &&&& \cdots.\!{} }\!\!\!\!\!\!\!\!\!\!\!\!{}
\end{gathered}
\label{vf5eq10}
\e

To prove this, note that the rows of \eq{vf5eq10} are $\bT_{\bSpec A_J^\bu}\vert_{v_J},\bT_{\bSpec A_K^\bu}\vert_{v_K}$, and are complexes, and the lower columns are induced by $\Phi_{JK}$, are surjective as $\Phi_{JK}$ is quasi-free, and induce isomorphisms on cohomology as in \S\ref{vf32}. Then:
\begin{itemize}
\setlength{\itemsep}{0pt}
\setlength{\parsep}{0pt}
\item[(i)] Define $T_{v_J}'V_J=\Ker\d\phi_{JK}\vert_{v_J}$.
\item[(ii)] Choose arbitrary $\ti T_{v_J}V_J$ with $T_{v_J}V_J\cong \ti T_{v_J}V_J\op T_{v_J}'V_J$. Then $\ti T_{v_J}V_J\cong T_{v_K}V_K$ as $\d\phi_{JK}$ is surjective.
\item[(iii)] Define $E'_J\vert_{v_J}=\d s_J\vert_{v_J}[T_{v_J}'V_J]$. Then $E'_J\vert_{v_J}\cong T_{v_J}'V_J$ as the columns of \eq{vf5eq10} are isomorphisms in cohomology, and $E'_J\vert_{v_J}\subseteq \Ker(\chi_{JK}\vert_{v_J})$ as the left hand square of \eq{vf5eq10} commutes. 
\item[(iv)] Choose $E''_J\vert_{v_J}$ with $\Ker(\chi_{JK}\vert_{v_J})=E'_J\vert_{v_J}\op E''_J\vert_{v_J}$.
\item[(v)] Since the columns of \eq{vf5eq10} are isomorphisms on cohomology, $t_J\vert_{v_J}$ is injective on $E''_J\vert_{v_J}$. Define $F''_J\vert_{v_J}=t_J\vert_{v_J}[E''_J\vert_{v_J}]$. Then $F''_J\vert_{v_J}\cong E''_J\vert_{v_J}$. Also $F''_J\vert_{v_J}\subseteq\Ker\xi_{JK}\vert_{v_J}$, as the right hand square of \eq{vf5eq10} commutes.
\item[(vi)] Choose $F'''_J\vert_{v_J}$ with $\Ker\xi_{JK}\vert_{v_J}=F''_J\vert_{v_J}\op F'''_J\vert_{v_J}$.
\item[(vii)] Since the columns of \eq{vf5eq10} are isomorphisms on cohomology, we have
\begin{align*}
F''_J\vert_{v_J}&=t_J\vert_{v_J}[E'_J\vert_{v_J}\op E''_J\vert_{v_J}]=
t_J\vert_{v_J}[\Ker\chi_{JK}\vert_{v_J}]\\
&=\Ker\xi_{JK}\vert_{v_J}\cap\Im t_J\vert_{v_J}=
(F''_J\vert_{v_J}\op F'''_J\vert_{v_J})\cap\Im t_J\vert_{v_J}.
\end{align*}
Thus we may choose $\ti F_J\vert_{v_J}$ with $F_J\vert_{v_J}=\ti F_J\vert_{v_J}\op F''_J\vert_{v_J}\op F'''_J\vert_{v_J}$ and $\Im t_J\vert_{v_J}\subseteq \ti F_J\vert_{v_J}\op F''_J\vert_{v_J}$. So the third row of $t_J\vert_{v_J}$ in \eq{vf5eq10} is zero. Also $\ti F_J\vert_{v_J}\cong F_K\vert_{v_K}$ by (vi) as $\xi_{JK}$ is surjective.
\item[(viii)] Set $\ti E_J^-\vert_{v_J}=E_J^-\vert_{v_J}\cap t_J\vert_{v_J}^{-1}(\ti F_J\vert_{v_J})$. We claim $\chi_{JK}\vert_{v_J}$ is injective on $\ti E_J^-\vert_{v_J}$. To see this, note that we have an exact sequence
\begin{equation*}
\smash{\xymatrix@C=15pt{ 0 \ar[r] & E_J^-\vert_{v_J}\cap \Ker t_J\vert_{v_J} \ar[r] & \ti E_J^-\vert_{v_J} \ar[r] & t_J\vert_{v_J}[E_J^-\vert_{v_J}] \cap \ti F_J\vert_{v_J} \ar[r] & 0,  }}
\end{equation*}
as $\Ker t_J\vert_{v_J}\subseteq t_J\vert_{v_J}^{-1}(\ti F_J\vert_{v_J})$. The last part of $(*)$ implies that $\chi_{JK}\vert_{v_J}$ maps $E_J^-\vert_{v_J}\cap \Ker t_J\vert_{v_J}$ injectively into $\Ker t_K\vert_{v_K}$. Also $\xi_{JK}\vert_{v_J}$ is injective on $\ti F_J\vert_{v_J}$, and the right square of \eq{vf5eq10} commutes, so the claim follows.
\item[(ix)] Choose $\ti E_J\vert_{v_J}\subseteq E_J\vert_{v_J}$ such that $\ti E_J^-\vert_{v_J}\subseteq\ti E_J\vert_{v_J}$ and $E_J\vert_{v_J}=\ti E_J\vert_{v_J}\op\Ker(\chi_{JK}\vert_{v_J})=\ti E_J\vert_{v_J}\op E'_J\vert_{v_J}\op E''_J\vert_{v_J}$ by (iv) and $t_J\vert_{v_J}[\ti E_J\vert_{v_J}]\subseteq\ti F_J\vert_{v_J}$. This is possible as $\chi_{JK}\vert_{v_J}$ is injective on $\ti E_J^-\vert_{v_J}$, and using (v),(vii) and (viii). Then $\ti E_J\vert_{v_J}\cong E_K\vert_{v_K}$ as $\chi_{JK}$ is surjective.
\item[(x)] Choose $\ti E_J''\vert_{v_J}$ such that $E_J^-\vert_{v_J}=\ti E_J^-\vert_{v_J}\op\ti E_J''\vert_{v_J}$ and $t_J\vert_{v_J}[\ti E_J''\vert_{v_J}]\subseteq F_J''\vert_{v_J}$. This is possible by (viii) and~$\Im t_J\vert_{v_J}\subseteq \ti F_J\vert_{v_J}\op F_J''\vert_{v_J}$.

As $t_J\vert_{v_J}(E_J^-\vert_{v_J})=t_J\vert_{v_J}(E_J\vert_{v_J})$ by \eq{vf3eq23} and $F''_J\vert_{v_J}=t_J\vert_{v_J}[E''_J\vert_{v_J}]$, we see that $t_J\vert_{v_J}[\ti E_J''\vert_{v_J}]=F_J''\vert_{v_J}$. Also $t_J\vert_{v_J}:\ti E_J''\vert_{v_J}\ra F_J''\vert_{v_J}$ is injective, as $\Ker t_J\vert_{E_J^-\vert_{v_J}}\subseteq\ti E_J^-\vert_{v_J}$ by (viii). Hence $\ti E_J''\vert_{v_J}\cong F''_J\vert_{v_J}$.
\end{itemize}

We can do all this, not just at one $v_J\in s_J^{-1}(0)\cap U_J$, but in an open neighbourhood $U_J'$ of $s_J^{-1}(0)\cap U_J$ in $U_J$. That is, we can choose $U_J'$, and splittings 
\e
\begin{aligned}
TV_J\vert_{U_J'}&=\ti TV_J\op T'V_J,&
E_J\vert_{U_J'}&=\ti E_J\op E'_J\op E''_J\vert_{v_J},\\ 
E_J^-\vert_{U_J'}&=\ti E_J^-\op \ti E_J'', &
F_J\vert_{U_J'}&=\ti F_J\op F''_J\op F'''_J,
\end{aligned}
\label{vf5eq11}
\e
with $\ti E_J^-\subseteq\ti E_J$, such that \eq{vf5eq10} holds at each $v_J\in s_J^{-1}(0)\cap U_J$. To see this, note that the argument above can be carried out on $s_J^{-1}(0)\cap U_J$ regarded as a $C^\iy$-subscheme of $U_J$, in the sense of $C^\iy$-algebraic geometry in \cite{Joyc1}, and the splittings \eq{vf5eq11} with $\ti E_J^-\subseteq\ti E_J$ can then be extended from $s_J^{-1}(0)\cap U_J$ to an open neighbourhood $U_J'$. Making $U_J'$ smaller, we can suppose the component of $\chi_{JK}$ mapping $\ti E_J\ra \phi_{JK}\vert_{U_J'}^*(E_K)$ is an isomorphism. We can also choose the splittings so that away from 
$s_J^{-1}(0)\cap U_J$, $t_J\vert_{U_J'}$ has the form
\e
t_J\vert_{U_J'}=\begin{pmatrix} \st * & \st * & \st 0 \\
\st * & \st * & \st \cong \\ \st * & \st * & \st 0 \end{pmatrix}:
\ti E_J\vert_{v_J}\op E'_J\op E''_J\longra\ti F_J\op F''_J\op F'''_J.
\label{vf5eq12}
\e

Write $s_J\vert_{U_J'}=\ti s_J\op s_J'\op s_J''$, for $\ti s_J\in C^\iy(\ti E_J)$, $s_J'\in C^\iy(E_J')$ and $s_J''\in C^\iy(E_J'')$. Then \eq{vf5eq12} and $t_J\ci s_J=0$ imply that $s_J''=0$. From \eq{vf5eq10} we see that at each $v_J\in s_J^{-1}(0)\cap U_J$, $\d s_J'\vert_{v_J}:T_{v_J}V_J\ra E_J'\vert_{v_J}$ is surjective, and $\d\phi_{JK}\vert_{v_J}:\Ker\bigl(\d s_J'\vert_{v_J}\bigr)\ra T_{v_K}V_K$ is an isomorphism. Hence $s_J'$ is transverse near $v_J$, so that $(s_J')^{-1}(0)$ is an embedded submanifold of $V_J$ near $v_J$ with tangent space $\Ker\bigl(\d s_J'\vert_{v_J}\bigr)$ at $v_J$, and $\phi_{JK}\vert_{(s_J')^{-1}(0)}:(s_J')^{-1}(0)\ra V_K$ is a local diffeomorphism near $v_J$. Thus, making $U_J'$ smaller, we can suppose that $s_J'$ is transverse on $U_J'$, so that $(s_J')^{-1}(0)$ is an embedded submanifold of $U_J'$, and $\phi_{JK}\vert_{(s_J'')^{-1}(0)}:(s_J')^{-1}(0)\ra V_K$ is a local diffeomorphism. But $\phi_{JK}$ is injective on $s_J^{-1}(0)\cap U_J$, so making $U_J'$ smaller, we can also suppose that $\phi_{JK}\vert_{(s_J')^{-1}(0)}$ is a diffeomorphism with an open set $U_K$ in $V_K$, with inverse~$\th_{KJ}:U_K\,{\buildrel\cong\over\longra}\,(s_J')^{-1}(0)\subseteq U_J'\subseteq U_J$.

We now have a vector bundle $\th_{KJ}^*(E_J)$ over $U_K$, and vector subbundles $\th_{KJ}^*(\ti E_J,E_J',E_J'',E_J^-,\ti E_J^-,\ti E_J'')$ with $\th_{KJ}^*(E_J)\ab=\ab\th_{KJ}^*(\ti E_J)\op\th_{KJ}^*(E_J')\op\th_{KJ}^*(E_J'')$,  $\th_{KJ}^*(E_J^-)=\th_{KJ}^*(\ti E_J^-)\op\th_{KJ}^*(\ti E_J'')$ and $\th_{KJ}^*(\ti E_J^-)\subseteq\th_{KJ}^*(\ti E_J)$. Since $\phi_{JK}\ci\th_{KJ}=\id_{U_K}$, pulling back $\chi_{JK}:E_J\ra\phi_{JK}^*(E_K)$ by $\th_{KJ}$ gives a surjective vector bundle morphism $\th_{KJ}^*(\chi_{JK}):\th_{KJ}^*(E_J)\ra E_K\vert_{U_K}$, where $\th_{KJ}^*(\chi_{JK})$ restricts to an isomorphism $\th_{KJ}^*(\ti E_J)\ra E_K$. We also have a section $\th_{KJ}^*(s_J)$ of $\th_{KJ}^*(E_J)$, whose components in $\th_{KJ}^*(\ti E_J),\th_{KJ}^*(E_J'),\th_{KJ}^*(E_J'')$ are $\th_{KJ}^*(\ti s_J),0,0$. Applying $\th_{KJ}^*$ to \eq{vf3eq14} and using $E_J''\subseteq\Ker\chi_{JK}$ shows that
\e
\smash{\th_{KJ}^*(\chi_{JK})[\th_{KJ}^*(s_J)]=\th_{KJ}^*(\chi_{JK})[\th_{KJ}^*(\ti s_J)]=s_K\vert_{U_K}.}
\label{vf5eq13}
\e

Define a vector subbundle $E_K^-\subseteq E_K\vert_{U_K}$ by $E_K^-=\th_{KJ}^*(\chi_{JK})[\th_{KJ}^*(\ti E_J^-)]$. This is valid as $\th_{KJ}^*(\ti E_J^-)\subseteq\th_{KJ}^*(\ti E_J)$, and $\th_{KJ}^*(\chi_{JK})$ is an isomorphism on $\th_{KJ}^*(\ti E_J)$. We claim that $(U_K,E_K^-)$ satisfies condition $(*)$. To see this, let $v_K\in s_K^{-1}(0)\cap U_K$, and set $v_J=\th_{KJ}(v_K)$. Then $v_J\in s_J^{-1}(0)\cap U_J'$ with $\phi_{JK}(v_J)=v_K$, so \eq{vf5eq9}--\eq{vf5eq10} hold, with the columns of \eq{vf5eq10} isomorphisms on cohomology. From this and $(*)$ for $(U_J,E_J^-)$ at $v_J$, we can deduce $(*)$ for $(U_K,E_K^-)$ at $v_K$.

Writing $E_J^+=E_J\vert_{U_J}/E_J^-$, $s_J^+=s_J+E_J^-\in C^\iy(E_J^+)$, and similarly for $E_K^+,s_K^+$, define a vector bundle morphism $\eta_{KJ}:E_K^+\ra \th_{KJ}^*(E_J^+)$ by
\begin{equation*}
\smash{\eta_{KJ}:e_K+E_K^-\longmapsto \th_{KJ}^*(\chi_{JK})\vert_{\th_{KJ}^*(\ti E_J)}^{-1}[e_K]+\th_{KJ}^*(E_J^-).}
\end{equation*}
This is well-defined as $\th_{KJ}^*(\chi_{JK})\vert_{\th_{KJ}^*(\ti E_J)}:\th_{KJ}^*(\ti E_J)\ra E_K$ is an isomorphism, with inverse $\th_{KJ}^*(\chi_{JK})\vert_{\th_{KJ}^*(\ti E_J)}^{-1}:E_K\ra \th_{KJ}^*(\ti E_J)$, which maps $E_K^-\ra \th_{KJ}^*(\ti E_J^-)\ab\subseteq\th_{KJ}^*(E_J^-)$ by definition of $E_K^-$. Also \eq{vf5eq13} implies that $\eta_{KJ}(s_K^+)=\th_{KJ}^*(s_J^+)$. Using \eq{vf5eq10} we can also show that the analogue of \eq{vf2eq8} for $\th_{KJ},\eta_{KJ}$ at $v_K$ is exact. Therefore, if $(U_J,E_J^-),(U_K,E_K^-)$ also satisfy $(\dag)$, then $(U_K,\th_{KJ},\eta_{KJ})$ in \eq{vf5eq8} is a coordinate change. This completes Definition~\ref{vf5def}.
\label{vf5def}
\end{dfn}

We now prove Theorem \ref{vf3thm2}(c). Suppose $\bX,\om_\bX^*,A^\bu,\bs\al,V,E,F,s,t,\psi$ and $(U,E^-)$ satisfying $(*)$ are as in Definition \ref{vf3def3}. Then $\bs{X'}:=\bs\al(\bSpec A^\bu)\subseteq\bX$ is an affine derived $\C$-subscheme of $\bX$. Let $v\in s^{-1}(0)\cap U$, and set $x=\psi(v)\in X_\an$. Write $(A_1^\bu,\bs\al_1)=(A^\bu,\bs\al)$, $V_1=V$, $E_1=E$, $v_1=v$ and so on. Applying Theorem \ref{vf2thm3} to $(\bX',\om_\bX^*\vert_{\bX'})$ at $x$ gives a pair $(A_2^\bu,\om_{A_2^\bu})$ in $-2$-Darboux form and a Zariski open inclusion $\bs\al_2:\bSpec A_2^\bu\hookra\bX'\subseteq\bX$ which is minimal at $x\in\Im\bs\al_2$ with $\bs\al_2^*(\om_\bX^*)\simeq\om_{A_2^\bu}$. Section \ref{vf32} applied to $A_2^\bu,\bs\al_2$ gives $V_2,E_2,s_2,\ldots.$ Set~$v_2=\psi_2^{-1}(x)\in s_2^{-1}(0)\subseteq V_2$.

Applying Theorem \ref{vf3thm1} to the derived $\C$-scheme $\bX'$ with $I=\{1,2\}$ and initial data $\bigl\{(A_1^\bu,\bs\al_1),(A_2^\bu,\bs\al_2)\bigr\}$ gives $(A_{12}^\bu,\bs\al_{12})$ with image $\Im\bs\al_{12}=\Im\bs\al_1\cap\Im\bs\al_2$ and quasi-free morphisms $\Phi_{12,1}:A_1^\bu\ra A_{12}^\bu$, $\Phi_{12,2}:A_2^\bu\ra A_{12}^\bu$ such that \eq{vf3eq3} homotopy commutes in $\dSch_\C$. Section \ref{vf32} applied to $A_{12}^\bu$ gives $V_{12},E_{12},s_{12},\ldots,$ and to $\Phi_{12,1},\ab\Phi_{12,2}$ gives $\phi_{12,1}:V_{12}\ra V_1=V$, $\chi_{12,1},\xi_{12,1}$ and $\phi_{12,2}:V_{12}\ra V_2$, $\chi_{12,2},\xi_{12,2}$, simplifying notation a little. Set $v_{12}=\psi_{12}^{-1}(x)\in s_{12}^{-1}(0)\subseteq V_{12}$, so that $\phi_{12,1}(v_{12})=v_1$ and~$\phi_{12,2}(v_{12})=v_2$.

We have $(U,E^-)$ satisfying $(*)$ for $A_1^\bu,\bs\al_1,V_1,E_1,s_1,\ldots.$ Thus by Lemma \ref{vf3lem1}, we can choose $(U_{12},E_{12}^-)$ satisfying $(*)$ for $V_{12},E_{12},s_{12},\ldots$ and compatible with $(U,E^-)$ under $\phi_{12,1},\chi_{12,1}$ in the sense of \S\ref{vf34}, such that $v_{12}\in s_{12}^{-1}(0)\cap\phi_{12,1}^{-1}(U)\subseteq U_{12}\subseteq V_{12}$. Also \S\ref{vf34} defines $\chi_{12,1}^+$ such that if $(U,E^-)$ and $(U_{12},E_{12}^-)$ satisfy $(\dag)$ (we do {\it not\/} assume this), then
\begin{equation*}
\smash{(U_{12},\phi_{12,1}\vert_{U_{12}},\chi_{12,1}^+):(U_{12},E_{12}^+,s_{12}^+,\psi_{12}^+)\longra(U,E^+,s^+,\psi^+)}
\end{equation*}
is a coordinate change of Kuranishi neighbourhoods, as in Corollary~\ref{vf3cor2}.

Now apply Definition \ref{vf5def} to pushforward $(U_{12},E_{12}^-)$ in $V_{12},E_{12},s_{12},\ldots$ along $\phi_{12,2},\chi_{12,2},\xi_{12,2}$. This yields $(U_2,E_2^-)$ satisfying $(*)$ for $V_2,E_2,s_2,\ldots$ with $\phi_{12,2}(s_{12}^{-1}(0)\cap U_{12})\subseteq U_2\subseteq V_2$, so in particular $v_2\in U_2$, and data $\th_{2,12},\eta_{2,12}$ such that if $(U_2,E_2^-)$ and $(U_{12},E_{12}^-)$ satisfy $(\dag)$ (we do {\it not\/} assume this), then
\e
\smash{(U_2,\th_{2,12},\eta_{2,12}):(U_2,E_2^+,s_2^+,\psi_2^+) \longra(U_{12},E_{12}^+,s_{12}^+,\psi_{12}^+)}
\label{vf5eq14}
\e
is a coordinate change of Kuranishi neighbourhoods, as in~\eq{vf5eq8}. 

Since $(A_2^\bu,\om_{A_2^\bu})$ is in $-2$-Darboux form and minimal at $x$, Example \ref{vf3ex2} proves that there exists an open neighbourhood $U_2'$ of $v_2$ in $U_2$ such that $s_2^{-1}(0)\cap U_2'=(s_2^+)^{-1}(0)\cap U_2'$. Then $(U_2',E_2^-\vert_{\smash{U_2'}})$ satisfies $(\dag)$. The construction in Definition \ref{vf5def} implies that $\th_{2,12}$ identifies $s_2^{-1}(0)$ near $v_2$ with $s_{12}^{-1}(0)$ near $v_{12}$, and identifies $(s_2^+)^{-1}(0)$ near $v_2$ with $(s_{12}^+)^{-1}(0)$ near $v_{12}$ (the second follows from the fact that the analogue of \eq{vf2eq8} for $\th_{2,12},\eta_{2,12}$ at $v_2,v_{12}$ is exact, so \eq{vf5eq14} is a coordinate change of Kuranishi neighbourhoods near $v_2,v_{12}$). Since $s_2^{-1}(0)=(s_2^+)^{-1}(0)$ near $v_2$, it follows that $s_{12}^{-1}(0)=(s_{12}^+)^{-1}(0)$ near $v_{12}$. That is, there exists an open neighbourhood $U_{12}'$ of $v_{12}$ in $U_{12}$ such that~$s_{12}^{-1}(0)\cap U_{12}'=(s_{12}^+)^{-1}(0)\cap U_{12}'$.  

Similarly, $\phi_{12,1}$ identifies $s_{12}^{-1}(0)$ near $v_{12}$ with $s^{-1}(0)$ near $v$, and identifies $(s_{12}^+)^{-1}(0)$ near $v_{12}$ with $(s^+)^{-1}(0)$ near $v$, so there exists an open neighbourhood $U_v'$ of $v$ in $U$ such that $s^{-1}(0)\cap U_v'=(s^+)^{-1}(0)\cap U_v'$. This holds for all $v\in s^{-1}(0)\cap U$. Define $U'=\bigcup_{v\in s^{-1}(0)}U_v'$. Then $U'$ is an open neighbourhood of $s^{-1}(0)\cap U$ in $U$, and $s^{-1}(0)\cap U'=(s^+)^{-1}(0)\cap U'$. Theorem \ref{vf3thm2}(c) follows.

\section{Proofs of some auxiliary results}
\label{vf6}

Next we prove Propositions \ref{vf3prop1}, \ref{vf3prop2} and~\ref{vf3prop3}.

\subsection{Proof of Proposition \ref{vf3prop1}}
\label{vf61}

Let $Z$ be a paracompact, Hausdorff topological space and $\{R_i:i\in I\}$ an open cover of $Z$. By paracompactness we can choose a locally finite refinement $\{S_i:i\in I\}$. That is, $S_i\subseteq R_i\subseteq Z$ is open with $\bigcup_{i\in I}S_i=Z$, and each $z\in Z$ has an open $z\in U_z\subseteq Z$ with $U_z\cap S_i\ne\es$ for only finitely many~$i\in I$. 

By a standard result in topology known as the Shrinking Lemma, we can choose open sets $T^1_i\subseteq Z$ with closures $\ov T{}^1_i\subseteq Z$ for $i\in I$ such that $T^1_i\subseteq\ov T{}^1_i\subseteq S_i$ for $i\in I$ and $\bigcup_{i\in I}T_i^1=Z$. The next part of the proof broadly follows that of McDuff and Wehrheim \cite[Lem.~7.1.7]{McWe}, who prove a similar result with $Z$ compact and $I$ finite. By induction on $k=2,3,\ldots,$ choose open $T_i^k\subseteq Z$ with
\e
\smash{T_i\subseteq \ov T{}_i^1\subseteq T_i^2\subseteq \ov T{}_i^2\subseteq T_i^3\subseteq \ov T{}_i^3\subseteq\cdots\subseteq S_i\subseteq Z}
\label{vf6eq1}
\e
for $i\in I$. Here to choose $T_i^k$ we note that $Z$ is normal as it is paracompact and Hausdorff, so we can choose open $T_i^k,U\subseteq Z$ with $\ov T{}_i^{k-1}\subseteq T_i^k$, $Z\sm S_i\subseteq U$ and $T_i^k\cap U=\es$. Then $T_i^k\subseteq Z\sm U\subseteq S_i$, and $Z\sm U$ is closed, so~$\ov T{}_i^k\subseteq S_i$. 

Now for each finite $\es\ne J\subseteq I$, define a closed subset $C_J\subseteq Z$ by
\e
\smash{C_J=\ts\bigcap_{j\in J}\ov T{}_j^{\md{J}}\sm\bigcap_{i\in I\sm J} T_i^{\md{J}+1}.}
\label{vf6eq2}
\e
Then part (i) of the proposition follows from $\ov T{}_j^{\md{J}}\subseteq S_j\subseteq R_j$ for $j\in J$ by \eq{vf6eq1}, and (ii) from $\{S_i:i\in I\}$ locally finite with $C_J\subseteq\bigcap_{i\in I}S_i$. For (iii), suppose $\es\ne J,K\subseteq I$ are finite with $J\not\subseteq K$ and $K\not\subseteq J$. Without loss of generality, suppose $\md{J}\le\md{K}$. Then there exists $j\in J\sm K$, and \eq{vf6eq2} gives $C_J\subseteq \ov T{}_j^{\md{J}}$ and $C_K\subseteq Z\sm T_j^{\md{K}+1}$, which forces $C_J\cap C_K=\es$ as $\ov T{}_j^{\md{J}}\subseteq T_j^{\md{K}+1}$ by \eq{vf6eq1}.

For part (iv), if $z\in Z$, define
\e
\smash{J_z=\ts\bigcup_{\text{$J\subseteq I$ finite, $z\in\bigcap_{j\in J}\ov T{}_j^{\md{J}}$}}J.}
\label{vf6eq3}
\e
Then $J_z$ is finite as $\{S_i:i\in I\}$ is locally finite, so $z\in S_j$ for only finitely many $j\in I$, and $J_z$ is nonempty as $\{T_i^1:i\in I\}$ covers $Z$, so $z\in T_i^1\subseteq\ov T{}_i^2$ for some $i\in I$, and $J=\{i\}$ is a possible set in the union \eq{vf6eq3}. If $j\in J_z$ then $j\in J$ for some $J$ in the union \eq{vf6eq3}, so that $z\in\ov T{}_j^{\md{J}}\subseteq \ov T{}_j^{\md{J_z}}$ as $\md{J}\le\md{J_z}$. If $i\in I\sm J_z$ then $z\notin \bigcap_{j\in J_z\cup\{i\}}\ov T{}_j^{\md{J_z}+1}$, as $J_z\cup\{i\}$ is not one of the sets $J$ in \eq{vf6eq3}, but $z\in \bigcap_{j\in J_z}\ov T{}_j^{\md{J_z}+1}$, so $z\notin\ov T{}_i^{\md{J_z}+1}$. Hence $z\in C_{J_z}$ by \eq{vf6eq2}, and part (iv) follows. This completes the proof of Proposition~\ref{vf3prop1}.

\subsection{Proof of Proposition \ref{vf3prop2}}
\label{vf62}

We work in the situation of \S\ref{vf35} just after Remark \ref{vf3rem2}, so that we have data $X_\an,I$, $V_J,E_J,s_J,\psi_J$ and $C_J\subseteq R_J=\bigcap_{i\in J}R_i\subseteq X_\an$ for all $J\in A$, and $\phi_{JK},\chi_{JK}$ for all $J,K\in A$ with $K\subsetneq J$. We will first prove the following inductive hypothesis $(+)_m$, by induction on~$m=1,2,\ldots:$
\begin{itemize}
\setlength{\itemsep}{0pt}
\setlength{\parsep}{0pt}
\item[$(+)_m$] For all $J\in A$ with $\md{J}\le m$, we can choose $(\ti U_J,\ti E_J^-)$ satisfying condition $(*)$ for $A_J^\bu,V_J,E_J,F_J,s_J,t_J,\psi_J,\ldots,$ such that $\psi_J^{-1}(C_J)\subseteq\ti U_J\subseteq V_J,$ and if $J,K\in A$ with $K\subsetneq J$ and $0<\md{K}<\md{J}\le m$ then there exists open $\ti U_{JK}\subseteq\ti U_J$ with $\psi_J^{-1}(C_J\cap C_K)\subseteq\ti U_{JK}$ such that $(\ti U_{JK},\ti E_J^-\vert_{\smash{\ti U_{JK}}})$ is compatible with $(\ti U_K,\ti E_K^-),$ in the sense of \S\ref{vf34}. That is, $\phi_{JK}(\ti U_{JK})\subseteq\ti U_K\subseteq V_K$ and $\chi_{JK}\vert_{\smash{\ti U_{JK}}}(\ti E_J^-\vert_{\smash{\ti U_{JK}}})\subseteq \phi_{JK}\vert_{\ti U_{JK}}^*(\ti E_K^-)\subseteq \phi_{JK}\vert_{\ti U_{JK}}^*(E_K)$.
\end{itemize}

For the first step, to prove $(+)_1$, for all $J=\{i\}$ with $i\in I$ we choose $(\ti U_J,\ti E_J^-)$ for $A_J^\bu,V_J,E_J,\ldots$ satisfying $(*)$ with $s_J^{-1}(0)\subseteq\ti U_J$, so that $\psi_J^{-1}(C_J)\subseteq\ti U_J$, by applying Theorem \ref{vf3thm2}(b) with $C=U=\es$. The second part of $(+)_1$ is trivial, as there are no $J,K\in A$ with~$0<\md{K}<\md{J}\le 1$.

For the inductive step, suppose $(+)_{m-1}$ holds for some $m>1$. We will prove $(+)_m$. Using the existing choices of $(\ti U_J,\ti E_J^-)$ and $\ti U_{JK}$ for $J,K\in A$ with $\md{J},\md{K}<m$ from $(+)_{m-1}$, it remains to choose $(\ti U_J,\ti E_J^-)$ when $\md{J}=m$, and $\ti U_{JK}$ when $0<\md{K}<\md{J}=m$. So fix $J\subseteq I$ with~$\md{J}=m$.

Then $(+)_{m-1}$ gives $(\ti U_K,\ti E_K^-)$ satisfying $(*)$ for all $\es\ne K\subsetneq J$. Using the notation of Lemma \ref{vf3lem1}, set $\ti U_{JK}'=\phi_{JK}^{-1}(\ti U_K)\subseteq V_J$, and define $\ti E_{JK}'=\chi_{JK}\vert_{\ti U_{JK}'}^{-1}(\ti E_K^-)$, a vector subbundle of $E_J\vert_{\ti U_{JK}'}$. Then $\ti U_{JK}'$ is an open neighbourhood of $\psi_J^{-1}(C_K)$ in $V_J$, by~\eq{vf3eq16}.

If $\es\ne L\subsetneq K\subsetneq J$ then by $(+)_{m-1}$ there exists open $\ti U_{KL}\subseteq\ti U_K$ with $\psi_K^{-1}(C_K\cap C_L)\subseteq\ti U_{KL}$ such that $\phi_{KL}(\ti U_{KL})\subseteq\ti U_L$ and $\chi_{KL}\vert_{\ti U_{KL}}(\ti E_K^-)\subseteq \phi_{KL}\vert_{\ti U_{KL}}^*(\ti E_L^-)\subseteq \phi_{KL}\vert_{\ti U_{KL}}^*(\ti E_L)$. Pulling back by $\phi_{JK}$, applying $\chi_{JK}$, and using the last part of Corollary \ref{vf3cor1}(ii) then shows that we have an open neighbourhood $\ti U'_{JKL}=\phi_{JK}^{-1}(\ti U_{KL})$ of $\psi_J^{-1}(C_K\cap C_L)$ in $\ti U_{JK}'\cap\ti U_{JL}'\subseteq V_J$, such that
\begin{equation*}
\smash{\ti E_{JK}'\vert_{\ti U'_{JKL}}\subseteq\ti E_{JL}'\vert_{\ti U'_{JKL}}\subseteq E_J\vert_{\ti U'_{JKL}}.}
\end{equation*}

As in Lemma \ref{vf3lem1}, choose vector subbundles $\ti E_{JK}''\subseteq E_J\vert_{\ti U_{JK}'}$ with $E_J\vert_{\ti U_{JK}'}=\ti E_{JK}'\op\ti E_{JK}''$ on $\ti U_{JK}'$ for all $\es\ne K\subsetneq J$. Choose a connection $\nabla$ on $E_J$. As in Lemma \ref{vf3lem1}, for all $\es\ne K\subsetneq J$, $\ti E_{JK}''':=(\nabla s_J)[\Ker\d\phi_{JK}]$ is a vector subbundle of $E_J$ near $s_J^{-1}(0)$ in $V_J$. Making the open neighbourhoods $\ti U_{JK}',\ti U'_{JKL}$ smaller, we can suppose $\ti E_{JK}'''$ is a vector subbundle of $E_J\vert_{\ti U_{JK}'}$. If $\es\ne L\subsetneq K\subsetneq J\subseteq I$ then $\Ker\d\phi_{JK}\subseteq\Ker\d\phi_{JL}$, as $\phi_{JL}=\phi_{KL}\ci\phi_{JK}$, and so 
\begin{equation*}
\smash{\ti E_{JK}'''\vert_{\ti U'_{JKL}}\subseteq\ti E_{JL}'''\vert_{\ti U'_{JKL}}\subseteq E_J\vert_{\ti U'_{JKL}}.}
\end{equation*}

Next, by reverse induction on $l=m-1,m-2,\ldots,1$, we will prove the following inductive hypothesis $(\t)_{J,l}$:
\begin{itemize}
\setlength{\itemsep}{0pt}
\setlength{\parsep}{0pt}
\item[$(\t)_{J,l}$] For all $\es\ne L\subsetneq J$ with $l\le\md{L}$ we can choose an open neighbourhood $\hat U_{JL}$ of $\psi_J^{-1}(C_J\cap C_L)$ in $\ti U_{JL}$ and a vector subbundle $\hat E_{JL}^-$ of $E_{JL}'\vert_{\smash{\hat U_{JL}}}$ such that 
\e
\smash{E_J\vert_{\hat U_{JL}}=\hat E_{JL}^-\op E_{JL}''\vert_{\hat U_{JL}}\op E_{JL}'''\vert_{\hat U_{JL}},}
\label{vf6eq4}
\e
or equivalently, identifying $E_{JL}'$ with $E_J/E_{JL}''$ on $\hat U_{JL}$,
\e
\smash{E_{JL}'\vert_{\hat U_{JL}}=\hat E^-_{JL}\op \bigl[(E_{JL}''\op E_{JL}''')/E_{JL}''\bigr]\vert_{\hat U_{JL}},}
\label{vf6eq5}
\e
and such that if $\es\ne L\subsetneq K\subsetneq J$ with $l\le\md{L}<\md{K}$ then there exists an open neighbourhood $\hat U_{JKL}$ of $\psi_J^{-1}(C_J\cap C_K\cap C_L)$ in $\hat U_{JK}\cap\hat U_{JL}$ with~$\hat E_{JL}^-\vert_{\hat U_{JKL}}=\hat E_{JK}^-\vert_{\hat U_{JKL}}$.
\end{itemize}

For the first step $l=m-1$, for each $L\subsetneq J$ with $\md{L}=m-1$ we take $\hat U_{JL}=\ti U_{JL}$ and take $\hat E^-_{JL}$ to be an arbitrary complement to $\bigl[(E_{JL}''\op E_{JL}''')/E_{JL}''\bigr]$ in $E_{JL}'\vert_{\ti U_{JL}}$, as in \eq{vf6eq5}, which implies \eq{vf6eq4}. The second part of $(\t)_{J,m-1}$ is trivial as there are no $K,L$ with~$m-1\le\md{L}<\md{K}<\md{J}=m$.

For the inductive step, suppose $(\t)_{J,l+1}$ holds for some $1\le l<m-1$, and fix $L\subsetneq J$ with $\md{L}=l$. Choose open neighbourhoods $\hat U_{JKL}$ of $\psi_J^{-1}(C_J\cap C_K\cap C_L)$ in $V_J$ for all $L\subsetneq K\subsetneq J$ with the properties that:
\begin{itemize}
\setlength{\itemsep}{0pt}
\setlength{\parsep}{0pt}
\item[(a)] $\hat U_{JKL}\subseteq \hat U_{JK}\cap\ti U_{JL}$, where $\hat U_{JK}$ is already chosen by~$(\t)_{J,l+1}$.
\item[(b)] If $L\subsetneq K_1,K_2\subsetneq J$ with $K_1\subsetneq K_2$ and $K_2\subsetneq K_1$ then~$\hat U_{JK_1L}\cap \hat U_{JK_2L}=\es$.
\item[(c)] If $L\subsetneq K_2\subsetneq K_1\subsetneq J$ then $\hat U_{JK_1L}\cap \hat U_{JK_2L}\subseteq\hat U_{JK_1K_2}$, where $\hat U_{JK_1K_2}$ is already chosen by~$(\t)_{J,l+1}$.
\end{itemize}
This is possible, using Proposition \ref{vf3prop1}(iii) to ensure (b).

Next, we have to choose an open neighbourhood $\hat U_{JL}$ of $\psi_J^{-1}(C_J\cap C_L)$ in $\ti U_{JL}$ and a vector subbundle $\hat E_{JL}^-$ of $E_{JL}'\vert_{\smash{\hat U_{JL}}}$ satisfying \eq{vf6eq4}--\eq{vf6eq5}, such that for all $K$ with $L\subsetneq K\subsetneq J$ we have $\hat U_{JKL}\subseteq \hat U_{JL}$ and $\hat E_{JL}^-\vert_{\hat U_{JKL}}=\hat E_{JK}^-\vert_{\hat U_{JKL}}$.

First note from Lemma \ref{vf3lem1} that \eq{vf6eq4}--\eq{vf6eq5} near $\psi_J^{-1}(C_J\cap C_L)$ are equivalent to $(\hat U_{JL},\hat E_{JL}^-)$ near $\psi_J^{-1}(C_J\cap C_L)$
satisfying $(*)$ and being compatible with $(\ti U_L,\ti E_L^-)$. By $(\t)_{J,l+1}$ we already know that $\hat E_{JK}^-\vert_{\hat U_{JKL}}$ near $\psi_J^{-1}(C_J\cap C_L)$ satisfies $(*)$ and is compatible with $(\ti U_K,\ti E_K^-)$, so $\hat E_{JK}^-\vert_{\hat U_{JKL}}$ is compatible with $(\ti U_L,\ti E_L^-)$ near $\psi_J^{-1}(C_J\cap C_L)$ as $(\ti U_K,\ti E_K^-)$ is compatible with $(\ti U_L,\ti E_L^-)$ by $(+)_{m-1}$. Therefore the prescribed value $\hat E_{JK}^-\vert_{\hat U_{JKL}}$ for $\hat E_{JL}^-$ on $\hat U_{JKL}$ satisfies \eq{vf6eq4}--\eq{vf6eq5} near $\psi_J^{-1}(C_J\cap C_L)$, and making $\hat U_{JKL}$ smaller, we can suppose $\hat E_{JK}^-\vert_{\hat U_{JKL}}$ satisfies \eq{vf6eq4}--\eq{vf6eq5} on $\hat U_{JKL}$. This proves that \eq{vf6eq4}--\eq{vf6eq5} are compatible with the conditions $\hat E_{JL}^-\vert_{\hat U_{JKL}}=\hat E_{JK}^-\vert_{\hat U_{JKL}}$ for all~$\es\ne L\subsetneq K\subsetneq J$.

Next, observe that the prescribed values $\hat E_{JK}^-\vert_{\hat U_{JKL}}$ for $\hat E_{JL}^-$ on $\hat U_{JKL}$ for different $K_1,K_2$ with $L\subsetneq K_1,K_2\subsetneq J$ agree on overlaps $\hat U_{JK_1L}\cap \hat U_{JK_2L}$. This follows from (b),(c) above and $\hat E_{JK_1}^-\vert_{\hat U_{JK_1K_2}}=\hat E_{JK_2}^-\vert_{\hat U_{JK_1K_2}}$, which holds by $(\t)_{J,l+1}$. Therefore the last part of $(\t)_{J,l}$ can be rewritten to say that we have one prescribed value for $\hat E_{JL}^-$ on the subset $\dot U_{JL}:=\bigcup_{K:L\subsetneq K\subsetneq J}\hat U_{JKL}$, which satisfies \eq{vf6eq4}--\eq{vf6eq5} on~$\dot U_{JL}$.

So, we are given a prescribed value of $\hat E_{JL}^-$ on an open set $\dot U_{JL}\subseteq V_J$ satisfying \eq{vf6eq5}, and we have to extend it to a larger open set $\hat U_{JL}\subseteq V_J$ containing both $\dot U_{JL}$ and $\psi_J^{-1}(C_J\cap C_K\cap C_L)$. This may not be possible: if we have chosen previous $\hat E_{JK}^-$'s badly near the `edge' of $\dot U_{JL}$ in $V_J$, then the prescribed values of $\hat E_{JL}^-$ may not extend continuously to the closure $\ov{\dot U_{JL}}$ of $\dot U_{JL}$ in $V_J$, and in particular, may not extend continuously over points in $[\psi_J^{-1}(C_J\cap C_K\cap C_L)]\cap [\ov{\dot U_{JL}}\sm\dot U_{JL}]$. However, we can deal with this problem by shrinking all the $\hat U_{JKL}$'s, such that the closure $\ov{\dot U_{JL}}$ of the new $\dot U_{JL}$ lies inside the old $\dot U_{JL}$. Then it is guaranteed that the prescribed value of $\hat E_{JL}^-$ on $\dot U_{JL}$ extends smoothly to an open neighbourhood of $\ov{\dot U_{JL}}$ in $V_J$, so we can choose $(\hat U_{JL},\hat E_{JL}^-)$ satisfying all the required conditions. As this holds for all $L\subsetneq J$ with $\md{L}=l$, this completes the inductive step, and $(\t)_{J,l}$ holds for all~$l=m-1,m-2,\ldots,1$.

Fix data $\hat U_{JL},\hat E_{JL}^-,\hat U_{JKL}$ as in $(\t)_{J,1}$. For all $\es\ne K\subsetneq J$, choose open neighbourhoods $\check U_{JK}$ of $\psi_J^{-1}(C_J\cap C_K)$ in $\hat U_{JK}$ such that if $K_1\subsetneq K_2$ and $K_2\subsetneq K_1$ then $\check U_{JK_1}\cap\check U_{JK_2}=\es$, and if $\es\ne L\subsetneq K\subsetneq J$ then $\check U_{JK}\cap\check U_{JL}\subseteq\hat U_{JKL}$. This is possible provided the $\check U_{JK}$ are small enough, using Proposition \ref{vf3prop1}(iii) to ensure~$\check U_{JK_1}\cap\check U_{JK_2}=\es$.

Define $\check U_J=\bigcup_{K:\es\ne K\subsetneq J}\check U_{JK}$. It is an open neighbourhood of the closed set $\check C_J$ in $V_J$, where $\check C_J=\bigcup_{K:\es\ne K\subsetneq J}\psi_J^{-1}(C_J\ab\cap C_K)$ in $V_J$. Define a vector subbundle $\check E_J^-$ of $E_J\vert_{\check U_J}$ by $\check E_J^-\vert_{\check U_{JK}}=\hat E_{JL}^-\vert_{\check U_{JK}}$ for all $\es\ne K\subsetneq J$. These prescribed values for different $K_1,K_2$ are compatible on the overlap $\check U_{JK_1}\cap\check U_{JK_2}$ by construction, so $\check E_J^-$ is well-defined.

Now apply Theorem \ref{vf3thm2}(b) to $A_J^\bu,V_J,E_J,s_J,\ldots,$ with closed set $\check C_J\subseteq V_J$ and pair $(\check U_J,\check E_J^-)$ satisfying $(*)$ with $\check C_J\subseteq\check U_J$. This shows that there exists $(\ti U_J,\ti E_J^-)$ satisfying $(*)$ for $A_J^\bu,V_J,E_J,s_J,\ldots,$ and an open neighbourhood $\check U_J'$ of $\check C_J$ in $\check U_J\cap\ti U_J$ such that $\check E_J^-\vert_{\check U'_J}=\ti E_J^-\vert_{\check U_J'}$. For all $\es\ne K\subsetneq J$, set $\ti U_{JK}=\check U_J'\cap \check U_{JK}$. Then $\ti U_{JK}$ is an open neighbourhood of $\psi_J^{-1}(C_J\ab\cap C_K)$ in $V_J$, and $\ti E_J^-\vert_{\ti U_{JK}}=\check E_J^-\vert_{\ti U_{JK}}=\hat E_{JK}^-\vert_{\ti U_{JK}}$, which is compatible with $(\ti U_K,\ti E_K^-)$ by definition. This completes the proof of the inductive step of $(+)_m$. So by induction, $(+)_m$ holds for all~$m=1,2,\ldots.$

Fix data $(\ti U_J,\ti E_J^-)$ for all $J\in A$ and $\ti U_{JK}$ for all  $J,K\in A$ with $K\subsetneq J$ as in $(+)_m$ as $m\ra\iy$ (or $m=\md{I}$ if $I$ is finite). For all $J\in A$, choose open neighbourhoods $U_J$ of $\psi_J^{-1}(C_J)$ in $\ti U_J$, such that setting $E_J^-=\ti E_J^-\vert_{U_J}$ and $S_J=\psi_J(s_J^{-1}(0)\cap U_J)$, so that $S_J$ is an open neighbourhood of $C_J$ in $X_\an$, then $(U_J,E_J^-)$ satisfies condition $(\dag)$, and for all $J,K\in A$, if $J\not\subseteq K$ and $K\not\subseteq J$ then $S_J\cap S_K=\es$, and if $K\subsetneq J$ then $\psi_J^{-1}(S_J\cap S_K)\subseteq\ti U_{JK}$. If $K\subsetneq J$, we define $U_{JK}=\ti U_{JK}\cap U_J\cap\phi_{JK}^{-1}(U_K)$. Then $s_J^{-1}(0)\cap U_{JK}=\psi_J^{-1}(S_J\cap S_K)$, and $(U_{JK},E_J^-\vert_{U_{JK}})$ is compatible with~$(U_K,E_K^-)$.

To see that we can choose $U_J$ for all $J\in A$ satisfying all these conditions, note that by Theorem \ref{vf3thm2}(c), if $U_J$ is small enough then $(U_J,E_J^-)$ satisfies $(\dag)$, as $(\ti U_J,\ti E_J^-)$ satisfies $(*)$. If $J\not\subseteq K$ and $K\not\subseteq J$ then Proposition \ref{vf3prop1}(iii) implies that $S_J\cap S_K=\es$ provided both $U_J,U_K$ are sufficiently small. Similarly, if $K\subsetneq J$ then $\psi_J^{-1}(S_J\cap S_K)\subseteq\ti U_{JK}$ holds provided both $U_J,U_K$ are sufficiently small. Now if $I$ is infinite, it is possible that an individual set $U_J$ may have to satisfy infinitely many smallness conditions, for compatibility with  infinitely many sets $\es \ne K\subseteq I$. However, the local finiteness condition Proposition \ref{vf3prop1}(ii) means that in an open neighbourhood of any $v_J\in\psi_J^{-1}(C_J)$, only finitely many smallness conditions on $U_J$ are relevant, so we can solve them. This completes the proof of Proposition~\ref{vf3prop2}.

\subsection{Proof of Proposition \ref{vf3prop3}}
\label{vf63}

Let $(\bX,\om_\bX^*),X_\an,\cK$ and $\bX_\dm$ be as in Theorems \ref{vf3thm3} and \ref{vf3thm4}, and use the notation of \S\ref{vf35}. First we relate orientations on $(\bX,\om_\bX^*)$ and $\bX_\dm$ at one point $x\in X_\an$. Pick $J\in A$ with $x\in S_J=\Im\psi_J^+$. From \eq{vf2eq7} and \eq{vf2eq9} we have
\ea
\!\!\bigl\{\text{orientations on $(\bX,\om_\bX^*)$ at $x$}\bigr\}\!&\cong\!\bigl\{\text{$\C$-orientations on $(H^1(\bT_\bX\vert_x),Q_x)$}\bigr\},\!\!
\label{vf6eq6}
\\
\bigl\{\text{orientations on $\bX_\dm$ at $x$}\bigr\}\!&\cong\!\bigl\{\text{orientations on $T_x^*\bX_\dm\!\op\! O_x\bX_\dm$}\bigr\},
\label{vf6eq7}
\ea
where $Q_x=\om_\bX^0\cdot$ is the nondegenerate complex quadratic form on $H^1(\bT_\bX\vert_x)$ in \eq{vf2eq6}. There is a unique $v_J$ in $s_J^{-1}(0)\cap U_J=(s_J^+)^{-1}(0)\subseteq U_J\subseteq V_J$ with $\psi_J(v_J)=x$. Equation \eq{vf3eq9} gives an isomorphism of complex vector spaces
\e
H^1\bigl(\bT_{\bs\al_J}\vert_{v_J}\bigr):\frac{\ts \Ker \bigl(t_J\vert_{v_J}:E_J\vert_{v_J}\ra F_J\vert_{v_J}\bigr)}{\ts \Im \bigl(\d s_J\vert_{v_J}:T_{v_J}V_J\ra E_J\vert_{v_J}\bigr)}\longra H^1\bigl(\bT_\bX\vert_x\bigr).
\label{vf6eq8}
\e
Write $\ti Q_{v_J}$ for the complex quadratic form on $\Ker (t_J\vert_{v_J})/\Im(\d s_J\vert_{v_J})$ identified with $Q_x$ by \eq{vf6eq8}, as in Definition \ref{vf3def3}. Then by \eq{vf6eq6} we have
\e
\begin{split}
&\bigl\{\text{orientations on $(\bX,\om_\bX^*)$ at $x$}\bigr\}\cong \\
&\big\{\text{$\C$-orientations on $\bigl(\Ker (t_J\vert_{v_J})/\Im(\d s_J\vert_{v_J}),\ti Q_{v_J}\bigr)$}\bigr\}.
\end{split}
\label{vf6eq9}
\e

Condition $(*)$ for $(U_J,E_J^-)$ at $v_J$ requires that
\begin{equation*}
\Pi_{v_J}:E_J^-\vert_{v_J}\cap \Ker \bigl(t_J\vert_{v_J}:E_J\vert_{v_J}\ra F_J\vert_{v_J}\bigr)\longra
\frac{\ts \Ker \bigl(t_J\vert_{v_J}:E_J\vert_{v_J}\ra F_J\vert_{v_J}\bigr)}{\ts \Im \bigl(\d s_J\vert_{v_J}:T_{v_J}V_J\ra E_J\vert_{v_J}\bigr)}
\end{equation*}
should be injective, with image $\Im\Pi_{v_J}$ a real vector subspace of half the real dimension of $\Ker (t_J\vert_{v_J})/\Im(\d s_J\vert_{v_J})$, on which the real quadratic form $\Re\ti Q_{v_J}$ is negative definite. As $(U_J,E_J^+,s_J^+,\psi_J\vert_{s_J^{-1}(0)\cap U_J})$ is a Kuranishi neighbourhood on $\bX_\dm$ by the proof of Theorem \ref{vf3thm4}, equation \eq{vf2eq10} gives an exact sequence
\begin{equation*}
\smash{\xymatrix@C=20pt{ 0 \ar[r] & T_x\bX_\dm \ar[r] & T_{v_J}V_J \ar[rr]^{\d s_J^+\vert_{v_J}} && E_J^+\vert_{v_J} \ar[r] & O_x\bX_\dm \ar[r] & 0. }}
\end{equation*}
Condition $(*)$ implies that $\Ker(\d s_J\vert_{v_J})=\Ker(\d s_J^+\vert_{v_J})$, so we have
\e
\smash{T_x\bX_\dm\cong \Ker\bigl(\d s_J\vert_{v_J}:T_{v_J}V_J\ra E_J\vert_{v_J}\bigr).}
\label{vf6eq10}
\e
Also from $(*)$ we see there is a canonical isomorphism
\e
\smash{O_x\bX_\dm\cong\frac{\Ker (t_J\vert_{v_J})/\Im(\d s_J\vert_{v_J})}{\Im\Pi_{v_J}}\,.}
\label{vf6eq11}
\e
By \eq{vf6eq10}, $T_x\bX_\dm$ is a complex vector space, so $T_x\bX_\dm$ and $T_x^*\bX_\dm$ have natural orientations as real vector spaces. Thus by \eq{vf6eq11} we have a bijection
\e
\begin{split}
&\bigl\{\text{orientations on $T_x^*\bX_\dm\op O_x\bX_\dm$}\bigr\}\cong\\
&\bigl\{\text{orientations on $\bigl[\Ker(t_J\vert_{v_J})/\Im(\d s_J\vert_{v_J})\bigr]/\Im\Pi_{v_J}$}\bigr\}.
\end{split}
\label{vf6eq12}
\e

Suppose we are given a complex basis $e_1,\ldots,e_k$ of $\Ker(t_J\vert_{v_J})/\Im(\d s_J\vert_{v_J})\cong \C^k$ which is orthonormal w.r.t.\ $\ti Q_{v_J}$. As $e_1,\ldots,e_k$ are orthonormal w.r.t.\ $\ti Q_{v_J}$, the real quadratic form $\Re\ti Q_{v_J}$ is positive definite on the real span $\langle e_1,\ldots,e_k\rangle_\R$, and $\Re\ti Q_{v_J}$ is negative definite on $\Im\Pi_{v_J}$,
so $\langle e_1,\ldots,e_k\rangle_\R\cap\Im\Pi_{v_J}=\{0\}$. Therefore $e_1+\Im\Pi_{v_J},\ldots,e_k+\Im\Pi_{v_J}$ are linearly independent in the real vector space  $\bigl[\Ker(t_J\vert_{v_J})/\Im(\d s_J\vert_{v_J})\bigr]/\Im\Pi_{v_J}\cong\R^k$, so they are a basis as $\Im\Pi_{v_J}$ has half the real dimension of $\Ker(t_J\vert_{v_J})/\Im(\d s_J\vert_{v_J})$. Define an identification
\e
\begin{split}
&\big\{\text{$\C$-orientations on $\bigl(\Ker (t_J\vert_{v_J})/\Im(\d s_J\vert_{v_J}),\ti Q_{v_J}\bigr)$}\bigr\}\\
&\qquad\cong\bigl\{\text{orientations on $\bigl[\Ker(t_J\vert_{v_J})/\Im(\d s_J\vert_{v_J})\bigr]/\Im\Pi_{v_J}$}\bigr\},
\end{split}
\label{vf6eq13}
\e
such that orientations on both sides are identified if, whenever $e_1,\ldots,e_k$ is an oriented orthonormal complex basis for $\bigl(\Ker (t_J\vert_{v_J})/\Im(\d s_J\vert_{v_J}),\ti Q_{v_J}\bigr)$, then $e_1+\Im\Pi_{v_J},\ldots,e_k+\Im\Pi_{v_J}$ is an oriented basis for $[\Ker(t_J\vert_{v_J})/\Im(\d s_J\vert_{v_J})]/\Im\Pi_{v_J}$. Combining equations \eq{vf6eq7}, \eq{vf6eq9},  \eq{vf6eq12} and \eq{vf6eq13} gives an identification
\e
\smash{\bigl\{\text{orientations on $(\bX,\om_\bX^*)$ at $x$}\bigr\}\cong
\bigl\{\text{orientations on $\bX_\dm$ at $x$}\bigr\}.}
\label{vf6eq14}
\e

It is not difficult to show that the isomorphism \eq{vf6eq14} is independent of the choice of $J\in A$ with $x\in S_J$, and depends continuously on $x\in X_\an$. Thus we get a canonical 1-1 correspondence between the sets in Proposition \ref{vf3prop1}(a),(c). The last part of Theorem \ref{vf2thm4} gives a 1-1 correspondence between the sets in Proposition \ref{vf3prop1}(b),(c). This completes the proof.

\subsection{Proof of Proposition \ref{vf3prop4}}
\label{vf64}

Suppose $(\bX,\om_\bX^*)$ is a separated, $-2$-shifted symplectic derived $\C$-scheme with $\vdim_\C\bX=n$, whose complex analytic topological space $X_\an$ is second countable. Let $\cK,\cK'$ be different possible Kuranishi atlases constructed in Theorem \ref{vf3thm3}, and $\bX_\dm,\bX_\dm'$ the corresponding derived manifolds in Theorem \ref{vf3thm4}.

As in \S\ref{vf35}, let $\cK$ be constructed using the family $\bigl\{(A_i^\bu,\bs\al_i):i\in I\bigr\}$, and data $A_J^\bu,\bs\al_J$ for $J\in A$, $\Phi_{JK}$ for $K\subseteq J$ in $A$ from Theorem \ref{vf3thm1}, where $A=\{J:\es\ne J\subseteq I$, $J$ finite$\}$, as in \S\ref{vf32} use notation $V_J,E_J,F_J,s_J,t_J,\psi_J$ and $R_J=\bigcap_{i\in J}R_i\subseteq X_\an$ from $A_J^\bu,\bs\al_J$ and $\phi_{JK},\chi_{JK},\xi_{JK}$ from $\Phi_{JK}$. Let $\cK$ be defined using closed subsets $C_J\subseteq X_\an$ for $J\in A$ in Proposition \ref{vf3prop1} and pairs $(U_J,E_J^-)$ and open subsets $U_{JK}\subseteq U_J$ in Proposition~\ref{vf3prop2}. Similarly, let $\cK'$ be constructed using $\bigl\{(A_{i'}^{\prime\bu},\bs\al'_{i'}):i'\in I'\bigr\},A_{J'}^{\prime\bu},\bs\al'_{J'},V'_{J'},E'_{J'},\ldots,U'_{J'K'}\subseteq U'_{J'}$.

We must build a derived manifold with boundary $\bW_\dm$ with topological space $X_\an\t[0,1]$ and $\vdim\bW_\dm=n+1$, and an equivalence $\pd\bW_\dm\simeq \bX_\dm\amalg\bX_\dm'$ topologically identifying $\bX_\dm$ with $X_\an\t\{0\}$ and $\bX_\dm'$ with $X_\an\t\{1\}$. 

Write $\bs{\ti\pi}:\bs{\ti X}\ra Z$ to be the projection $\bs\pi_{\bA{}^1}:\bX\t\bA{}^1\ra\bA{}^1$, so that $Z=\bA{}^1=\Spec B$ with $B=\C[z]$, and $Z_\an=\C$. Define $\om_{\smash{\bs{\ti X}/Z}}=\bs\pi_\bX^*(\om_\bX^0)$. Then $\om_{\bs{\ti X}/Z}$ is a family of $-2$-shifted symplectic structures on $\bX/Z$ in the sense of \S\ref{vf37}, the constant family over $Z=\bA{}^1$ with fibre $(\bX,\om_\bX^*)$. We now carry out the programme of \S\ref{vf37} for $\bs{\ti\pi}:\bs{\ti X}\ra Z,\om_{\smash{\bs{\ti X}/Z}}$, choosing data as follows:
\begin{itemize}
\setlength{\itemsep}{0pt}
\setlength{\parsep}{0pt}
\item[(a)] Set $\ti I=I\amalg I'$, the disjoint union of $I$ and $I'$.
\item[(b)] Define $(\ti A{}_i^\bu,\bs{\ti\al}_i,\ti\be_i)$ for $i\in I$ by $\ti A{}_i^\bu=A{}_i^\bu\ot_\C\C[z,(z-1)^{-1}]$, so that $\bSpec\ti A{}_i^\bu=(\bSpec A{}_i^\bu)\t(\bA{}^1\sm\{1\})$, and $\bs{\ti\al}_i=\bs\al_i\t{\rm inc}:(\bSpec A{}_i^\bu)\t(\bA{}^1\sm\{1\})\ra\bX\t\bA{}^1$, and $\ti\be_i:\C[z]\ra A{}_i^0\ot_\C\C[z,(z-1)^{-1}]$, $\ti\be_i:z\mapsto 1\ot z$.
Similarly, define $(\ti A{}_{i'}^{\prime\bu},\bs{\ti\al}_{i'},\ti\be_{i'})$ for $i'\in I'$ by $\ti A{}_{i'}^\bu=A{}_{i'}^{\prime\bu}\ot_\C\C[z,z^{-1}]$, so  $\bSpec\ti A{}_{i'}^{\prime\bu}=(\bSpec A{}_{i'}^{\prime\bu})\t(\bA{}^1\sm\{0\})$, and $\bs{\ti\al}_{i'}=\bs\al'_{i'}\t{\rm inc}:(\bSpec A{}_{i'}^{\prime\bu})\t(\bA{}^1\sm\{0\})\ra\bX\t\bA{}^1$, and $\ti\be'_{i'}:\C[z]\ra A{}_{i'}^{\prime 0}\ot_\C\C[z,z^{-1}]$, $\ti\be'_{i'}:z\mapsto 1\ot z$.
\item[(c)] Write $\ti A=\{\ti J:\es\ne\ti J\subseteq\ti I$, $\ti J$ finite$\}$. Then $A\subseteq\ti A$ and $A'\subseteq\ti A$.
\item[(d)] When we apply Theorem \ref{vf3thm1} to choose $\ti A_{\ti J}^\bu,\bs{\ti\al}_{\ti J},\ti\be_{\ti J}$ for $\ti J\in\ti A$ and $\ti\Phi_{\ti J\ti K}$ for $\ti K\subseteq\ti J$, we make these choices so that 
\begin{gather*}
\ti A{}_J^\bu=A{}_J^\bu\ot_\C\C[z,(z-1)^{-1}], \qquad\qquad \ti A{}_{J'}^\bu=A{}_{J'}^{\prime\bu}\ot_\C\C[z,z^{-1}],\\
\bs{\ti\al}_J=\bs\al_J\t{\rm inc}:(\bSpec A{}_J^\bu)\t(\bA{}^1\sm\{1\})\longra\bX\t\bA{}^1,\;\> \ti\be_J:z\mapsto 1\ot z,\\
\bs{\ti\al}_{J'}=\bs\al'_{J'}\t{\rm inc}:(\bSpec A{}_{J'}^{\prime\bu})\t(\bA{}^1\sm\{0\})\longra\bX\t\bA{}^1,\;\> \ti\be_{J'}:z\mapsto 1\ot z,\\
\ti\Phi_{JK}=\Phi_{JK}\ot\id:A_K^\bu\ot_\C\C[z,(z-1)^{-1}]\longra A_J^\bu\ot_\C\C[z,(z-1)^{-1}],\\
\ti\Phi_{J'K'}=\Phi'_{J'K'}\ot\id:A{}_{K'}^{\prime\bu}\ot_\C\C[z,z^{-1}]\longra A{}_{J'}^{\prime\bu}\ot_\C\C[z,z^{-1}],
\end{gather*}
for all $K\subseteq J$ in $A$ and $K'\subseteq J'$ in $A'$. This is clearly possible. Note that this does not determine $\ti A_{\ti J}^\bu,\bs{\ti\al}_{\ti J},\ti\be_{\ti J}$ or $\Phi_{\ti J\ti K}$ if $\ti J\in \ti A\sm(A\amalg A')$.
\item[(e)] When we translate to complex geometry using \S\ref{vf32}, part (d) implies that $\ti V_J=V_J\t(\C\sm\{1\})$ for $J\in A\subseteq\ti A$. Also $\ti E_J,\ti F_J,\ti s_J,\ti t_J,\ti\phi_{JK},\ti\chi_{JK}$ for $J,K\in A$ is obtained from $E_J,\ldots,\chi_{JK}$ by taking products with $\C\sm\{1\}$.
Similarly, $\ti V_{J'},\ti E_{J'},\ti F_{J'},\ti s_{J'},\ti t_{J'},\ti\phi_{J'K'},\ti\chi_{J'K'}$ for $J',K'\in A'\subseteq\ti A$ are obtained from $V_{J'},\ldots,\chi_{J'K'}$ by taking products with $\C\sm\{0\}$.
\item[(f)] When we choose data $\ti C_{\ti J},(\ti U_{\ti J},\ti E_{\ti J}^-)$ for $\ti J\in\ti A$, we do this so that 
\begin{align*}
\ti C_J\cap(X_{\an}\t\{0\})&=C_J\t\{0\}, & \ti U_J\cap V_J\t\{0\}&=U_J\t\{0\}, \\ 
\ti E_J^-\vert_{U_J\t\{0\}}&=E_J^-\t 0, & \ti C_{J'}\cap(X_{\an}\t\{1\})&=C'_{J'}\t\{1\}, \\ 
\ti U_{J'}\cap V'_{J'}\t\{1\}&=U'_{J'}\t\{1\}, & 
\ti E_{J'}^-\vert_{U'_{J'}\t\{1\}}&=E_{J'}^{\prime -}\t 1,
\end{align*}
whenever $J\in A$ and $J'\in A'$. This is clearly possible.
\end{itemize}

Theorem \ref{vf3thm7} constructs a relative Kuranishi atlas $\ti\cK$ for $\pi_\C:X_\an\t\C\ra\C$, of dimension $n+2$. By construction, over $X_\an\t\{0\}$ this restricts to the Kuranishi atlas $\cK$, and over $X_\an\t\{1\}$ it restricts to~$\cK'$.

Theorem \ref{vf3thm8} gives a derived manifold $\bs{\ti X}_\dm$ with $\vdim\bs{\ti X}_\dm=n+2$ and topological space $X_\an\t\C$, with a morphism $\bs{\ti\pi}_\dm:\bs{\ti X}_\dm\ra\C$. From Theorem \ref{vf3thm8}(iii) we see that $\bs{\ti X}{}_\dm^0=\bs{\ti\pi}_\dm^{-1}(0)\simeq\bX_\dm$ and~$\bs{\ti X}{}_\dm^1=\bs{\ti\pi}_\dm^{-1}(1)\simeq\bX_\dm'$.

Now define $\bW_\dm=\bs{\ti X}_\dm\t_{\bs{\ti\pi}_\dm,\C,{\rm inc}}[0,1]$, as a fibre product in the 2-category $\dManc$ of d-manifolds with corners from \cite{Joyc2,Joyc3,Joyc4}, where ${\rm inc}:[0,1]\hookra\C$ is the inclusion. By properties of fibre products in $\dManc$ from \cite{Joyc2,Joyc3,Joyc4}, this has topological space $X_{\an}\t[0,1]$ and $\vdim\bW_\dm=n+1$, and boundary
\e
\smash{\pd\bW_\dm\!\simeq\! \bs{\ti X}_\dm\t_{\bs{\ti\pi}_\dm,\C,{\rm inc}}\pd[0,1]\!\simeq\!\bs{\ti X}_\dm\t_{\bs{\ti\pi}_\dm,\C,{\rm inc}}\{0,1\}\!\simeq\!\bX_\dm\amalg\bX_\dm'.}
\label{vf6eq15}
\e
This proves the first part of Proposition~\ref{vf3prop4}. 

For the last part, orientations on $(\bX,\om_\bX^*)$ correspond naturally to orientations for $\bs{\ti\pi}:\bs{\ti X}\ra Z,\om_{\smash{\bs{\ti X}/Z}}$, by pullback along $\bs{\ti X}\ra\bX$, and these correspond to orientations on $\bs{\ti X}_\dm$ by Proposition \ref{vf3prop5}, and thus (using oriented fibre products) to orientations on $\bW_\dm$. Since $\pd[0,1]=-\{0\}\amalg\{1\}$ in oriented manifolds, we see that as in \eq{vf6eq15} that $\pd\bW_\dm\simeq -\bX_\dm\amalg\bX_\dm'$ in oriented derived manifolds. This completes the proof.

\medskip

\noindent{\small\sc 

\noindent Dennis Borisov, Mathematisches Institut, Georg-August Universit\"at G\"ot\-tin\-gen, Bunsenstrasse 3-5, D-37073 G\"ottingen, Germany.

\noindent E-mail: {\tt dennis.borisov@gmail.com}.
\medskip

\noindent Dominic Joyce, The Mathematical Institute, Radcliffe Observatory Quarter, Woodstock Road, Oxford, OX2 6GG, U.K.

\noindent E-mail: {\tt joyce@maths.ox.ac.uk}.}

\end{document}